\documentclass{article}
\usepackage{lineno,hyperref}
\modulolinenumbers[5]

\usepackage{authblk}
\usepackage{graphicx}
\usepackage{amsmath, amsfonts, amsthm, amssymb, bm}
\usepackage{color}
\usepackage{tabularx,hhline}
\usepackage{multirow}
\usepackage{booktabs}
\usepackage{fullpage}[2cm]
\usepackage{algorithm}
\usepackage{algorithmic}
\usepackage{enumerate}
\usepackage[margin = 1.5cm, font = small, labelfont = bf, labelsep = period]{caption}
\usepackage[font = small, labelfont = bf, labelsep = period]{caption}

\newcommand{\subjectto}{\textnormal{subject to}}
\newcommand{\st}{\textnormal{s.t.}}
\newcommand{\maximize}{\textnormal{maximize}}
\newcommand{\minimize}{\textnormal{minimize}}

\DeclareMathOperator*{\argmin}{arg\,min}
\DeclareMathOperator*{\diag}{diag}

\newcommand{\eg}{\textit{e.g.}}
\newcommand{\ie}{\textit{i.e.}}
\newcommand{\PP}{\mathbb{P}}
\newcommand{\QQ}{\mathbb{Q}}
\newcommand{\EE}{\mathbb{E}}
\newcommand{\RR}{\mathbb{R}}

\DeclareMathAlphabet{\mathscr}{U}{dutchcal}{m}{n}
\SetMathAlphabet{\mathscr}{bold}{U}{dutchcal}{b}{n}
\DeclareMathAlphabet{\mathbscr} {U}{dutchcal}{b}{n}

\newcommand{\Tb}{{\bm{\mathcal T}}}
\newcommand{\hb}{{\bm{\mathscr h}}}
\newcommand{\Qb}{{\bm{\mathcal Q}}}
\newcommand{\Wb}{{\bm{\mathcal W}}}
\newcommand{\rb}{{\bm{\mathscr q}}}
\newcommand{\rbl}{{\mathscr q}}
\newcommand{\thetab}{\bm{\pi}}

\newcommand{\Zover}{\overline{\mathcal Z}}
\newcommand{\Zunder}{\underline{\mathcal Z}}
\newcommand{\Phat}{\hat{\mathcal P}}
\newcommand{\PPhat}{\hat{\mathbb P}}
\newcommand{\xih}{\hat{\bm \xi}}
\newcommand{\xit}{\tilde{\bm \xi}}

\newcommand{\tr}{\textup{tr}}

\newtheorem{thm}{Theorem}
\newtheorem{prop}{Proposition}
\newtheorem{lem}{Lemma}

\newtheorem{coro}{Corollary}
\newtheorem{ex}{Example}
\newtheorem{rem}{Remark}
\newtheorem{defi}{Definition}

\begin{document}
%%%%%%%%%%%%%%%%

\title{Conic Programming Reformulations of Two-Stage Distributionally Robust Linear Programs over Wasserstein Balls}

\author[1]{Grani A. Hanasusanto}
\author[2]{Daniel Kuhn}

\affil[1]{\small Graduate Program in Operations Research and Industrial Engineering, The University of Texas at Austin, USA}
\affil[2]{\small Risk Analytics and Optimization Chair, {\'E}cole Polytechnique F{\'e}d{\'e}rale de Lausanne, Switzerland}

\maketitle
%\textbf{Contact:} \href{grani.hanasusanto@utexas.edu}{grani.hanasusanto@utexas.edu} (GAH);  \href{daniel.kuhn@epfl.ch}{daniel.kuhn@epfl.ch} (DK)
%\affil[3]{\small Imperial College Business School, Imperial College London, United Kingdom}
%
%\author[rvt]{C.V.~Radhakrishnan\corref{cor1}\fnref{fn1}}
%\ead{cvr@river-valley.com}
%\author[rvt,focal]{K.~Bazargan\fnref{fn2}}
%\ead{kaveh@river-valley.com}
%\author[els]{S.~Pepping\corref{cor2}\fnref{fn1,fn3}}
%\ead[url]{http://www.elsevier.com}
%%%%%%%%%%%%%%%%%%%%%%%%%%%%%%%%%%%%%%%%%%%%%%%%%%%%%%%%%%%%%%%%%%%%%

\begin{abstract}
	Adaptive robust optimization problems are usually solved {\em approximately} by restricting the adaptive decisions to simple parametric decision rules. However, the corresponding approximation error can be substantial. In this paper we show that two-stage robust and distributionally robust linear programs can often be reformulated {\em exactly} as conic programs that scale polynomially with the problem dimensions. Specifically, when the ambiguity set constitutes a 2-Wasserstein ball centered at a discrete distribution, then the distributionally robust linear program is equivalent to a copositive program (if the problem has complete recourse) or can be approximated arbitrarily closely by a sequence of copositive programs (if the problem has sufficiently expensive recourse). These results directly extend to the classical robust setting and motivate strong tractable approximations of two-stage problems based on semidefinite approximations of the copositive cone. We also demonstrate that the two-stage distributionally robust optimization problem is equivalent to a tractable linear program when the ambiguity set constitutes a 1-Wasserstein ball centered at a discrete distribution and there are no support constraints.
\end{abstract}


\section{Introduction}\label{sec:introduction}

In two-stage optimization under uncertainty an agent selects a here-and-now decision before observing the realization of some decision-relevant random vector. Once the uncertainty has been revealed, a wait-and-see decision is taken in order to correct any undesired effects of the here-and-now decision in the realized scenario. Classical stochastic programming seeks a single here-and-now decision and a family of (possibly infinitely many) wait-and-see decisions---one for each possible uncertainty realization---with the goal to minimize the sum of a deterministic here-and-now cost and the expectation of an uncertain wait-and-see cost~\cite{shapiro2014lectures}. Classical robust optimization, in contrast, seeks decisions that minimize the worst case of the total cost across all possible uncertainty realizations~\cite{ben2009robust}. While stochastic programming assumes full knowledge of the distribution governing the uncertain problem parameters, which is needed to evaluate the expectation of the total costs, robust optimization denies (or ignores) any knowledge of this distribution except for its support.

Distributionally robust optimization is an alternative modeling paradigm pioneered in \cite{dupacova:66, Scarf:58, SK02:minimax_analysis}. It has gained new thrust over the last decade and challenges the black-and-white view of stochastic and robust optimization. Specifically, it assumes that the decision maker has access to some limited probabilistic information ({\em e.g.}, in the form of the distribution's moments, its structural properties or its distance to a reference distribution); but not enough to pin down the true distribution precisely. In this setting, a meaningful objective is to minimize the worst-case expected total cost, where the worst case is evaluated across an ambiguity set that contains all distributions consistent with the available probabilistic information. Distributionally robust models enjoy strong theoretical justification from decision theory \cite{GS89:maxmin_exp_utility}, and there is growing evidence that they provide high-quality decisions at a moderate computational cost \cite{DY10:DRO, goh2010distributionally, WKS13:drco}.

Two-stage decision problems under uncertainty---whether stochastic, robust or distributionally robust---typically involve a continuum of wait-and-see decisions and thus constitute infinite-dimensional functional optimization problems. Therefore, they can only be solved approximately, except in contrived circumstances. The existing approximation methods can roughly be subdivided into {\em discretization schemes} \cite{hadjiyiannis2011scenario, kleywegt2002sample, shapiro2003monte} and {\em decision rule methods} \cite{BGGN04:LDR, georghiou2015generalized, goh2010distributionally}. Discretization schemes approximate the support of the uncertain parameters with a finite subset, which entails a relaxation of the original problem and encourages optimistically biased solutions. Decision rule methods, on the other hand, approximate the infinite-dimensional space of all wait-and-see decisions with a finite-dimensional subspace of linearly parameterized decision rules, which entails a restriction of the original problem and leads to pessimistically biased solutions. In this paper, we introduce a new method for approximating two-stage distributionally robust linear programs, which can neither be classified as a discretization scheme nor as a decision rule method: We first reformulate the original infinite-dimensional optimization problem as an equivalent finite-dimensional conic program of polynomial size, which absorbs all the complexity in its cones, and then replace the cones with tractable inner approximations. % with tractable cones to obtain a tractable conservative approximation for the original problem.

Our exposition focuses on distributionally robust linear programs whose ambiguity sets contain {\color{black} all discrete and continuous distributions} supported on a polytope that have a Wasserstein distance of at most~$\epsilon$ from a discrete reference distribution {\color{black} (such as the empirical distribution corresponding to finitely many samples from the unknown true distribution)}. This problem class encapsulates the two-stage stochastic linear programs with discrete distributions (for~$\epsilon=0$) and the two-stage robust optimization problems with {\color{black} bounded} polyhedral uncertainty sets (for~$\epsilon=\infty$) as special cases. Wasserstein ambiguity sets have first been used in the context of portfolio optimization~\cite{pflug2007ambiguity}. The corresponding distributionally robust optimization models were initially perceived as difficult and thus tackled with methods from global optimization~\cite{wozabal2012framework}; see also~\cite[Chapter~7]{pflug2014multistage}. Recently it has been discovered, however, that distributionally robust optimization problems with Wasserstein ambiguity sets can often be reformulated as finite convex programs~\cite{MEK15:Wasserstein, ZG15:Wasserstein}. Single-stage problems with piecewise linear cost functions, for instance, are tractable and admit convex reformulations of polynomial sizes~\cite{MEK15:Wasserstein}. Two-stage problems, on the other hand, are generically NP-hard. Their convex reformulations have exponential size but are amenable to Benders-type decomposition algorithms~\cite{ZG15:Wasserstein}. Alternatively, two-stage problems can be converted to single-stage problems via a decision rule approximation, in which case they admit again a convex reformulation of polynomial size and thus regain tractability~\cite{GK16:Wasserstein}.

This paper extends the state-of-the-art in two-stage distributionally robust linear programming along several dimensions. We highlight the following main contributions:
\begin{enumerate}
	\item[(i)] We prove that any two-stage distributionally robust linear program with {\em complete recourse} is equivalent to a copositive program of polynomial size if the ambiguity set constitutes a 2-Wasserstein ball centered at a discrete distribution. %, while distances between uncertainty realizations are measured by the 2-norm.
	\item[(ii)] We prove that any two-stage distributionally robust linear program with {\em sufficiently expensive recourse} can be approximated arbitrarily closely by a sequence of copositive programs of a fixed polynomial size if the ambiguity set constitutes a 2-Wasserstein ball centered at a discrete distribution. %, while distances between uncertainty realizations are measured by the 2-norm. 
	%\item Assume that the ambiguity set constitutes a 2-Wasserstein ball centered at a discrete distribution, while distances between uncertainty realizations are measured by the 2-norm. Then, we prove that the two-stage distributionally robust optimization problem is equivalent to a copositive program of polynomial size (if the problem has complete recourse) or can be approximated closely by a sequence of copositive programs of a fixed polynomial size (if the problem has sufficiently expensive recourse). 
	\item[(iii)] By using nested hierarchies of tractable convex cones to approximate the (intractable) copositive cones from the inside~\cite{bomze2002solving, DKP02:copositive, parrilo2000structured}, we obtain sequences of tractable conservative approximations for the two-stage distributionally robust linear programs described in (i) and (ii). These approximations can be made arbitrarily accurate. However, numerical tests suggest that even the coarsest of these approximations distinctly outperform the state-of-the-art decision rule approximations in terms of accuracy.
	\item[(iv)] We prove that any two-stage distributionally robust linear program with {\em fixed costs} is equivalent to a tractable linear program if the ambiguity set constitutes a 1-Wasserstein ball centered at a discrete distribution and if there are no support constraints. %, while distances between uncertainty realizations are measured by the 1-norm. 
	We also show that this tractability result is sharp. 
	%\item Assume that the ambiguity set constitutes a 1-Wasserstein ball centered at a discrete distribution, while distances between uncertainty realizations are measured by the 1-norm. Moreover, assume that the objective function is not affected by the uncertainty and that there are no support constraints. Then, we prove that the two-stage distributionally robust optimization problem is equivalent to a tractable linear program. We also show that tractability is lost if some of the assumptions are relaxed. 
	%\item We derive an exact copositive programming reformulation for the generic two-stage distributionally robust linear programming over Wasserstein balls. In contrast with the current solution schemes in the literature, our exactness result requires very minimal assumptions and is able to handle problems with induced constraints. 
	%\item We develop a conservative semidefinite programming reformulation for the two-stage problem and we highlight its effectiveness in out-of-sample experiments on the newsvendor problem. % that the resulting approximation outperforms the state-of-the-art moment-based distributionally robust model and the standard sample-average approximation scheme. 
	%We are not aware of any similar numerical results in the current literature on two-stage data-driven distributionally robust linear programming. 
	%\item We derive a new tractability result for the two-stage problem under a particular setting on the Wasserstein balls. We show that violation of the main tractability assumption immediately leads to NP-hard optimization problems. 
	\item[(v)] We demonstrate that all of the above results carry directly over to classical two-stage robust optimization problems with {\color{black} bounded} polyhedral uncertainty sets. To our best knowledge, we provide the first (polynomially-sized) conic programming reformulations for generic problem instances in this class. 
\end{enumerate}
%We emphasize that our approach applies to all problems with sufficiently expensive recourse even they fail to have relatively complete recourse, that is, even if some here-and-now decisions lead to wait-and-see problems that are infeasible for some uncertainty realizations.

%These results directly extend to the classical robust setting and motivate strong tractable approximations of two-stage problems based on semidefinite approximations of the copositive cone. 	
%	Two-stage stochastic programming model originated in the seminal work of Dantzig \cite{dantzig1955linear}. In the classical setting where the distribution $\PP$ is unambiguous, the problem is known to be generically $\#$P-hard \cite{HKW16:SPComplexity}. Over the past five decades, several tractable approximation schemes have been developed to alleviate this intractability. We refer the reader to \cite{shapiro2014lectures} for the state-of-the-art treatment on stochastic programming. If the distribution~$\PP$ is ambiguous then we arrive at the distributionally robust model \eqref{eq:DRSP}. For the particular setting where only the support information $\Xi$ is known the problem reduces to the standard two-stage robust optimization model~\eqref{eq:robust} which has been studied  extensively in the last two decades; see~\cite{ben2009robust} for a comprehensive overview of robust optimization. 

Two-stage distributionally robust linear programs with objective uncertainty are studied in~\cite{BDNT10:Models_Minimax}. Assuming that only the first- and second-order moments of the uncertain cost coefficients are known, these problems can be reformulated as tractable semidefinite programs. In the presence of constraint uncertainty, however, these problems become intractable. Two-stage distributionally robust binary programs with polyhedral moment information are studied in \cite{hanasusanto2016k}. If only the cost coefficients are uncertain, these problems can be reformulated as explicit mixed-integer linear programs of polynomial sizes. While two-stage distributionally robust optimization endeavors to minimize the worst-case (maximal) expected wait-and-see cost, a parallel stream of research investigates the best-case (minimal) expectations of the minima of mixed zero-one linear programs with objective uncertainty. Under first- and second-order moment information, any such best-case expectation can be reformulated as the optimal value of a completely positive program~\cite{NRZ11:mixed01}. In fact, this best-case expectation even reduces to the optimal value of a tractable semidefinite program whenever the convex hull of all rank-1 outer products of feasible wait-and-see decisions with themselves is semidefinite-representable~\cite{natarajan2016reduced}. These deep theoretical results have recently opened up new avenues for modeling and solving stochastic appointment scheduling problems~\cite{kong2013scheduling} and have also ramifications for computing best-worst choice probabilities in discrete choice models~\cite{natarajan2016reduced}.  A comprehensive survey of recent results at the interface of distributionally robust optimization and completely positive programming is provided in~\cite{li2014distributionally}.

In contrast to the existing literature, here we develop copositive programming reformulations for generic two-stage distributionally robust linear programs where both the objective function and the constraints may be affected by the uncertainty. We also present new linear programming reformulations for two-stage distributionally robust linear programs where the uncertainty affects only the constraints. {\color{black} These exact reformulations are reminiscent of the conservative approximation models for two-stage robust optimization models derived in~\cite{ardestani2016linearized} by leveraging popular reformulation-linearization techniques from bilinear programming. Another main difference to the existing literature is our focus on Wasserstein balls instead of moment ambiguity sets to capture distributional uncertainty.} This has the advantage that the degree of ambiguity aversion can be controlled by tuning the radius of the Wasserstein ball.

A key benefit of Wasserstein balls is that they provide natural confidence sets for the unknown distribution of the uncertain problem parameters. Specifically, the Wasserstein ball around the empirical distribution on $I$ independent historical samples contains the unknown true distribution with confidence $1-\beta$ if its radius exceeds an explicit threshold $\epsilon_I(\beta)$ that is known in closed form~\cite{MEK15:Wasserstein, ZG15:Wasserstein}. Thus, the corresponding distributionally robust optimization problem offers a $1-\beta$ upper confidence bound on the optimal value of the true stochastic program. One can also show that this data-driven distributionally robust optimization problem converges to the corresponding true stochastic program as the sample size $I$ tends to infinity~\cite{MEK15:Wasserstein, ZG15:Wasserstein}. Other data-driven distributionally robust optimization models that offer finite sample and asymptotic guarantees are discussed in~\cite{BGL2014:robust_SAA} based on goodness-of-fit ambiguity sets, in~\cite{jiang2015risk} based on $L^1$-norm ball ambiguity sets and in~\cite{love2015phi} based on $\Phi$-divergence ambiguity sets.

{\color{black} While this paper was under review, we became aware of the paper \cite{xu2016copositive} by Xu and Burer, which was submitted simultaneously. It turns out that our Corollary~\ref{thm:robust} is equivalent to Theorem~1 in \cite{xu2016copositive}, and so we mention it here for the reader�s reference. While \cite{xu2016copositive} focuses on two-stage {\em robust} linear programs with right hand side uncertainty, we develop copositive programming reformulations for {\em distributionally robust} two-stage linear programs with objective {\em and} constraint uncertainty.}

%Bertsimas et al.~\cite{BGL2014:robust_SAA} develop data-driven models based on hypothesis testing. Under their setting, if~$\bm T(\bm x)=\bm 0$ then the two-stage problem \eqref{eq:DRSP} can be reformulated equivalently as a tractable linear program. Jiang and Guan \cite{jiang2015risk} considered the  case where the ambiguity set is defined via an $L^1$-norm ball centered around the empirical distribution. They derive an exact unambiguous reformulation as a two-stage optimization problem whose objective function is described by a mixture of a conditional value-at-risk and a worst-case risk measures. As the latter measure leads to intractability, they further devise a concise MILP reformulation for the restricted case where  $\Xi$ is a knapsack polytope. Love and Bayraksan~\cite{love2015phi} investigate two-stage problems with discrete distributions where the ambiguity sets are defined through $\Phi$-divergences. They  devise a specialized Bender's decomposition algorithm to solve the emerging reformulations.  Using the Wasserstein ambiguity set,  Zhao and Guan \cite{ZG15:Wasserstein} develop a Bender's decomposition algorithm for the specific setting where $\bm Q=\bm 0$, $\bm T(\bm x)=\mathbb I$, and $\Xi$ is a hyperrectangle. 

%In, the model in \cite{NRZ11:mixed01} is employed to reformulate the worst-case total waiting time for an outpatient clinic as a completely positive program. The optimal schedule for the patients is then obtained by solving the minimization problem corresponding to the dual copositive program.  

The rest of the paper is structured as follows. Section~\ref{sect:problem-formulation} provides a formal problem statement and reviews some fundamental results from~\cite{MEK15:Wasserstein,ZG15:Wasserstein}. In Section~\ref{sec:CP} we derive copositive programming reformulations for two-stage distributionally robust linear programs over 2-Wasserstein balls and discuss tractable approximations. Exact tractable linear programming reformulations for two-stage distributionally robust linear programs over 1-Wasserstein balls are described in Section~\ref{sec:tractability}. Section~\ref{sec:numerical} reports on numerical results.

\paragraph{Notation:} For any $I\in\mathbb N$, we define $[I]$ as the index set $\{1,\dots,I\}$. We denote by $\mathbb I$ the identity matrix and by $\mathbf e$ the vector of all ones. Their dimensions will be clear from the context. The trace of a square matrix $\bm M$ is denoted as $\tr(\bm M)$. We define $\diag(\bm v)$ as the diagonal matrix with the vector $\bm v$ on its main diagonal.  The set of  non-negative (positive) reals is denoted as $\mathbb R_+$ ($\mathbb R_{++}$).
The set of all symmetric matrices in $\RR^{K\times K}$ is denoted as $\mathbb S^K$, while the cone of positive semidefinite matrices in $\RR^{K\times K}$ is denoted as $\mathbb S_+^K$. We define the cone of copositive matrices as $\mathcal C=\{\bm M\in\mathbb S^K:\bm{\xi}^\top\bm{M}\bm{\xi}\geq 0\;\forall\bm{\xi}\geq \bm 0\}$ and the cone of completely positive matrices as $\mathcal C^*=\{\bm{M}\in\mathbb S^K:\bm M=\bm B\bm B^\top \text{ for some } \bm B\in\RR_+^{K\times g(K)}\}$, where $g(K)= \max\{{K+1\choose 2}-4,K\}$ \cite{SBBJ15:CP-RANK}.
%\{\bm{M}\in\mathbb S^K~:~\bm M=\bm B\bm B^\top \text{ for some } \bm B\in\RR_+^{K\times K}\}$. 
For any $\bm Q, \bm R\in\mathbb S^K$, the relations $\bm Q\succeq\bm R$, $\bm Q\succeq_{\mathcal C}\bm R$, and $\bm Q\succeq_{\mathcal {C}^*}\bm R$ mean that $\bm Q-\bm R$ is an element of $\mathbb S_+^K$, $\mathcal C$, and $\mathcal C^*$, respectively. {\color{black} We denote the $j$-th row ($j$-th column) of a matrix $\bm M$ as $\bm M_{j:}$ ($\bm M_{:j}$).} All random variables are designated by tilde signs (\eg, $\tilde{\bm\xi}$), while their realizations are denoted without tildes (\eg, $\bm{\xi}$).  The characteristic function of a set $\mathcal S$ is defined as $\chi_{\mathcal S}(\bm \xi)=0$ if $\bm \xi\in\mathcal S$; $=\infty$ otherwise.

\section{Problem Formulation}
\label{sect:problem-formulation}
We study two-stage distributionally robust linear programs of the form
\begin{equation}
\label{eq:DRSP}
\begin{array}{l@{\quad}l@{\quad}l}
\minimize  &\displaystyle\bm c^\top\bm x + \mathcal Z(\bm x)\\
\subjectto & \displaystyle \bm x\in\mathcal X,
\end{array}
\end{equation}
where $\mathcal X\subseteq\mathbb R^{N_1}$ is the feasible set of the here-and-now decisions, $\bm c^\top\bm x$ is the here-and-now cost, and $\mathcal Z(\bm x)$ is the worst-case expected wait-and-see cost. Formally, we set
\begin{equation}
	\label{eq:WCE}
	\mathcal Z(\bm x)=\sup_{\PP\in\Phat}\mathbb E_{\mathbb P}\left[Z(\bm x,\tilde{\bm\xi})\right],
\end{equation}
where $\xit\in\Xi\subseteq\mathbb R^K$ is a random vector comprising the uncertain problem parameters, and $\Phat$ is an ambiguity set that contains the possible distributions of $\xit$.
The recourse function $Z(\bm x,\bm\xi)$ in \eqref{eq:WCE} constitutes the optimal value of the recourse problem, that is, 
\begin{equation}
\label{eq:Recourse}
\begin{array}{ccll}
Z(\bm x,\bm{\xi})=&\inf  &\displaystyle(\bm Q\bm\xi+\bm q)^\top \bm y\\
&\st & \displaystyle\bm y\in\mathbb R^{N_2}\\
&& \displaystyle\bm T(\bm x)\bm\xi+\bm h(\bm x)\leq \bm W\bm y,
\end{array}
\end{equation}
where $\bm T(\bm x)\in\RR^{M\times K}$ and $\bm h(\bm x)\in\RR^M$ are matrix- and vector-valued affine functions, respectively. 
The dual of the recourse problem is given by
\begin{equation}
\label{eq:Recourse_dual}
\begin{array}{lcll}
Z_{\mathrm d}(\bm x,\bm{\xi})=&\sup  &\displaystyle(\bm T(\bm x)\bm\xi+\bm h(\bm x))^\top \bm 
p\\
&\st & \displaystyle\bm p\in\mathbb R^{M}_+\\
&& \bm Q\bm\xi+\bm q=\bm W^\top \bm p.
\end{array}
\end{equation}

	We introduce the following standard terminology that will be used throughout the paper. 
\begin{defi}[Complete Recourse]
	\label{defi:complete_recourse}
	We say that the two-stage distributionally robust linear program \eqref{eq:DRSP} has complete recourse if there exists $\bm y^+\in\RR^{N_2}$ with $\bm W\bm y^+>\bm 0$. 
\end{defi}
Complete recourse implies that problem \eqref{eq:Recourse} is feasible for every $\bm x\in\RR^{N_1}$ and $\bm\xi\in\RR^K$. Indeed, it implies that there is always a $\lambda>0$ such that $\bm y=\lambda\bm y^+$ exceeds $\bm T(\bm x)\bm\xi+\bm h(\bm x)$. 
\begin{defi}[Sufficiently Expensive Recourse]
	\label{as:dual_complete}
	We say that the two-stage distributionally robust linear program \eqref{eq:DRSP} has sufficiently expensive recourse if for any fixed $\bm\xi\in\Xi$ the dual problem \eqref{eq:Recourse_dual} is feasible. 
\end{defi}
If problem \eqref{eq:DRSP} has complete recourse, then $Z(\bm x,\bm\xi)<+\infty$ for every $\bm x\in\RR^{N_1}$ and $\bm\xi\in\RR^K$. On the other hand, if problem \eqref{eq:DRSP} has sufficiently expensive recourse, then $Z(\bm x,\bm\xi)>-\infty$ for every $\bm x\in\mathcal X$ and $\bm\xi\in\Xi$. If both conditions are satisfied, then $Z(\bm x,\bm\xi)$ is finite. Each condition on itself implies that  strong duality holds between the  primal and dual linear programs \eqref{eq:Recourse} and~\eqref{eq:Recourse_dual}, respectively. Throughout this paper, we will always assume that problem \eqref{eq:DRSP} has sufficiently expensive recourse. This is a weak condition that is even satisfied by many problems with induced constraints. The complete recourse assumption, which rules out induced constraints, will only be imposed occasionally to obtain stronger results. 

%\begin{defi}[Dual Complete Recourse]
%	\label{defi:dual_complete_recourse}
%	We say that the two-stage distributionally robust linear program \eqref{eq:DRSP} has dual complete recourse if $\{\bm W^\top\bm p:\bm p\in\mathbb R_+^{M}\}=\mathbb R^{N_2}$.
%\end{defi}
%One can show via duality that dual complete recourse holds if and only if  the vector $\bm 0$ is the only element in the recession cone $\{\bm y\in\mathbb R^{N_2}:\bm W\bm y\geq \bm 0\}$, see \cite[Section 2.1.3]{shapiro2014lectures}.
%Assumption \ref{as:dual_complete} ensures that strong duality holds for the primal and dual pair \eqref{eq:Recourse} and \eqref{eq:Recourse_dual}. It is satisfied, for example, if the recourse problem  has the \emph{dual} complete recourse property.

Following \cite{MEK15:Wasserstein,ZG15:Wasserstein}, we assume henceforth that the true distribution of $\xit$ is unknown but that we have access to $I$ samples $\xih_1,\dots,\xih_I$ from this distribution. In this case, we can define the empirical distribution $\PPhat_I=\frac{1}{I}\sum_{i\in[I]}\delta_{\xih_i}$, that is, the uniform distribution on the samples. The ambiguity set $\hat{\mathcal P}$ in \eqref{eq:DRSP} can then be defined as the family of all distributions that are close to the empirical distribution $\hat{\mathbb P}_I$ with respect to the Wasserstein metric. 
\begin{defi}[Wasserstein Metric]
	For any $r\geq 1$, let $\mathcal M^r(\Xi)$ be the set of all probability distributions $\mathbb P$ supported on $\Xi$ satisfying $\mathbb E_{\mathbb P}[d(\tilde{\bm\xi},\bm \xi_0)^r]=\int_\Xi d(\bm\xi,\bm \xi_0)^r~\mathbb P(\rm{d}\bm\xi)<\infty$ where $\bm{\xi}_0\in\Xi$ is some reference point, and $d({\bm\xi},\bm \xi_0)$ is a continuous reference metric on $\Xi$. For any $r\geq 1$, the $r$-Wasserstein distance between two distributions $\mathbb P_1,\mathbb P_2\in\mathcal M^r(\Xi)$ is defined as
	\begin{equation*}
	W^r(\mathbb P_1,\mathbb P_2)=\inf\left\{\left(\int_{\Xi^2}  d(\bm\xi_1,\bm\xi_2)^r~\QQ(\textup{d}\bm\xi_1,\textup{d}\bm\xi_2)\right)^{\frac{1}{r}}~:~
	\begin{array}{l}\QQ \text{ is a joint distribution of }\bm\xi_1\text{ and }\bm\xi_2\\ \text{with marginals $\mathbb P_1$ and $\mathbb P_2$, respectively}\end{array}\right\}.
	\end{equation*}
\end{defi}
We denote the Wasserstein ball of radius $\epsilon$ centered at the empirical distribution by
\begin{equation*}
\mathcal B_\epsilon^{r}(\hat{\mathbb P}_I)=\left\{\mathbb P\in\mathcal M^r(\Xi)~:~W^r(\mathbb P,\hat{\mathbb P}_I)\leq \epsilon\right\}.
\end{equation*}
The following theorem, which is adapted from  \cite{MEK15:Wasserstein,GK16:Wasserstein} {\color{black} and relies on the Knothe-Rosenblatt rearrangement~\cite{villani2008optimal}}, establishes that the worst-case expectation \eqref{eq:WCE} over a Wasserstein ambiguity set $\Phat=\mathcal B_\epsilon^{r}(\hat{\mathbb P}_I)$ can be reformulated in terms of a generalized moment problem and the corresponding dual robust optimization problem. To keep this paper self-contained, we prove this theorem in the appendix. 
\begin{thm}
		\label{thm:moment_problem}
If $\hat{\mathcal P}=\mathcal B_\epsilon^{r}(\hat{\mathbb P}_I)$, the worst-case expectation \eqref{eq:WCE} coincides with the optimal value of the generalized moment problem 
\begin{equation}
	\label{eq:primal_moment}
	\begin{array}{lc@{\quad}l@{\quad}l}
		\mathcal Z(\bm x)=&\displaystyle\sup&\displaystyle\frac{1}{I}\sum_{i\in[I]}\int_{\Xi} Z(\bm x,\bm\xi)\;\mathbb P_i(\textup{d}\bm\xi)\\
		&\displaystyle\st&\displaystyle\mathbb P_i\in\mathcal M^r(\Xi)\qquad\forall i\in[I]\\
		&&\displaystyle\frac{1}{I}\sum_{i\in[I]}\int_{\Xi}{d}(\bm\xi,\hat{\bm\xi}_i)^r\;\mathbb P_i(\textup{d}\bm\xi)\leq\epsilon^r.
	\end{array}
\end{equation}
Furthermore, for $\epsilon>0$ this problem admits the strong dual robust optimization problem 
\begin{equation}
	\label{eq:dual_semiinf}
	\begin{array}{lc@{\quad}l@{\quad}l}
		\mathcal Z(\bm x)=&\displaystyle\inf_{\lambda\in\mathbb R_+}&\displaystyle\epsilon^r\lambda+\frac{1}{I}\sum_{i\in[I]}\sup_{\bm\xi\in\Xi} Z(\bm x,\bm\xi)-\lambda{d}(\bm\xi,\xih_i)^r.
	\end{array}
\end{equation}
\end{thm}

All results of this paper directly extend to the class of two-stage {\em robust} optimization problems, which model the uncertainties only through their uncertainty set $\Xi$. 
\begin{rem}[Two-Stage Robust Optimization]
\label{rem:robust}
If $\Xi$ is compact and the radius $\epsilon$ of the Wassersein ball is larger than the diameter of $\Xi$, then the two-stage distributionally robust linear program \eqref{eq:DRSP} simplifies to the two-stage robust optimization problem 
\begin{equation}
\label{eq:robust}
	\begin{array}{l@{\quad}l@{\quad}l}
	\minimize  &\displaystyle\bm c^\top\bm x + \max_{\bm\xi\in\Xi} Z(\bm x,\bm\xi)\\
	\subjectto & \displaystyle \bm x\in\mathcal X.
	\end{array}
\end{equation} 
Indeed, if we set $\epsilon\geq\max_{\bm\xi,\bm\xi'\in\Xi}d(\bm\xi,\bm\xi')$, then the Wasserstein ball  $\mathcal B_\epsilon^{r}(\hat{\mathbb P}_I)$ contains all Dirac distributions $\delta_{\bm\xi}$, $\bm\xi\in\Xi$. This implies that the worst-case expected cost \eqref{eq:WCE} reduces to the worst-case cost $\max_{\bm\xi\in\Xi} Z(\bm x,\bm\xi)$, irrespective of the number and positions of the samples $\xih_i$, $i\in[I]$. 
\end{rem}

\section{Copositive Programming Reformulation}
\label{sec:CP}
Throughout this section we work with the 2-Wasserstein metric  and the 2-norm reference distance $d(\bm{\xi}_1,\bm\xi_2)=\|\bm{\xi}_1-\bm{\xi}_2\|_2$. We further assume that the support set $\Xi$ is a non-empty polyhedron of the form
\begin{equation}
	\label{eq:support}
	\Xi=\left\{\bm\xi\in\mathbb R_+^K~:~\bm S\bm\xi\leq\bm t\right\}
\end{equation}
for some $\bm S\in\RR^{J\times K}$ and $\bm t\in\RR^J$. 
Note that we assume without loss of generality that $\Xi$ is a {\color{black} (possibly unbounded)} subset of the non-negative orthant. Finally, we also assume that problem \eqref{eq:DRSP} has sufficiently expensive recourse. 
We will show that under these assumptions the two-stage distributionally robust linear program \eqref{eq:DRSP} admits an equivalent reformulation as a copositive program. 

\subsection{A Copositive Upper Bound on \texorpdfstring{$\mathcal Z(\bm x)$}{}}
To derive a copositive programming-based upper bound on $\mathcal Z(\bm x)$, we will need the following technical lemma. 
\begin{lem}
	\label{lem:copositive_semiinf_equiv}
	For any symmetric matrix $\bm M\in\mathbb S^K$, we have $\bm M\succeq_{\mathcal C}\bm 0$ if and only if 
	\begin{equation}
		\label{eq:copositive_equiv}
		\left[\bm z^\top\;1\right]\bm M\left[\bm z^\top\;1\right]^\top\geq 0\quad\forall \bm z\in\mathbb R_+^{K-1}.
	\end{equation}
\end{lem}
\begin{proof}
	To prove sufficiency, we recall that $\bm M\succeq_{\mathcal C}\bm 0$ if and only if $\bm\xi^\top\bm M\bm \xi\geq 0$ for all $\bm \xi\in\mathbb R_+^K$. Thus, \eqref{eq:copositive_equiv} follows by focusing on those $\bm{\xi}\in\mathbb R^K$ with $\xi_K=1$.
	
	To prove the converse implication, assume that \eqref{eq:copositive_equiv} holds. Hence, we have
	\begin{equation*}
		\begin{array}{cll}
			\left[\bm z^\top\;1\right]\bm M\left[\bm z^\top\;1\right]^\top\geq 0\quad\forall \bm z\in\mathbb R_+^{K-1}&\Longrightarrow&\left[t\bm z^\top\;t\right]\bm M\left[t\bm z^\top\;t\right]^\top\geq 0
			\quad\forall \bm z\in\mathbb R_+^{K-1}\;\forall t\in\mathbb R_{++}\\
			&\Longrightarrow&\left[\bm y^\top\;t\right]\bm M\left[\bm y^\top\;t\right]^\top\geq 0
			\quad\forall \bm y\in\mathbb R_+^{K-1}\;\forall t\in\mathbb R_{++},
		\end{array}
	\end{equation*}
	that is, for any fixed $\bm y\in\mathbb R_+^{K-1}$, the univariate quadratic function $\left[\bm y^\top\;t\right]\bm M\left[\bm y^\top\;t\right]^\top$ is non-negative for all $t>0$. As this function is continuous, it must in fact be non-negative for all $t\geq 0$. Thus, $\bm M\succeq_{\mathcal C} \bm 0$.
\end{proof}

{\color{black} We are now ready to derive an upper bound on the worst-case expectation~$\mathcal Z(\bm x)$. This bound is expressed as the optimal value of a copositive minimization problem and can thus be used to conservatively approximate~\eqref{eq:DRSP} with a finite-dimensional minimization problem that is principally amenable to numerical~solution.}

%We are now ready to state our first result. The following theorem provides a conservative approximation for the worst-case expectation in \eqref{eq:WCE}. 
\begin{thm}[Copositive Upper Bound]
	\label{thm:CP}
	For any fixed first-stage decision $\bm x\in\mathcal X$, the worst-case expectation $\mathcal Z(\bm x)$ in \eqref{eq:WCE} is  bounded above by the optimal value of the copositive program
	\begin{equation}
		\label{eq:CP}
		\begin{array}{rc@{\quad}l@{\quad}l}
			&\Zover(\bm x)\\
			=&\displaystyle\inf_{} & \displaystyle\epsilon^2\lambda +\frac{1}{I}\sum_{i\in[I]}\left[s_i+\rb^\top\bm\psi_i-\lambda\|\xih_i\|_2^2+\sum_{j\in[N_2+J]}\phi_{ij}{\rbl}^2_j\right]\\
			&\st &\displaystyle \lambda\in\mathbb R_+,\; s_i\in\mathbb R,\;\bm\psi_i,\bm\phi_i\in\RR^{N_2+J}&\forall i\in[I]\\
			&&\displaystyle\begin{bmatrix}
				\displaystyle\lambda\mathbb I+\Qb^\top\diag(\bm\phi_i)\Qb&\displaystyle-\frac{1}{2}\Tb(\bm x)^\top-\Qb^\top\diag(\bm\phi_i)\Wb^\top&\displaystyle-\lambda\xih_i-\frac{1}{2}\Qb^\top\bm\psi_i\\
				\displaystyle-\frac{1}{2}\Tb(\bm x)-\Wb\diag(\bm\phi_i)\Qb&\displaystyle \Wb\diag(\bm\phi_i)\Wb^\top&\displaystyle \frac{1}{2}(\Wb\bm\psi_i-\hb(\bm x))\\\displaystyle
				(-\lambda\xih_i-\frac{1}{2}\bm Q^\top\bm\psi_i)^\top &\displaystyle\frac{1}{2}(\Wb\bm\psi_i-\hb(\bm x))^\top& \displaystyle s_i
			\end{bmatrix}\succeq_{\mathcal C}\bm 0&\forall i\in[I],
		\end{array}
	\end{equation}
where \begin{equation}
	\label{eq:expanded_params}
	\quad\Qb=\begin{bmatrix}\bm Q\\\bm S\end{bmatrix},\quad\rb=\begin{bmatrix}\bm q\\-\bm t\end{bmatrix},\quad\Tb(\bm x)=\begin{bmatrix}\bm T(\bm x)\\\bm 0\end{bmatrix},\quad	\hb(\bm x)=\begin{bmatrix}\bm h(\bm x)\\\bm 0\end{bmatrix},\;\text{ and }\;\Wb=\begin{bmatrix}\bm W&\bm 0\\ \bm 0& -\mathbb I\end{bmatrix}.
\end{equation}
\end{thm}
{\color{black} 
\begin{rem}
Note that the extended recourse parameters defined in \eqref{eq:expanded_params} combine the input data of the recourse problem~\eqref{eq:Recourse} with the parameters characterizing the support set~\eqref{eq:support}.
\end{rem}}
\begin{proof}[Proof of Theorem~\ref{thm:CP}]
	By strong linear programming duality, which holds because problem \eqref{eq:DRSP} has sufficiently expensive recourse, we have  $ Z(\bm x,\bm\xi)= Z_{\mathrm d}(\bm x,\bm\xi)$ for every $\bm x\in\mathcal X$ and $\bm\xi\in\Xi$. Recalling that $r=2$ and $d(\bm{\xi}_1,\bm\xi_2)=\|\bm{\xi}_1-\bm{\xi}_2\|_2$, {\color{black} the explicit formula \eqref{eq:Recourse_dual} for the optimal value $ Z_{\mathrm d}(\bm x,\bm\xi)$ of the dual recourse problem} and the polyhedral representation \eqref{eq:support} for $\Xi$ allow us to reformulate \eqref{eq:dual_semiinf} as 
	\begin{align}
		\mathcal Z(\bm x)\;=&\;\displaystyle\inf_{\lambda\geq 0} \epsilon^2\lambda +\frac{1}{I}\sum_{i\in[I]}\sup_{\substack{\bm\xi\geq\bm 0\\\bm S\bm\xi\leq \bm t}}\; \sup_{\substack{\bm p\geq\bm 0\\\bm Q\bm{\xi}+\bm q=\bm W^\top \bm p}}\;(\bm T(\bm x)\bm\xi+\bm h(\bm x))^\top \bm p-\lambda\|\bm\xi-\hat{\bm\xi}_i\|_2^2\nonumber \\\label{eq:case2_dual_}
		=&\;\displaystyle\inf_{\lambda\geq 0} \epsilon^2\lambda +\frac{1}{I}\sum_{i\in[I]} \sup_{\substack{\bm\xi,\thetab\geq\bm 0\\\Qb\bm{\xi}+\rb=\Wb^\top \thetab}}(\Tb(\bm x)\bm\xi+\hb(\bm x))^\top  \thetab-\lambda\|\bm\xi-\hat{\bm\xi}_i\|_2^2,
	\end{align}
	where the second equality uses the definitions in \eqref{eq:expanded_params}. Note that the first $M$ components of the new decision variable $\thetab\in\RR^{M+J}$ correspond to the dual variable $\bm p$, while the remaining $J$ components represent slack variables for the support constraints $\bm S\bm\xi\leq \bm t$.  
	Next, we add the following non-convex constraints to each of the $I$ inner maximization problems in \eqref{eq:case2_dual_}. 
	\begin{equation*}
	{\rbl}_j^2=(\Wb_{:j}^\top\thetab-\Qb_{j:}^\top\bm\xi)^2=(\Qb_{j:}^\top\bm\xi)^2-2\Qb_{j:}^\top\bm \xi\thetab^\top\Wb_{:j}+(\Wb_{:j}^\top\thetab)^2\quad\forall j\in[N_2+J]
	\end{equation*}
	Note that these constraints are redundant as they follow from $\Qb\bm{\xi}+\rb=\Wb^\top\thetab$.
	As will be revealed later, however, these constraints ensure that the optimal value of the completely positive program dual to \eqref{eq:CP} coincides with $\mathcal Z(\bm x)$. Thanks to a recent result from the theory of quadratic programming~\cite{Burer09:copositive}, each of the emerging (nonconvex) quadratically constrained quadratic subproblems in \eqref{eq:case2_dual_} can be reformulated as a completely positive maximization problem. Thus, we could apply standard dualization techniques to reformulate \eqref{eq:case2_dual_} as a finite copositive minimization problem. Here, we instead pursue a more direct approach which leverages Lemma \ref{lem:copositive_semiinf_equiv}. By expressing all linear and quadratic constraints of the subproblems in Lagrangian form, we can reformulate~\eqref{eq:case2_dual_} as
	\begin{align}
		\label{eq:case2_dual__}
		\mathcal Z(\bm x)\;=&\;\;\;\;\displaystyle\inf_{\lambda\in\mathbb R_+} \epsilon^2\lambda +\frac{1}{I}\sum_{i\in[I]} \sup_{\bm\xi,\thetab\geq\bm 0}\inf_{\bm\psi_i,\bm\phi_i}\Big[(\Tb(\bm x)\bm\xi+\hb(\bm x))^\top  \thetab-\lambda\|\bm\xi-\hat{\bm\xi}_i\|_2^2+\bm\psi_i^\top\left(\Qb\bm{\xi}+\rb-\Wb^\top \thetab\right)\nonumber\\
		&\qquad\qquad\qquad\qquad\qquad\qquad\qquad\displaystyle+\sum_{j\in[N_2+J]}\phi_{ij}\left({\rbl}_j^2-(\Qb_{j:}^\top\bm\xi)^2+2\Qb_{j:}^\top\bm \xi\thetab^\top\Wb_{:j}-(\Wb_{:j}^\top\thetab)^2\right)\Big]\nonumber\\
		\leq&\;\displaystyle\inf_{\lambda\geq 0,\bm\psi_i,\bm\phi_i} \epsilon^2\lambda +\frac{1}{I}\sum_{i\in[I]} \sup_{\bm\xi,\thetab\geq\bm 0}\Big[(\Tb(\bm x)\bm\xi+\hb(\bm x))^\top  \thetab-\lambda\|\bm\xi-\hat{\bm\xi}_i\|_2^2+\bm\psi_i^\top\left(\Qb\bm{\xi}+\rb-\Wb^\top \thetab\right)\\
		&\qquad\qquad\qquad\qquad\qquad\qquad\qquad\displaystyle+\sum_{j\in[N_2+J]}\phi_{ij}\left(\rbl_j^2-(\Qb_{j:}^\top\bm\xi)^2+2\Qb_{j:}^\top\bm \xi\thetab^\top\Wb_{:j}-(\Wb_{:j}^\top\thetab)^2\right)\Big],\nonumber
	\end{align}
	{\color{black} where $\phi_{ij}$ denotes the $j$-th entry of $\bm\phi_i$.} Here, the inequality follows from interchanging the order of the supremum and the infimum operators. We observe now that the terms in square brackets constitute quadratic forms in $\bm\xi$ and $\thetab$. Thus, by introducing auxiliary epigraphical variables $s_i$, $i\in[I]$, {\color{black} to eliminate the suprema over $\bm\xi$ and $\thetab$,} we can reformulate \eqref{eq:case2_dual__} as the quadratically parameterized semi-infinite linear program
	\begin{equation*}
		\begin{array}{lc@{\quad}l@{\quad}l}
			&\displaystyle\inf &\displaystyle\epsilon^2\lambda +\frac{1}{I}\sum_{i\in[I]}\left[s_i+\rb^\top\bm\psi_i-\lambda\|\xih_i\|_2^2+\sum_{j\in[N_2+J]}\phi_{ij}{\rb}^2_j\right]\\
			&\st&\displaystyle \lambda\in\mathbb R_+,\;s_i\in\mathbb R,\; \bm\psi_i,\bm\phi_i\in\RR^{N_2+J} \qquad\qquad\qquad\qquad\qquad\qquad\qquad\qquad\;\;\;\forall i\in[I]\\
			&&\displaystyle\begin{bmatrix}\bm\xi\\\thetab \\ 1\end{bmatrix}^\top\begin{bmatrix}
				\displaystyle\lambda\mathbb I+\Qb^\top\diag(\bm\phi_i)\Qb&\displaystyle-\frac{1}{2}\Tb(\bm x)^\top-\Qb^\top\diag(\bm\phi_i)\Wb^\top&\displaystyle-\lambda\xih_i-\frac{1}{2}\Qb^\top\bm\psi_i\\
				\displaystyle-\frac{1}{2}\Tb(\bm x)-\Wb\diag(\bm\phi_i)\Qb&\displaystyle \Wb\diag(\bm\phi_i)\Wb^\top&\displaystyle \frac{1}{2}(\Wb\bm\psi_i-\hb(\bm x))\\\displaystyle
				(-\lambda\xih_i-\frac{1}{2}\bm Q^\top\bm\psi_i)^\top &\displaystyle\frac{1}{2}(\Wb\bm\psi_i-\hb(\bm x))^\top& \displaystyle s_i
			\end{bmatrix}\begin{bmatrix}\bm\xi\\\thetab \\ 1\end{bmatrix}\geq 0\\
			&&\qquad\qquad\qquad\qquad\qquad\qquad\qquad\qquad\qquad\qquad\qquad\qquad\qquad\qquad\quad\qquad\forall i\in[I]\; \forall (\bm\xi,\thetab)\in\RR_+^{K+M+J},
		\end{array}
	\end{equation*}
	 which is equivalent to the copositive program \eqref{eq:CP} by virtue of Lemma \ref{lem:copositive_semiinf_equiv}. 
\end{proof}

\subsection{A Completely Positive Reformulation of \texorpdfstring{$\mathcal Z(\bm x)$}{}}
We now derive the dual of the copositive program \eqref{eq:CP}. As we will see later, even though \eqref{eq:CP} provides an upper bound on \eqref{eq:WCE}, the optimal value of its dual problem coincides with the worst-case expectation \eqref{eq:WCE}. 
\begin{prop}
	\label{prop:CPP}
	The copositive program \eqref{eq:CP} is dual to the following completely positive program. 
	\begin{equation}
		\label{eq:CPP}
		\begin{array}{lc@{\quad}l@{\quad}l}
			\Zunder(\bm x)=&\sup  &\displaystyle\frac{1}{I}\sum_{i\in[I]}\tr(\Tb(\bm x)\bm Y_i)+\hb(\bm x)^\top \bm\gamma_i \\
			&\st & \displaystyle\bm \gamma_i\in\mathbb R^{M+J}_+,\;\bm\mu_i\in\mathbb R_+^K,\;\bm\Gamma_i\in\mathbb S_+^{M+J},\;\bm\Omega_i\in\mathbb S_+^K,\;\bm Y_i\in\mathbb R^{K\times(M+J)}&\forall i\in[I]\\
			&& \displaystyle\Qb\bm\mu_i+{\rb}=\Wb^\top\bm\gamma_i&\forall i\in[I]\\
			&&\displaystyle\Qb_{j:}^\top\bm\Omega_i\Qb_{j:}-2\Qb_{j:}^\top\bm Y_i\Wb_{:j}+\Wb_{:j}^\top\bm\Gamma_i\Wb_{:j}={\rbl}_j^2&\forall i\in[I]\;\forall j\in[N_2+J]\\
			&&\displaystyle\frac{1}{I}\sum_{i\in[I]}\tr(\bm \Omega_i)-2\xih_i^\top\bm\mu_i+\xih_i^\top\xih_i\leq\epsilon^2\\
			&&\begin{bmatrix}
				\bm\Omega_i&\bm Y_i&\bm\mu_i\\
				\bm Y_i^\top & \bm\Gamma_i & \bm\gamma_i\\
				\bm\mu_i^\top &\bm\gamma_i^\top& 1
			\end{bmatrix}\succeq_{\mathcal C^*}\bm 0&\forall i\in[I]
		\end{array}
	\end{equation}
\end{prop}
\begin{proof}
The claim follows from standard conic duality theory. Details are omitted for brevity.
\end{proof}

In the following, we show that the worst-case expectation $\mathcal Z(\bm x)$ is in fact equal to the optimal value $\Zunder(\bm x)$ of the completely positive program \eqref{eq:CPP}. 

\begin{thm}[Completely Positive Reformulation]
	\label{thm:exact_CPP}
For any fixed $\bm x\in\mathcal X$ we have $\mathcal Z(\bm x)=\Zunder(\bm x)$.
\end{thm}
\begin{proof}
	Recall from Theorem~\ref{thm:moment_problem} that $\mathcal Z(\bm x)$ coincides with the optimal value of the moment problem \eqref{eq:primal_moment}. We first prove that $\mathcal Z(\bm x)\leq\Zunder(\bm x)$. To this end, we show that any 
	feasible solution $\{\PP_i\}_{i\in[I]}$ of \eqref{eq:primal_moment} gives rise to a feasible solution to \eqref{eq:CPP} with the same objective function value. Let $\bm p(\bm\xi)$ be a measurable selector of the dual feasible set mapping $\bm\xi\rightrightarrows\{\bm{p}\in\RR_+^M:\bm Q\bm{\xi}+\bm q=\bm W^\top\bm{p}\}$, $\bm{\xi}\in\Xi$, which exists due to \cite[Corollary 14.6]{RW09:VA} and because problem \eqref{eq:DRSP} has sufficiently expensive recourse. Next, define $\thetab(\bm{\xi})=(\bm p(\bm\xi),\bm t-\bm S\bm\xi)$. By construction we have $\Qb\bm\xi+\rb=\Wb^\top \thetab(\bm\xi)$ for all $\bm\xi\in\Xi$. 
	Next, define the following candidate solution for~\eqref{eq:CPP}:
	\begin{equation}
		\label{eq:candidate_solution_1}
		\left.
		\begin{array}{ll}
			\displaystyle\bm\mu_i=\int_{\Xi}\bm{\xi}\;\mathbb P_i(\textup{d}\bm\xi),&\displaystyle\bm\Omega_i=\int_{\Xi}\bm{\xi}\bm{\xi}^\top\;\mathbb P_i(\textup{d}\bm\xi),\\[2mm]		\displaystyle\bm\gamma_i=\int_{\Xi}\thetab(\bm{\xi})\;\mathbb P_i(\textup{d}\bm\xi),&\displaystyle\bm\Gamma_i=\int_{\Xi}\thetab(\bm{\xi})\thetab(\bm\xi)^\top\;\mathbb P_i(\textup{d}\bm\xi),\quad\bm Y_i=\int_{\Xi}\bm{\xi}\thetab(\bm{\xi})^\top\;\mathbb P_i(\textup{d}\bm\xi)
		\end{array}\right\}\;\forall i\in[I].
	\end{equation}
	Since $\xit$ and $\thetab(\xit)$ are non-negative random vectors, the matrix of their moments of degree $\leq 2$ is completely positive. Thus, the candidate solution \eqref{eq:candidate_solution_1} satisfies the last constraint in \eqref{eq:CPP}. We further have
	\begin{equation*}
		\Qb\bm{\xi}+\rb=\Wb^\top \thetab(\bm{\xi})\;\;\forall\bm{\xi}\in\Xi\quad\implies\quad\displaystyle\Qb\bm\mu_i+\rb=\Wb^\top\bm\gamma_i
	\end{equation*}
	and
	\begin{equation*}
		(\Qb_{j:}^\top\bm{\xi}+{\rbl}_j)^2=(\Wb_{:j}^\top\thetab(\bm\xi))^2\;\;\forall\bm{\xi}\in\Xi\quad\implies\quad \Qb_{j:}^\top\bm\Omega_i\Qb_{j:}-2\Qb_{j:}^\top\bm Y_i\Wb_{:j}+\Wb_{:j}^\top\bm\Gamma_i\Wb_{:j}={\rbl}_j^2
	\end{equation*}
	for all $j\in[N_2+J]$. Here, the implications follow from taking expectations with respect to $\PP_i$ on both sides of the semi-infinite constraints. Thus, the candidate solution \eqref{eq:candidate_solution_1} also satisfies the first and the second constraint systems in~\eqref{eq:CPP}.  Next, the feasibility of $\{\PP_i\}_{i\in[I]}$ in the generalized moment problem \eqref{eq:primal_moment} implies that  $\frac{1}{I}\sum_{i\in[I]}\int_{\Xi}\|\bm{\xi}-\xih_i\|_2^2~\PP_i(\textup{d}\bm\xi)\leq\epsilon^2$. Expanding the squared norm term, we obtain
	\begin{equation*}
		\begin{array}{rl}
			\epsilon^2\geq\displaystyle\displaystyle\frac{1}{I}\sum_{i\in[I]}\int_{\Xi}\bm{\xi}^\top\bm\xi-2\xih_i^\top\bm\xi+\xih_i^\top\xih_i~\PP_i(\textup{d}\bm\xi)=\displaystyle\displaystyle\frac{1}{I}\sum_{i\in[I]}\tr(\bm \Omega_i)-2\xih_i^\top\bm\mu_i+\xih_i^\top\xih_i,
		\end{array}
	\end{equation*}
	and thus the candidate solution \eqref{eq:candidate_solution_1} also satisfies the penultimate constraint in \eqref{eq:CPP}. Lastly, the objective function of \eqref{eq:primal_moment} can be reformulated as 
	\begin{equation*}
		\begin{array}{c@{\quad}l@{\quad}l}
			\displaystyle\frac{1}{I}\sum_{i\in[I]}\int_{\Xi} Z(\bm x,\bm{\xi})~\PP_i(\textup{d}\bm\xi)=		\displaystyle\frac{1}{I}\sum_{i\in[I]}\int_{\Xi}(\Tb(\bm x)\bm{\xi}+\hb(\bm x))^\top  
			\thetab(\bm{\xi})~\PP_i(\textup{d}\bm\xi)=\displaystyle\frac{1}{I}\sum_{i\in[I]}\tr(\Tb(\bm x)\bm Y_i)+\hb(\bm x)^\top \bm\gamma_i,
		\end{array}
	\end{equation*}
	where the first equality follows from the observation that 	$(\Tb(\bm x)\bm{\xi}+\hb(\bm x))^\top\thetab(\bm{\xi})=(\bm T(\bm x)\bm{\xi}+\bm h(\bm x))^\top\bm p(\bm\xi)$ for all $\bm\xi\in\Xi$. Note that the rightmost term in the above equation corresponds to the objective value of the candidate solution \eqref{eq:candidate_solution_1} in \eqref{eq:CPP}. 
	We have thus shown that from any feasible solution  $\{\PP_i\}_{i\in[I]}$ to the moment problem \eqref{eq:primal_moment} we can construct a feasible solution $\{(\bm\mu_i,\bm\gamma_i,\bm\Omega_i,\bm\Gamma_i,\bm Y_i)\}_{i\in[I]}$ 	to the completely positive program \eqref{eq:CPP} that attains the same objective value. This demonstrates that $\mathcal Z(\bm x)\leq \Zunder(\bm x)$. 
	
	 To prove the converse inequality, consider any feasible solution $\{(\bm\mu_i,\bm\gamma_i,\bm\Omega_i,\bm\Gamma_i,\bm Y_i)\}_{i\in[I]}$ to \eqref{eq:CPP}, which gives rise to a moment matrix with the following completely positive decomposition,
	 \begin{equation}
	 	\label{eq:CPP_decomposition}
	 	\begin{bmatrix}
	 		\bm\Omega_i&\bm Y_i&\bm\mu_i\\
	 		\bm Y_i^\top & \bm\Gamma_i & \bm\gamma_i\\
	 		\bm\mu_i^\top &\bm\gamma_i^\top& 1
	 	\end{bmatrix}=\sum_{\ell\in\mathcal L_i}\begin{bmatrix}\bm\chi_{i\ell}\\\bm\eta_{i\ell}\\\alpha_{i\ell}\end{bmatrix}\begin{bmatrix}\bm\chi_{i\ell}\\\bm\eta_{i\ell}\\\alpha_{i\ell}\end{bmatrix}^\top,
	 \end{equation}
	 where $\mathcal L_i$ is a finite index set, while $\bm\chi_{i\ell}\in\mathbb R_+^K$, $\bm\eta_{i\ell}\in\mathbb R_+^{M+J}$ and $\alpha_{i\ell}\in\mathbb R_+$ for every $\ell\in\mathcal L_i$. Partitioning $\mathcal L_i$ into $\mathcal L_i^+=\{\ell\in\mathcal L:\alpha_{i\ell}>0\}$ and $\mathcal L_i^0=\{\ell\in\mathcal L:\alpha_{i\ell}=0\}$, the decomposition
	  \eqref{eq:CPP_decomposition} reduces to 
	 \begin{equation}
	 	\label{eq:CPP_decomposition_new}
	 	\begin{array}{rl}
	 		\begin{bmatrix}
	 			\bm\Omega_i&\bm Y_i&\bm\mu_i\\
	 			\bm Y_i^\top & \bm\Gamma_i & \bm\gamma_i\\
	 			\bm\mu_i^\top &\bm\gamma_i^\top& 1
	 		\end{bmatrix}&\displaystyle=
	 		\displaystyle\sum_{\ell\in\mathcal L_i^+}\begin{bmatrix}\bm\chi_{i\ell}\bm\chi_{i\ell}^\top& \bm\chi_{i\ell}\bm\eta_{i\ell}^\top & \alpha_{i\ell}\bm\chi_{i\ell}\\
	 			\bm\eta_{i\ell}\bm\chi_{i\ell}^\top& \bm\eta_{i\ell}\bm\eta_{i\ell}^\top & \alpha_{i\ell}\bm\eta_{i\ell}\\
	 			\alpha_{i\ell}\bm\chi_{i\ell}^\top& \alpha_{i\ell}\bm\eta_{i\ell}^\top & \alpha_{i\ell}^2\\
	 		\end{bmatrix}
	 		+\sum_{\ell\in\mathcal L_i^0}\begin{bmatrix}\bm\chi_{i\ell}\bm\chi_{i\ell}^\top& \bm\chi_{i\ell}\bm\eta_{i\ell}^\top & \bm 0\\ 
	 				\bm\eta_{i\ell}\bm\chi_{i\ell}^\top & \bm\eta_{i\ell}\bm\eta_{i\ell}^\top& \bm 0\\
	 			\bm 0^\top & \bm 0^\top& 0
	 		\end{bmatrix}.
	 	\end{array}
	 \end{equation}
	 Next, we construct a sequence of discrete distributions $\PP^\kappa_i$, $i\in[I]$, parametrized by $\kappa\in[0,1]$, that satisfy
	 \begin{equation*}
	 	\left.\begin{array}{ll}
	 	\displaystyle\PP^\kappa_i\left(\xit=\frac{\bm\chi_{i\ell}}{\alpha_{i\ell}}\right)=(1-\kappa^2){\alpha_{i\ell}^2}&\forall \ell\in\mathcal L_i^+\\
	 	\displaystyle\PP^\kappa_i\left(\xit=\xih_i+\frac{1}{\kappa}{\sqrt{|\mathcal L_i^0|}}{\bm\chi_{i\ell}}\right)=\frac{\kappa^2}{|\mathcal L_i^0|}&\forall \ell\in\mathcal L_i^0
	 	\end{array}\right\}\forall i\in[I].
	 \end{equation*}
	 Observe that each $\PP_i^\kappa$ is indeed a probability distribution since $\sum_{\ell\in\mathcal{L}_i^+}\alpha^2_{i\ell}=1$ due to~\eqref{eq:CPP_decomposition_new}. 
	 Lemma \ref{lem:CPP_decomposition_property} below implies that $\bm\chi_{i\ell}/\alpha_{i\ell}\in\Xi$ for every $\ell\in\mathcal L_i^+$ and $\xih_i+\frac{1}{\kappa}{\sqrt{|\mathcal L_i^0|}}{\bm\chi_{i\ell}}\in\Xi$ for every $\ell\in\mathcal L_i^0$. Thus, $\PP^\kappa_i$ is supported on~$\Xi$. We further have 
	 \begin{equation*}
	 	\begin{array}{rl}
	 		\displaystyle\frac{1}{I}\sum_{i\in[I]}\EE_{\PP^\kappa_i}\left[\|\xit-\xih_i\|_2^2\right]=&\displaystyle\frac{1}{I}\sum_{i\in[I]}\left[\sum_{\ell\in\mathcal L_i^+}(1-\kappa^2)\alpha_{i\ell}^2\|\bm\chi_{i\ell}/\alpha_{i\ell}-\xih_i\|_2^2 + \sum_{\ell\in\mathcal L_i^0}\frac{\kappa^2}{|\mathcal L_i^0|}\|\xih_i+{\frac{1}{\kappa}\sqrt{|\mathcal L_i^0|}}{\bm\chi_{i\ell}}-\xih_i\|_2^2\right]\\
	 		\leq&\displaystyle\frac{1}{I}\sum_{i\in[I]}\left[\sum_{\ell\in\mathcal L_i^+}\left(\bm\chi_{i\ell}^\top\bm\chi_{i\ell}-2\xih_i^\top(\alpha_{i\ell}\bm\chi_{i\ell})\right)+\xih_i^\top\xih_i+\sum_{\ell\in\mathcal L_i^0}\bm\chi_{i\ell}^\top\bm\chi_{i\ell}\right]\\
	 		=& \displaystyle\frac{1}{I}\sum_{i\in[I]}\tr(\bm\Omega_i)-2\xih_i^\top\bm\mu_i+\xih_i^\top\xih_i\leq\epsilon^2,
	 	\end{array}
	 \end{equation*}
	 where the first inequality holds since $(1-\kappa^2)\leq 1$, the second equality follows from the decomposition \eqref{eq:CPP_decomposition_new}, and the last inequality follows from the penultimate constraint in \eqref{eq:CPP}. Thus, the distributions $\PP^\kappa_i$, $i\in[I]$, are feasible in the generalized moment problem \eqref{eq:primal_moment}. 
	 We next construct feasible solutions for the dual recourse problem \eqref{eq:Recourse_dual}. For any $i\in[I]$ and $\ell\in\mathcal L_i$, we define $\bm \rho_{i\ell}$ as the vector of the first $M$ elements of $\bm\eta_{i\ell}\in\mathbb R_+^{M+J}$.   
	 Lemma \ref{lem:CPP_decomposition_property} implies that
	 	$\bm Q\bm\chi_{i\ell}/\alpha_{i\ell}+\bm q=\bm W^\top\bm\rho_{i\ell}/\alpha_{i\ell}$ for every $\ell\in\mathcal L_i^+$.
	 Thus, $\bm p_{i\ell}=\bm\rho_{i\ell}/\alpha_{i\ell}$ is feasible in the dual recourse problem \eqref{eq:Recourse_dual} at $\bm\xi=\bm\chi_{i\ell}/\alpha_{i\ell}$ for $\ell\in\mathcal L_i^+$. Next, for any $i\in[I]$, let $\hat{\bm p}_i$ be a feasible solution to the dual recourse problem \eqref{eq:Recourse_dual} at $\bm\xi=\xih_i$, which exists because problem \eqref{eq:DRSP} has sufficiently expensive recourse.  Hence, we have $\bm Q\xih_i+\bm q=\bm W^\top\hat{\bm p}_i$.  Lemma \ref{lem:CPP_decomposition_property} further implies that
	 $\bm Q\bm\chi_{i\ell}=\bm W^\top\bm\rho_{i\ell}$ for every $\ell\in\mathcal L_i^0$. Combining the last two equalities yields
	 \begin{equation*}
	 \bm Q\left(\xih_i+\frac{1}{\kappa}{\sqrt{|\mathcal L_i^0|}}\bm\chi_{i\ell}\right)+\bm q=\bm W^\top\left(\hat{\bm p}_i+\frac{1}{\kappa}{\sqrt{|\mathcal L_i^0|}}\bm\rho_{i\ell}\right)\quad\forall \ell\in\mathcal L_i^0.
	 \end{equation*}
	 Thus, $\bm p_{i\ell}=\hat{\bm p}_i+\frac{1}{\kappa}{\sqrt{|\mathcal L_i^0|}}\bm\rho_{i\ell}$ constitutes a feasible solution in the dual recourse problem \eqref{eq:Recourse_dual} at $\bm\xi=\xih_i+\frac{1}{\kappa}{\sqrt{|\mathcal L_i^0|}}\bm\chi_{i\ell}$ for $\ell\in\mathcal L_i^0$. In summary, we have 
	 \begin{equation*}
	 \label{eq:obj_bound_1}
	  Z\left(\bm x,\frac{\bm\chi_{i\ell}}{\alpha_{i\ell}}\right)= Z_{\mathrm d}\left(\bm x,\frac{\bm\chi_{i\ell}}{\alpha_{i\ell}}\right)\geq\left(\bm T(\bm{x})\frac{\bm\chi_{i\ell}}{\alpha_{i\ell}}+\bm h(\bm x)\right)^\top\bm p_{i\ell}\qquad\forall \ell\in\mathcal L_i^+
	 \end{equation*}
	 and
	 \begin{equation*}
	 \label{eq:obj_bound_2}
	  Z\left(\bm x,\xih_i+\frac{1}{\kappa}{\sqrt{|\mathcal L_i^0|}}\bm\chi_{i\ell}\right)= Z_{\mathrm d}\left(\bm x,\xih_i+\frac{1}{\kappa}{\sqrt{|\mathcal L_i^0|}}\bm\chi_{i\ell}\right)\geq\left(\bm T(\bm{x})\left(\xih_i+\frac{1}{\kappa}{\sqrt{|\mathcal L_i^0|}}\bm\chi_{i\ell}\right)+\bm h(\bm x)\right)^\top\bm p_{i\ell}\quad\forall \ell\in\mathcal L_i^0.
	 \end{equation*}
	 Using these estimates, we can now bound the objective value of the discrete distributions $\PP^\kappa_i$, $i\in[I]$, in~\eqref{eq:primal_moment}. Specifically, we obtain
	 \begin{equation*}
	 	\begin{array}{r@{\quad}l@{\quad}l}
	 		&\displaystyle\frac{1}{I}\sum_{i\in[I]}\EE_{\PP^\kappa_i}\left[Z(\bm x,\xit)\right]\\
	 		=&\displaystyle\frac{1}{I}\sum_{i\in[I]}\left[\sum_{\ell\in\mathcal L_i^+}(1-\kappa^2)\alpha_{i\ell}^2 Z\bigg(\bm x,\frac{\bm{\chi}_{i\ell}}{\alpha_{i\ell}}\bigg)+\sum_{\ell\in\mathcal L_i^0}\frac{\kappa^2}{|\mathcal L_i^0|} Z\bigg(\bm x,\xih_i+\frac{\sqrt{|\mathcal L_i^0|}}{\kappa}{\bm\chi_{i\ell}}\bigg)\right]\\
	 		\geq&\displaystyle\frac{1}{I}\sum_{i\in[I]}\left[\sum_{\ell\in\mathcal L_i^+}(1-\kappa^2)\alpha_{i\ell}^2\left(\bm T(\bm x)\frac{\bm\chi_{i\ell}}{\alpha_{i\ell}}+\bm h(\bm x)\right)^\top \bm p_{i\ell}
	 		+\sum_{\ell\in\mathcal L_i^0}\frac{\kappa^2}{|\mathcal L_i^0|}\left(\bm T(\bm x)\left(\xih_i+\frac{\sqrt{|\mathcal L_i^0|}}{\kappa}{\bm\chi_{i\ell}}\right)+\bm h(\bm x)\right)^\top \bm p_{i\ell}\right]\\
	 		=&\displaystyle\frac{1}{I}\sum_{i\in[I]}\Bigg[\sum_{\ell\in\mathcal L_i^+}(1-\kappa^2)\alpha_{i\ell}^2\left(\Tb(\bm x)\frac{\bm\chi_{i\ell}}{\alpha_{i\ell}}+\hb(\bm x)\right)^\top \frac{\bm\eta_{i\ell}}{\alpha_{i\ell}}\\
	 		&\;\displaystyle\qquad+\sum_{\ell\in\mathcal L_i^0}\frac{\kappa^2}{|\mathcal L_i^0|}\left(\Tb(\bm x)\left(\xih_i+\frac{\sqrt{|\mathcal L_i^0|}}{\kappa}{\bm\chi_{i\ell}}\right)+\hb(\bm x)\right)^\top\left([\hat{\bm p}_i^\top\;\bm 0]^\top+\frac{\sqrt{|\mathcal L_i^0|}}{\kappa}\bm\eta_{i\ell}\right)\Bigg]\\
	 		=&\displaystyle\frac{1}{I}\sum_{i\in[I]}\Bigg[\sum_{\ell\in\mathcal L_i^+}(1-\kappa^2)\left[\tr(\Tb(\bm x)\bm\chi_{i\ell}\bm\eta_{i\ell}^\top)+\hb(\bm x)^\top (\alpha_{i\ell}\bm\eta_{i\ell})\right]+\sum_{\ell\in\mathcal L_i^0}\tr(\Tb(\bm x)\bm\chi_{i\ell}\bm\eta_{i\ell}^\top)\Bigg]\\
	 		&\;\displaystyle\qquad+\displaystyle\frac{1}{I}\sum_{i\in[I]}\left[\sum_{\ell\in\mathcal L_i^0}\frac{\kappa^2}{|\mathcal L_i^0|}\left(\left(\bm T(\bm x)\left(\xih_i+\frac{\sqrt{|\mathcal L_i^0|}}{\kappa}{\bm\chi_{i\ell}}\right)+\bm h(\bm x)\right)^\top\hat{\bm p}_i+\left(\Tb(\bm x)\xih_i+\hb(\bm x)\right)^\top\frac{\sqrt{|\mathcal L_i^0|}}{\kappa}\bm\eta_{i\ell}\right)\right].
	 	\end{array}
	 \end{equation*}
	 Together with the decomposition \eqref{eq:CPP_decomposition_new}, the above estimate implies that
	 \begin{equation*}
	 \lim_{\kappa\downarrow 0}\frac{1}{I}\sum_{i\in[I]}\EE_{\PP^\kappa_i}\left[Z(\bm x,\xit)\right]\geq\frac{1}{I}\sum_{i\in[I]}\tr(\Tb(\bm x)\bm Y_i)+\hb(\bm x)^\top \bm\gamma_i.
	 \end{equation*}  
	 We have therefore shown that any feasible solution to \eqref{eq:CPP} can be used to construct a sequence of feasible solutions to the generalized moment problem \eqref{eq:primal_moment} that asymptotically attain a (weakly) larger objective value. This demonstrates that $\mathcal Z(\bm x)\geq \Zunder(\bm x)$. Thus the claim follows. 
\end{proof}
The proof of Theorem \ref{thm:exact_CPP} relies on the following lemma, which is inspired by Lemma 2.2 in~\cite{Burer09:copositive} and Proposition 3.1 in~\cite{NRZ11:mixed01}. 
\begin{lem}
	\label{lem:CPP_decomposition_property}
	If $\{(\bm\mu_i,\bm\gamma_i,\bm\Omega_i,\bm\Gamma_i,\bm Y_i)\}_{i\in[I]}$ is feasible in \eqref{eq:CPP} and if $(\bm\chi_{i\ell},\bm\eta_{i\ell},\alpha_{i\ell})\in\RR_+^K\times\RR_+^{M+J}\times\RR_+$, $\ell\in\mathcal L_i$, satisfies \eqref{eq:CPP_decomposition} for some $i\in[I]$, then 
	\begin{equation*}
			\bm\chi_{i\ell}/\alpha_{i\ell}\in\Xi \;\text{ and }\; \bm Q(\bm\chi_{i\ell}/\alpha_{i\ell})+{\bm q}=\bm W^\top(\bm\rho_{i\ell}/\alpha_{i\ell})\;\;\forall \ell\in\mathcal L_i^+,
	\end{equation*}
	while
	\begin{equation*}
		\bm\chi_{i\ell}\in\textup{recc}(\Xi) \;\text{ and }\; \bm Q\bm\chi_{i\ell}=\bm W^\top\bm\rho_{i\ell}\;\;\forall \ell\in\mathcal L_i^0,
	\end{equation*}
	where $\textup{recc}(\Xi)=\{\bm\xi\in\RR_+^K:\bm S\bm\xi\leq\bm 0\}$ is the recession cone of $\Xi$, and $\bm\rho_{i\ell}$ is the vector of the first $M$ elements of $\bm\eta_{i\ell}$. 
\end{lem}
\begin{proof}
We substitute the decomposition \eqref{eq:CPP_decomposition}  into the constraints of problem \eqref{eq:CPP} to obtain
\begin{equation}
\label{eq:linear_term}
 \displaystyle\sum_{\ell\in\mathcal L_i}\alpha_{i\ell}\left(\Wb_{:j}^\top\bm\eta_{i\ell}-\Qb_{j:}^\top\bm\chi_{i\ell}\right)={\rbl}_j
\end{equation}
and
\begin{equation}
	\label{eq:quadratic_term}
	\displaystyle\sum_{\ell\in\mathcal L_i}\left(\Wb_{:j}^\top\bm\eta_{i\ell}-\Qb_{j:}^\top\bm\chi_{i\ell}\right)^2={\rbl}_j^2
\end{equation}
for every $j\in[N_2+J]$. 
Squaring \eqref{eq:linear_term} and eliminating ${\rbl}_j^2$  by using \eqref{eq:quadratic_term} yields
\begin{equation*}
	\left(\sum_{\ell\in\mathcal L_i}\alpha_{i\ell}\left(\Wb_{:j}^\top\bm\eta_{i\ell}-\Qb_{j:}^\top\bm\chi_{i\ell}\right)\right)^2=\displaystyle\sum_{\ell\in\mathcal L_i}\left(\Wb_{:j}^\top\bm\eta_{i\ell}-\Qb_{j:}^\top\bm\chi_{i\ell}\right)^2=\displaystyle\left(\sum_{\ell\in\mathcal L_i}\alpha_{i\ell}^2\right)\sum_{\ell\in\mathcal L_i}\left(\Wb_{:j}^\top\bm\eta_{i\ell}-\Qb_{j:}^\top\bm\chi_{i\ell}\right)^2.
\end{equation*}
Here, the second equality follows from the fact that $\sum_{\ell\in\mathcal L}\alpha_{i\ell}^2=1$.
The tightness condition of the Cauchy-Schwartz inequality therefore implies that there exists $\tau\in\mathbb R$ with
\begin{equation}
	\label{eq:CS_equality}
	\Wb_{:j}^\top\bm\eta_{i\ell}-\Qb_{j:}^\top\bm\chi_{i\ell}=\tau\alpha_{i\ell}\qquad\forall\ell\in\mathcal L_i.
\end{equation}
Thus, we have
\begin{equation*}
	\Wb_{:j}^\top\bm\eta_{i\ell}-\Qb_{j:}^\top\bm\chi_{i\ell}=0\qquad\forall\ell\in\mathcal L_i^0,
\end{equation*}
which confirms the second claim.
Next, we observe from \eqref{eq:linear_term} and \eqref{eq:CS_equality} that 
\begin{equation}
\displaystyle{\rbl}_j=\sum_{\ell\in\mathcal L_i}\alpha_{i\ell}\left(\Wb_{:j}^\top\bm\eta_{i\ell}-\Qb_{j:}^\top\bm\chi_{i\ell}\right)=\sum_{\ell\in\mathcal L_i}\tau\alpha_{i\ell}^2=\tau. 
\end{equation}
Replacing $\tau$ with $\rbl_j$ in \eqref{eq:CS_equality} and using the fact that $\alpha_{i\ell}>0$ for $\ell\in\mathcal L_i^+$ yields 
\begin{equation*}
\Wb_{:j}^\top(\bm\eta_{i\ell}/\alpha_{i\ell})-\Qb_{j:}^\top(\bm\chi_{i\ell}/\alpha_{i\ell})=\rbl_j\qquad\forall\ell\in\mathcal L_i^+,
\end{equation*}
which establishes the first claim. 
\end{proof}

\subsection{A Copositive Reformulation of Problem \eqref{eq:DRSP}}
So far, we have seen that $\mathcal Z(\bm x)=\Zunder(\bm x)\leq\Zover(\bm x)$.
Unfortunately, as we exemplify below, the duality gap between $\Zunder(\bm x)$ and $\Zover(\bm x)$ can be strictly positive. 

	\begin{ex}[Infinite Duality Gap]
		\label{ex:infinite_gap}
		Consider the following pair of primal and dual recourse problems
		\begin{equation*}
			Z(x,\xi)=\inf_{y\in\RR}\{(\xi-1)y:\xi-1\leq 0y\}\;\text{ and }\;Z_d(x,\xi)=\sup_{p\in\RR_+}\{(\xi-1)p:\xi-1= 0p\},
		\end{equation*}
		respectively, and set $\Xi=\{\xi\in\RR_+:\xi\leq 1,\;-\xi\leq -1\}=\{1\}$.
Assume that there is only one sample $\hat\xi_1=1$ and that the Wasserstein radius is set to $\epsilon=1$. 
Note that both linear programs are feasible with the same optimal value $0$ for ${\xi}=1$. Theorem \ref{thm:exact_CPP} therefore implies that $\Zunder(\bm x)=\mathcal Z(\bm x)=0$. 

Under the current setting, the extended recourse parameters defined in \eqref{eq:expanded_params} are given by 
\begin{equation*}
	\Qb=\begin{bmatrix}1\\1\\-1\end{bmatrix},\;\rb=\begin{bmatrix} -1\\-1\\1\end{bmatrix},\;
	\Tb(\bm x)=\begin{bmatrix}
		1\\0\\0
	\end{bmatrix},\;\Wb=\begin{bmatrix}
	0&0&0\\0&-1&0\\0&0&-1
\end{bmatrix},\;\text{ and }\;\hb(\bm x)=\begin{bmatrix}
-1\\ 0\\0
\end{bmatrix}.
\end{equation*} 
Hence, the copositive program \eqref{eq:CP} simplifies to
	\begin{equation}
		\label{eq:CP_infinite_gap}
		\begin{array}{rc@{\quad}l@{\quad}l}
			\Zover(\bm x)=&\displaystyle\inf_{} & \displaystyle s-\psi_1-\psi_2+\psi_3+\phi_{1}+\phi_{2}+\phi_{3}\\
			&\st &\displaystyle \lambda\in\mathbb R_+,\; s\in\mathbb R,\;\bm\psi,\bm\phi\in\RR^{3}\\
			&&\displaystyle	\begin{bmatrix}
				\displaystyle\lambda+\phi_1+\phi_2+\phi_3&-\frac{1}{2}&\phi_2&-\phi_3&-\lambda-\frac{\psi_1+\psi_2-\psi_3}{2}\\
				-\frac{1}{2} & 0 & 0 & 0 & \frac{1}{2}\\
				\phi_2 & 0 & \phi_2 & 0 & -\frac{\psi_2}{2}\\
				-\phi_3 & 0 & 0 & \phi_3 & -\frac{\psi_3}{2}\\
				-\lambda-\frac{\psi_1+\psi_2-\psi_3}{2} & \frac{1}{2} & -\frac{\psi_2}{2}&-\frac{\psi_3}{2} & s
			\end{bmatrix}\succeq_{\mathcal C}\bm 0.
		\end{array}
	\end{equation}
	However, multiplying the copositive constraint from both sides with the vector $[\theta\;1\;0\;0\;0]^\top$, $\theta\geq 0$, implies that $(\lambda+\phi_1+\phi_2+\phi_3)\theta^2-\theta\geq 0$ for all $\theta\geq 0$. As no values of $\lambda$, $\phi_1$, $\phi_2$ and  $\phi_3$ can satisfy this inequality, problem \eqref{eq:CP_infinite_gap} is infeasible, \ie,~$\Zover(\bm x)=+\infty$. Thus, there is an infinite  duality gap between $\Zunder(\bm x)$ and $\Zover(\bm x)$.
	\end{ex}

Even though $\Zunder(\bm x)$ and $\Zover(\bm x)$ may differ, one can prove that $\Zunder(\bm x)=\mathcal Z(\bm x)=\Zover(\bm x)$ if the two-stage distributionally robust linear program \eqref{eq:DRSP} has complete recourse. To show this, we first prove two lemmas.
\begin{lem}
	\label{lem:complete_recourse_WW}
	If Problem \eqref{eq:DRSP} has complete recourse, then $\bm W\bm W^\top\succ_{\mathcal C}\bm 0$.
\end{lem}
\begin{proof}
	The complete recourse property is equivalent to the unboundedness of the linear program
\begin{equation*}
	\begin{array}{l@{\quad}l@{\quad}l}
		\maximize  &\displaystyle z\\
		\subjectto & \displaystyle \bm y\in\RR^{N_2},\;z\in\RR\\
		&\displaystyle \bm W\bm y\geq z\mathbf e,
	\end{array}
\end{equation*}	
whose dual linear program is given by
\begin{equation*}
	\label{eq:dual_LP_complete_recourse}
	\begin{array}{l@{\quad}l@{\quad}l}
		\minimize  &\displaystyle 0\\
		\subjectto & \displaystyle \bm\lambda\in\RR_+^M\\
		&\displaystyle \mathbf e^\top\bm\lambda=1,\;\bm W^\top\bm\lambda=\bm 0.
	\end{array}
\end{equation*}	
As the primal problem is unbounded, the dual problem is infeasible by weak duality, implying that $\bm W^\top\bm\lambda\neq\bm 0$ for all $\bm\lambda\in\RR_+^M$ such that $\mathbf e^\top\bm\lambda=1$. By rescaling, we thus have $\bm\lambda^\top\bm W\bm W^\top\bm\lambda> 0$  for all $\bm\lambda\in\RR_+^M$ such that $\bm\lambda\neq\bm 0$. This implies that
$\bm W\bm W^\top$ lies in the interior of the copositive cone $\mathcal C$.
\end{proof}
\begin{lem}[Copositive Schur Complements]
	\label{lem:M_copositive_interior}
	Consider the symmetric matrix
	\begin{equation*}
		\bm M=
		\begin{bmatrix}
			\bm A & \bm B \\
			 \bm B^\top & \bm C
		\end{bmatrix},
	\end{equation*}
	with $\bm A\succ \bm 0$. We then have $\bm M\succ_{\mathcal C} \bm 0$ if $\bm C-\bm B^\top\bm A^{-1}\bm B\succ_{\mathcal C}\bm 0$. 
\end{lem}
\begin{proof}
	Multiplying the matrix $\bm M$ from both sides with a non-negative vector $[\bm\xi^\top\;\bm\rho^\top]^\top\in\RR_+^{K+M}$ satisfying~{\color{black}$\mathbf e^\top\bm\xi+\mathbf e^\top\bm\rho=1$}, we obtain
\begin{equation*}
	\begin{array}{rll}
	[\bm\xi^\top\;\bm\rho^\top]\bm M[\bm\xi^\top\;\bm\rho^\top]^\top&=&\displaystyle\bm\xi^\top\bm A\bm\xi+2\bm\xi^\top\bm B\bm\rho+\bm\rho^\top\bm C\bm\rho\\
	&=&\displaystyle(\bm\xi+\bm A^{-1}\bm B\bm\rho)^\top\bm A(\bm\xi+\bm A^{-1}\bm B\bm\rho)+\bm\rho^\top(\bm C-\bm B^\top\bm A^{-1}\bm B)\bm\rho.
	\end{array}
\end{equation*}	
Since $\bm A\succ\bm 0$, the term $(\bm\xi+\bm A^{-1}\bm B\bm\rho)^\top\bm A(\bm\xi+\bm A^{-1}\bm B\bm\rho)$ is non-negative.  If $\bm\rho=\bm 0$, then we have $\mathbf e^\top\bm\xi=1$, which implies that this term is positive.  If $\bm\rho\neq\bm 0$, then the assumption $\bm C-\bm B^\top\bm A^{-1}\bm B\succ_{\mathcal C}\bm 0$ implies that the term $\bm\rho^\top(\bm C-\bm B^\top\bm A^{-1}\bm B)\bm\rho$ is positive. In both cases, by rescaling we find that $[\bm\xi^\top\;\bm\rho^\top]\bm M[\bm\xi^\top\;\bm\rho^\top]^\top>0$ for all $\bm\xi\in\RR_+^K$ and all $\bm\rho\in\RR_+^M$ such that $[\bm\xi^\top\;\bm\rho^\top]^\top\neq\bm 0$. Hence, $\bm M\succ_{\mathcal C}\bm 0$. 
\end{proof}
Lemmas \ref{lem:complete_recourse_WW} and \ref{lem:M_copositive_interior} enable us to prove the following exactness result. 
\begin{thm}
	\label{prop:complete_recourse}
	If problem \eqref{eq:DRSP} has complete recourse, then $\mathcal Z(\bm x)=\Zunder(\bm x)=\Zover(\bm x)$ for any fixed $\bm x\in\mathcal X$. 
\end{thm}	
\begin{proof}
	We already know from Theorem  \ref{thm:exact_CPP} that $\mathcal Z(\bm x)=\Zunder(\bm x)$. To show that $\Zover(\bm x)=\Zunder(\bm x)$, it suffices to prove strong duality between problems \eqref{eq:CP} and \eqref{eq:CPP}. Specifically, we will construct a Slater point $(\lambda^{\rm s},(s_i^{\rm s},\bm{\phi}_i^{\rm s},\bm{\psi}_i^{\rm s})_{i\in[I]})$ for problem \eqref{eq:CP}. To this end, we first set $\bm\phi_i^{\rm s}=\mathbf e$ and $\bm\psi_i^{\rm s}=\bm 0$ for all $i\in[I]$. Using this solution, the $i$-th constraint matrix in \eqref{eq:CP} can be decomposed as
	\begin{equation}
		\label{eq:constraint_decomposition}
		\begin{array}{rc@{\quad}l@{\quad}l}
			&&\displaystyle\begin{bmatrix}
			\lambda^{\rm s}\mathbb I&\displaystyle-\frac{1}{2}\bm T(\bm x)^\top-\bm Q^\top\bm W^\top&\bm S^\top &\bm 0\\
			\displaystyle-\frac{1}{2}\bm T(\bm x)-\bm W\bm Q&\displaystyle \bm W\bm W^\top &\bm 0&\displaystyle -\frac{1}{2}\bm h(\bm x)\\\displaystyle
			\displaystyle\bm S&\displaystyle \bm 0&\mathbb I&\displaystyle \bm 0\\\displaystyle
			\bm 0^\top &\displaystyle-\frac{1}{2}\bm h(\bm x)^\top& \bm 0^\top & \displaystyle \lambda^{\rm s}
			\end{bmatrix}+\begin{bmatrix}
			\Qb^\top\Qb&\bm 0&\bm 0 &\displaystyle-\lambda^{\rm s}\xih_i\\
			\displaystyle\bm 0&\bm 0&\bm 0&\bm 0\\
			\displaystyle\bm 0&\displaystyle \bm 0&\bm 0&\displaystyle \bm 0\\\displaystyle
			\displaystyle
			-\lambda^{\rm s}\xih_i^\top  &\bm 0^\top&\bm 0^\top& s_i^{\rm s}-\lambda^{\rm s}
		\end{bmatrix}.
		\end{array}
	\end{equation}
	Next, we select  $s_i^{\rm s}$ large enough to ensure that the right matrix in \eqref{eq:constraint_decomposition} is positive semidefinite. In this case, a Slater point can be obtained by ensuring that the left matrix is strictly copositive. As problem \eqref{eq:DRSP} has complete recourse, Lemma \ref{lem:complete_recourse_WW} is applicable and implies that 
	\begin{equation*}
		\begin{bmatrix}
\displaystyle \bm W\bm W^\top &\bm 0\\
\displaystyle \bm 0&\mathbb I
		\end{bmatrix}\succ_{\mathcal C} \bm 0.
	\end{equation*}
	Moreover Lemma \ref{lem:M_copositive_interior} implies that the left matrix in \eqref{eq:constraint_decomposition} is strictly copositive if 
		\begin{equation*}
			\begin{bmatrix}
				\displaystyle \bm W\bm W^\top &\bm 0\\
				\displaystyle \bm 0&\mathbb I
			\end{bmatrix}\succ_{\mathcal C}\frac{1}{\lambda^{\mathrm s}}\begin{bmatrix}
			\displaystyle-\frac{1}{2}\bm T(\bm x)^\top+\bm Q^\top\bm W^\top&\bm S^\top \\
\displaystyle -\frac{1}{2}\bm h(\bm x)^\top& \bm 0^\top
		\end{bmatrix}^\top\begin{bmatrix}
		\displaystyle-\frac{1}{2}\bm T(\bm x)^\top+\bm Q^\top\bm W^\top&\bm S^\top \\
		\displaystyle -\frac{1}{2}\bm h(\bm x)^\top& \bm 0^\top
	\end{bmatrix},
		\end{equation*}
		which is true whenever $\lambda^{\rm s}$ is sufficiently large. We have therefore constructed a strictly feasible solution to~\eqref{eq:CP}. Thus, $\mathcal \Zover(\bm x)=\Zunder(\bm x)=Z(\bm x)$ by strong conic duality. 
\end{proof}
Theorem \ref{prop:complete_recourse} implies that if problem \eqref{eq:DRSP} has complete recourse then it is equivalent to the copositive minimization problem obtained by replacing $\mathcal Z(\bm x)$ in \eqref{eq:DRSP} with $\Zover(\bm x)$. Conversely, if problem \eqref{eq:DRSP} fails to have complete recourse, it may only be possible to approximate $\mathcal Z(\bm x)$ by the optimal value of a copositive minimization problem. 
To show this, we construct a relaxation of problem \eqref{eq:CP} parameterized by $\delta\geq 0$.
\begin{equation}
	\label{eq:CP_delta}
	\begin{array}{rc@{\quad}l@{\quad}l}
		&\Zover_\delta(\bm x)\\
		=&\displaystyle\inf_{} & \displaystyle\epsilon^2\lambda+ \frac{1}{I}\sum_{i\in[I]}\left[s_i+\rb^\top\bm\psi_i-\lambda\|\xih_i\|_2^2+\sum_{j\in[N_2+J]}\phi_{ij}{\rbl}^2_j\right]\\
		&\st &\displaystyle \lambda\in\mathbb R_+,\; s_i\in\mathbb R,\;\bm\psi_i,\bm\phi_i\in\RR^{N_2+J}&\forall i\in[I]\\
		&&\displaystyle\begin{bmatrix}
			\displaystyle\lambda\mathbb I+\Qb^\top\diag(\bm\phi_i)\Qb&\displaystyle-\frac{1}{2}\Tb(\bm x)^\top-\Qb^\top\diag(\bm\phi_i)\Wb^\top&\displaystyle-\lambda\xih_i-\frac{1}{2}\Qb^\top\bm\psi_i\\
			\displaystyle-\frac{1}{2}\Tb(\bm x)-\Wb\diag(\bm\phi_i)\Qb&\displaystyle \Wb\diag(\bm\phi_i)\Wb^\top+\delta\mathbb I&\displaystyle \frac{1}{2}(\Wb\bm\psi_i-\hb(\bm x))\\\displaystyle
			(-\lambda\xih_i-\frac{1}{2}\bm Q^\top\bm\psi_i)^\top &\displaystyle\frac{1}{2}(\Wb\bm\psi_i-\hb(\bm x))^\top& \displaystyle s_i
		\end{bmatrix}\succeq_{\mathcal C}\bm 0&\forall i\in[I]
	\end{array}
\end{equation}
Note that $\delta$ only affects the middle block of the copositive matrix. One can further show that the completely positive program dual to \eqref{eq:CP_delta} constitutes a restriction of problem~\eqref{eq:CPP} with a perturbed objective function.
	\begin{equation}
		\label{eq:CPP_delta}
		\begin{array}{lc@{\quad}l@{\quad}l}
			\Zunder_\delta(\bm x)=&\sup  &\displaystyle\frac{1}{I}\sum_{i\in[I]}\left[\tr(\Tb(\bm x)\bm Y_i)+\hb(\bm x)^\top \bm\gamma_i-\delta\tr(\bm\Gamma_i) \right]\\
			&\st & \displaystyle\bm \gamma_i\in\mathbb R^{M+J}_+,\;\bm\mu_i\in\mathbb R_+^K,\;\bm\Gamma_i\in\mathbb S_+^{M+J},\;\bm\Omega_i\in\mathbb S_+^K,\;\bm Y_i\in\mathbb R^{K\times(M+J)}&\forall i\in[I]\\
			&& \displaystyle\Qb\bm\mu_i+{\rb}=\Wb^\top\bm\gamma_i&\forall i\in[I]\\
			&&\displaystyle\Qb_{j:}^\top\bm\Omega_i\Qb_{j:}-2\Qb_{j:}^\top\bm Y_i\Wb_{:j}+\Wb_{:j}^\top\bm\Gamma_i\Wb_{:j}={\rbl}_j^2&\forall i\in[I]\;\forall j\in[N_2+J]\\
			&&\displaystyle\frac{1}{I}\sum_{i\in[I]}\tr(\bm \Omega_i)-2\xih_i^\top\bm\mu_i+\xih_i^\top\xih_i\leq\epsilon^2\\
			&&\begin{bmatrix}
				\bm\Omega_i&\bm Y_i&\bm\mu_i\\
				\bm Y_i^\top & \bm\Gamma_i & \bm\gamma_i\\
				\bm\mu_i^\top &\bm\gamma_i^\top& 1
			\end{bmatrix}\succeq_{\mathcal C^*}\bm 0&\forall i\in[I]
		\end{array}
	\end{equation}
	Observe that $\delta$ only affects the objective function of this dual problem.
\begin{prop}
	\label{prop:relaxation}
	For any fixed $\bm x\in\mathcal X$, $\Zover_\delta(\bm x)=\Zunder_\delta(\bm x)$ is finite for all $\delta>0$, and $\lim_{\delta\downarrow 0}\Zover_\delta(\bm x)=\mathcal Z(\bm x)$. 
\end{prop}
\begin{proof}
	We first show that $\Zover_\delta(\bm x)=\Zunder_\delta(\bm x)$ by proving strong duality between problems \eqref{eq:CP_delta} and \eqref{eq:CPP_delta}. To this end, we first construct a Slater point $(\lambda^{\rm s},(s_i^{\rm s},\bm{\phi}_i^{\rm s},\bm{\psi}_i^{\rm s})_{i\in[I]})$ for problem \eqref{eq:CP_delta}. Specifically, we set $\bm{\psi}_i^{\rm s}=\bm 0$ and $\bm{\phi}_i^{\rm s}=\bm 0$ for all $i\in[I]$, and we select $\lambda^{\rm s}$ satisfying $\lambda^{\rm s}\mathbb I\succ\frac{1}{4\delta}\Tb(\bm x)\Tb(\bm x)^\top$. This is possible because $\delta>0$. A standard Schur complement argument then implies that for all sufficiently large $s_i^{\rm s}>0$, $i\in[I]$, we have
	\begin{equation*}
	\begin{bmatrix}
	\lambda^{\rm s}\mathbb I&\displaystyle-\frac{1}{2}\Tb(\bm x)^\top\\
	\displaystyle-\frac{1}{2}\Tb(\bm x)&\displaystyle \delta\mathbb I 
	\end{bmatrix}\succ\bm 0\implies	\begin{bmatrix}
	\lambda^{\rm s}\mathbb I&\displaystyle-\frac{1}{2}\Tb(\bm x)^\top\\
	\displaystyle-\frac{1}{2}\Tb(\bm x)&\displaystyle \delta\mathbb I 
\end{bmatrix}\succ\frac{1}{s_i^{\rm s}}\begin{bmatrix}
\displaystyle\lambda^{\rm s}\xih_i\\
\displaystyle\frac{1}{2}\hb(\bm x)
\end{bmatrix}\begin{bmatrix}
\displaystyle\lambda^{\rm s}\xih_i\\
\displaystyle\frac{1}{2}\hb(\bm x)
\end{bmatrix}^\top\qquad\forall i\in[I].
	\end{equation*}
A second Schur complement argument then ensures that
	\begin{equation}
	\label{eq:copositivity_Slater}
	\begin{array}{rc@{\quad}l@{\quad}l}
	&&\displaystyle\begin{bmatrix}
	\lambda^{\rm s}\mathbb I&\displaystyle-\frac{1}{2}\Tb(\bm x)^\top&\displaystyle-\lambda^{\rm s}\xih_i\\
	\displaystyle-\frac{1}{2}\Tb(\bm x)&\displaystyle \delta\mathbb I &\displaystyle -\frac{1}{2}\hb(\bm x)\\\displaystyle
	-\lambda^{\rm s}\xih_i^\top &\displaystyle-\frac{1}{2}\hb(\bm x)^\top& \displaystyle s_i^{\rm s}
	\end{bmatrix}\succ\bm 0\qquad\forall i\in[I].
	\end{array}
	\end{equation}
	Consequently, the matrix on the left-hand side of \eqref{eq:copositivity_Slater} is in the interior of the copositive cone $\mathcal C$. This proves that $(\lambda^{\rm s},(s_i^{\rm s},\bm{\phi}_i^{\rm s},\bm{\psi}_i^{\rm s})_{i\in[I]})$ is indeed a Slater point for \eqref{eq:CP_delta}. Hence, $\Zover_\delta(\bm x)=\Zunder_\delta(\bm x)$ by strong conic duality. As problem \eqref{eq:CP_delta} is feasible, we have $\Zover_\delta(\bm x)<+\infty$ for any fixed $\delta>0$. Moreover, as problem \eqref{eq:DRSP} has sufficiently expensive recourse, we have $Z(\bm x,\bm\xi)>-\infty$ for any fixed $\bm x\in\mathcal X$ and $\bm\xi\in\Xi$. Since $\Phat$ is non-empty, evaluating the worst-case expectation in \eqref{eq:WCE} yields $\Zunder(\bm x)=\mathcal Z(\bm x)>-\infty$, where the equality follows from Theorem \ref{thm:exact_CPP}. Thus, the completely positive program \eqref{eq:CPP} and its restriction \eqref{eq:CPP_delta} are both feasible, implying that $\Zover_\delta(\bm x)=\Zunder_\delta(\bm x)>-\infty$. This proves finiteness of $\Zover_\delta(\bm x)$. 
	
	To prove the second claim, we observe that 	$\Zunder_\delta(\bm x)$ constitutes a pointwise supremum of a family of affine functions in $\delta$. Thus, $\Zunder_\delta(\bm x)$ is convex and lower-semicontinuous in $\delta$ for every fixed $\bm x\in\mathcal X$. Since $\Zunder_\delta(\bm x)$ is also non-increasing in $\delta$ by construction, it is indeed right-continuous. Thus, we have
	\begin{equation}
		\label{eq:limit_delta}
			\lim_{\delta\downarrow 0}\Zover_\delta(\bm x)=\lim_{\delta\downarrow 0}\Zunder_\delta(\bm x)=\Zunder_0(\bm x)=\mathcal Z(\bm x).
	\end{equation}
	Here, the first equality holds because $\Zover_\delta(\bm x)=\Zunder_\delta(\bm x)$ for $\delta>0$, while the second equality follows from the right-continuity of $\Zunder_\delta(\bm x)$. The last equality is due to Theorem \ref{thm:exact_CPP}. This completes the proof. 
\end{proof}

The findings of this section culminate in the following main theorem.
\begin{thm}
	\label{thm:CP_DRSP}
	Consider the following family of copositive programs parametrized in $\delta$. 
	\begin{equation}
	\label{eq:case2}
	\begin{array}{rc@{\quad}l@{\quad}l}
	&\displaystyle\minimize & \displaystyle\bm c^\top\bm x +\epsilon^2\lambda +\frac{1}{I}\sum_{i\in[I]}\left[s_i+\rb^\top\bm\psi_i-\lambda\|\xih_i\|_2^2+\sum_{j\in[N_2+J]}\phi_{ij}{\rbl}^2_j\right]\\
	&\subjectto &\displaystyle \bm x\in\mathcal X,\;\lambda\in\mathbb R_+,\; s_i\in\mathbb R,\;\bm\psi_i,\bm\phi_i\in\RR^{N_2+J}&\forall i\in[I]\\
	&&\displaystyle\begin{bmatrix}
	\displaystyle\lambda\mathbb I+\Qb^\top\diag(\bm\phi_i)\Qb&\displaystyle-\frac{1}{2}\Tb(\bm x)^\top-\Qb^\top\diag(\bm\phi_i)\Wb^\top&\displaystyle-\lambda\xih_i-\frac{1}{2}\Qb^\top\bm\psi_i\\
	\displaystyle-\frac{1}{2}\Tb(\bm x)-\Wb\diag(\bm\phi_i)\Qb&\displaystyle \Wb\diag(\bm\phi_i)\Wb^\top+\delta\mathbb I&\displaystyle \frac{1}{2}(\Wb\bm\psi_i-\hb(\bm x))\\\displaystyle
	(-\lambda\xih_i-\frac{1}{2}\bm Q^\top\bm\psi_i)^\top &\displaystyle\frac{1}{2}(\Wb\bm\psi_i-\hb(\bm x))^\top& \displaystyle s_i
	\end{bmatrix}\succeq_{\mathcal C}\bm 0&\forall i\in[I]
	\end{array}
	\end{equation}
	Then, the following statements hold. 
	\begin{enumerate}[(i)]
		\item If $\delta=0$ and \eqref{eq:DRSP} has complete recourse, then \eqref{eq:case2} is equivalent to \eqref{eq:DRSP}. 
		\item If $\delta=0$ and \eqref{eq:DRSP} fails to have complete recourse, then \eqref{eq:case2} provides an upper bound on \eqref{eq:DRSP}. 
		\item If $\delta>0$, then \eqref{eq:case2} provides a lower bound on \eqref{eq:DRSP}. 
		\item If $\mathcal X$ is compact, then the optimal value of \eqref{eq:case2} converges to that of \eqref{eq:DRSP} for $\delta\downarrow 0$. Moreover, every cluster point $\bm x^\star$ of a sequence $\{\bm x^\star_\delta\}_{\delta\downarrow 0}$  of minimizers for \eqref{eq:case2} is a minimizer for \eqref{eq:DRSP}. 
	\end{enumerate}
\end{thm}	
\begin{proof}
	Replacing $\mathcal Z(\bm x)$ in \eqref{eq:DRSP} with $\Zover_\delta(\bm x)$ yields \eqref{eq:case2}. Assertion (i) thus follows from Theorem~\ref{prop:complete_recourse}, while assertion (ii) follows from Theorem \ref{thm:CP}, which implies that $\mathcal Z(\bm x)\leq\Zover(\bm x)=\Zover_0(\bm x)$ for every $\bm x\in\mathcal X$. Assertion~(iii) holds because $\Zover_\delta(\bm x)=\Zunder_\delta(\bm x)\leq \Zunder(\bm x)=\mathcal Z(\bm x)$, where the first equality follows from Proposition~\ref{prop:relaxation}, the inequality holds because problem \eqref{eq:CPP_delta} constitutes a relaxation of \eqref{eq:CP}, and the second equality is due to Theorem \ref{thm:exact_CPP}.  As for assertion~(iv),  recall that $\Zunder_\delta(\bm x)$ is the optimal value of problem \eqref{eq:CPP_delta} and thus constitutes a pointwise supremum of affine functions in $\bm x$. Therefore,  
	 $\Zover_\delta(\bm x)=\Zunder_\delta(\bm x)$ is convex and lower semicontinuous in $\bm x$ for every fixed $\delta>0$. %Proposition \ref{prop:relaxation} further implies that $\Zunder_\delta(\bm x)$ is finite for every $\bm x\in\mathcal X$ and $\delta>0$. 
	 As $\mathcal X$ is compact, we may thus conclude that there exists a minimizer  $\bm x^\star_\delta\in\argmin_{\bm x\in\mathcal X}\bm c^\top\bm x+\Zover_\delta(\bm x)$ for every $\delta>0$. Next, note that the sequence of functions $\{\Zover_\delta(\bm x)\}_{\delta\downarrow 0}$ is non-decreasing and thus epi-converges to $\mathcal Z(\bm x)$ by \cite[Proposition 7.4(d)]{RW09:VA}. Moreover, the sequence of minimizers $\{\bm x^\star_\delta\}_{\delta\downarrow 0}$ admits at least one cluster point $\bm x^\star\in\mathcal X$.  By \cite[Proposition 7.30]{shapiro2014lectures}, $\bm x^\star$ constitutes a minimizer for~\eqref{eq:DRSP}, and we have
	\begin{equation*}
	\lim_{\delta\downarrow 0}\min_{\bm x\in\mathcal X}\bm c^\top\bm x+\Zover_\delta(\bm x)=\min_{\bm x\in\mathcal X}\bm c^\top\bm x+\mathcal Z(\bm x).
	\end{equation*}
	This completes the proof. 
\end{proof}

Theorem \ref{thm:CP_DRSP} immediately extends to two-stage robust optimization problems of the form \eqref{eq:robust}.
%As a corollary of our derivations, we show that our formulation subsumes the classical two-stage robust optimization problem \eqref{eq:robust}.
\begin{coro}[Two-Stage Robust Optimization]\label{thm:robust} Assume that $\Xi$ is bounded and set $I=1$. Moreover, choose $\epsilon\geq 0$ and $\xih_1\in\Xi$ such that $d(\bm\xi,\xih_1)\leq\epsilon$ $\forall\bm\xi\in\Xi$.  
%	\label{thm:robust} Fix any arbitrary point $\xih_1\in\Xi$ and set $\epsilon=\max_{\bm\xi,\bm\xi'\in\Xi}d(\bm\xi,\bm\xi')$. Consider the family of copositive programs parametrized in $\delta$ given by
%	\begin{equation}
%		\label{eq:CP_robust}
%		\begin{array}{rc@{\quad}l@{\quad}l}
%			&\displaystyle\minimize & \displaystyle\bm c^\top\bm x +\epsilon^2\lambda+s-\lambda\|\xih_1\|_2^2+\rb^\top\bm\psi+\sum_{j\in[N_2+J]}\phi_{j}{\rbl}^2_j\\
%			&\subjectto &\displaystyle \bm x\in\mathcal X,\; s\in\mathbb R,\;\lambda\in\RR_+,\;\bm\psi,\bm\phi\in\RR^{N_2+J}\\
%			&&\displaystyle\begin{bmatrix}
%				\lambda\mathbb I+\displaystyle\Qb^\top\diag(\bm\phi)\Qb&\displaystyle-\frac{1}{2}\Tb(\bm x)^\top-\Qb^\top\diag(\bm\phi)\Wb^\top&\displaystyle-\lambda\xih_1-\frac{1}{2}\Qb^\top\bm\psi\\
%				\displaystyle-\frac{1}{2}\Tb(\bm x)-\Wb\diag(\bm\phi)\Qb&\displaystyle \Wb\diag(\bm\phi)\Wb^\top+\delta\mathbb I&\displaystyle \frac{1}{2}(\Wb\bm\psi-\hb(\bm x))\\\displaystyle
%				(-\lambda\xih_1-\frac{1}{2}\bm Q^\top\bm\psi)^\top &\displaystyle\frac{1}{2}(\Wb\bm\psi-\hb(\bm x))^\top& \displaystyle s
%			\end{bmatrix}\succeq_{\mathcal C}\bm 0,
%		\end{array}
%	\end{equation}
	Then, the following statements hold. 
	\begin{enumerate}[(i)]
		\item If $\delta=0$ and \eqref{eq:robust} has complete recourse, then \eqref{eq:case2} is equivalent to \eqref{eq:robust}. 
		\item If $\delta=0$ and \eqref{eq:robust} fails to have complete recourse, then \eqref{eq:case2} provides an upper bound on \eqref{eq:robust}. 
		\item If $\delta>0$, then \eqref{eq:case2} provides a lower bound on \eqref{eq:robust}. 
		\item If $\mathcal X$ is compact, then the optimal value of \eqref{eq:case2} converges to that of \eqref{eq:robust} for $\delta\downarrow 0$. Moreover, every cluster point $\bm x^\star$ of a sequence $\{\bm x^\star_\delta\}_{\delta\downarrow 0}$  of minimizers for \eqref{eq:case2} is a minimizer for \eqref{eq:robust}. 
	\end{enumerate}
\end{coro}
\begin{proof} By construction of $\epsilon$ and $\xih_1$, the Wasserstein ball $\mathcal B_\epsilon^{2}(\hat{\mathbb P}_I)$ contains all Dirac distributions $\delta_{\bm\xi}$, $\bm\xi\in\Xi$. Therefore, the worst-case expected cost \eqref{eq:WCE} reduces to $\max_{\bm\xi\in\Xi}Z(\bm x,\bm\xi)$.  The claim thus follows immediately from Theorem~\ref{thm:CP_DRSP}.
\end{proof}
To our best knowledge, Corollary \ref{thm:robust} provides the first exact finite conic programming reformulation for the \emph{generic} two-stage robust optimization problem \eqref{eq:robust}. When the uncertainty appears only  in the constraints of the recourse problem \eqref{eq:Recourse}, approximation schemes based on the cutting plane method are available~\cite{thiele2009robust,zeng2013solving}. These approximations construct increasingly tight lower bounds for the wait-and-see cost $\max_{\bm\xi\in\Xi} Z(\bm x,\bm\xi)$. Each iteration is costly, however,  as it involves the solution of a bilinear maximization problem. Thus, additional assumptions on the uncertainty set $\Xi$  are needed to alleviate the computational burden of generating the cuts. For example, when $\Xi$ is a finite or a budgeted uncertainty set, then one can reformulate the corresponding bilinear maximization problems as mixed integer linear programs of moderate sizes. If problem \eqref{eq:robust} fails to have relatively complete recourse, however, then many expensive iterations may be required to obtain a first feasible solution. The results portrayed in Theorem \ref{thm:CP_DRSP} and  Corollary \ref{thm:robust} motivate an alternative conservative approximation scheme to solve problem \eqref{eq:robust} that can immediately produce a feasible solution. This is achieved by employing a tractable inner approximation for the copositive cone $\mathcal C$. We discuss this approach in the next section. 

\subsection{A Hierarchy of Semidefinite Programming Approximations}
\label{sec:SDP_approximation}
To justify the use of an approximation scheme, we first establish that two-stage distributionally robust linear programs of the form \eqref{eq:DRSP} are intractable. 
\begin{prop}
	\label{prop:DRSP_NP_hard}
	The two-stage distributionally robust linear program \eqref{eq:DRSP} is {NP}-hard even if $\bm Q=\bm 0$ and $\mathcal X$ is a polyhedron specified by a list of linear inequalities.
\end{prop}
\begin{proof}
	Recall from Remark \ref{rem:robust} that the distributionally robust linear programs \eqref{eq:DRSP} encapsulate the class of all two-stage robust optimization problems of the form \eqref{eq:robust}. The claim thus follows immediately from \cite[Theorem~3.5]{G02:robust}, which asserts that two-stage robust optimization problems with fixed recourse are NP-hard.
\end{proof}
The complexity result of Proposition \ref{prop:DRSP_NP_hard} is also plausible in view of Theorem \ref{thm:CP_DRSP} and the known fact that linear programs over copositive cones are generically intractable \cite{MK87:NP_complete_quad}.
 A tractable conservative approximation for \eqref{eq:case2} can be obtained by replacing the copositive cone $\mathcal C$ with 
\begin{equation*}
\mathcal C^0=\left\{\bm M\in\mathbb S^K : \bm M=\bm P+\bm N,\;\bm P\succeq \bm 0,\;\bm N\geq \bm 0\right\}.
\end{equation*}
By construction, we have $\mathcal C^0\subseteq \mathcal C$, but for dimensions $K\leq 4$ one can prove that $\mathcal C^0=\mathcal C$ \cite{diananda62:copositive}. For $K>4$, $\mathcal C^0$ is a strict subset of $\mathcal C$. In this case, there exists a hierarchy of semidefinite representable cones $\{\mathcal C^\ell\}_{\ell\geq 1}$ that provide increasingly tight inner approximations for $\mathcal C$ and converge in finitely many iterations to $\mathcal C$~%as $\ell\uparrow\infty$
 \cite{parrilo2000structured,bomze2002solving,DKP02:copositive,lasserre2009convexity}. If these tractable cones are used to replace $\mathcal C$ in~\eqref{eq:case2}, then the sizes of the resulting approximate problems can, however, become prohibitively large for $\ell>0$. In practice, we find that replacing the cone $\mathcal C$ with $\mathcal C^0$ is sufficient to generate solutions that enjoy an acceptable accuracy. %(see \cite{li2014distributionally} and the references therein). 
 
Theorem \ref{prop:complete_recourse} implies that if problem \eqref{eq:DRSP} has complete recourse, then the copositive program \eqref{eq:case2} with $\delta=0$ is equivalent to the two-stage distributionally robust linear program \eqref{eq:DRSP}. In numerical tests, we observe that strong duality between the conic programs \eqref{eq:CP} and \eqref{eq:CPP} holds also for many problem instances that violate the complete recourse condition. In all these cases, near-optimal solutions for \eqref{eq:DRSP} can be computed by solving the semidefinite programming approximation obtained by setting $\delta=0$ and replacing the cone $\mathcal C$  in \eqref{eq:case2} with an inner approximation $\mathcal C^\ell$. On the other hand, if strong duality fails to hold, {\color{black} then} feasible candidate solutions for~\eqref{eq:DRSP} can be computed in a similar manner by solving the semidefinite programming approximation~\eqref{eq:case2} for increasingly small values of $\delta>0$ until a suitable termination criterion is met. 
%
% On the other hand, if strong duality fails to hold, then we suggest the following bisection algorithm to conservatively solve the two-stage distributionally robust linear program~\eqref{eq:DRSP}:
% \begin{algorithm}[h!]
% 	\caption{Bisection Algorithm}
% 	\begin{algorithmic}
% 		\STATE \textbf{INPUTS:} An optimality gap tolerance $\tau>0$ and an inner approximation $\mathcal C^\ell$ for some $\ell\geq 0$.
% 		\STATE \textbf{1.} Set $\delta^1$ to any positive value and set $t=1$.
% 		\STATE \textbf{2.} Fix $\delta=\delta^t$. Solve the semidefinite program that results from replacing the cone $\mathcal C$ with $\mathcal C^\ell$ in \eqref{eq:case2}. This yields a solution $\bm x^\star$ and a wait-and-see cost $\Zover_\delta(\bm x^\star)$.
% 		\STATE \textbf{3.} Compute $\mathcal Z(\bm x^\star)=\Zunder(\bm x^\star)$ by solving the semidefinite program that results from replacing the cone $\mathcal C^*$ with the dual cone $\mathcal C^{\ell*}$ in~\eqref{eq:CPP}.
% 		\STATE \textbf{4.} If $\mathcal Z(\bm x^\star)-\Zover_\delta(\bm x^\star)\leq \tau$  then we terminate the algorithm. Otherwise, set $t\leftarrow t+1$ and $\delta^t\leftarrow\delta^{t-1}/2$ and go back to Step 2.
% 		\STATE \textbf{OUTPUTS:} A solution $\bm x^\star$ and an optimality certificate $\bm c^\top\bm x^\star+\mathcal Z(\bm x^\star)$.
% 	\end{algorithmic}
% \end{algorithm} 

	\subsection{Accounting for Risk-Aversion}
	\label{sect:risk-aversion}
	Until now we have studied two-stage distributionally robust optimization problems with an ambiguity-averse but risk-neutral decision maker in mind. The worst-case expectation~\eqref{eq:WCE} is a natural objective criterion for such agents. However, many decision makers are both ambiguity-averse {\em and} risk-averse and may thus prefer to minimize a {\color{black} worst-case optimized certainty equivalent}. In the remainder we will argue that the main results of this section naturally extend to this setting. Specifically, we consider non-decreasing, convex and piecewise affine disutility functions of the form $U(y)=\max_{t\in[T]}\{\alpha_t y+\beta_t\}$, where $\bm \alpha\in\mathbb R^T_+,\;\bm \alpha\neq\bm 0$ and $\bm \beta\in\mathbb R^T$, and we replace the worst-case expectation~\eqref{eq:WCE} with the {\color{black} worst-case optimized certainty equivalent %\eqref{eq:DRSP} using a modified worst-case expected wait-and-see cost given by
	\begin{equation}
		\label{eq:WCE_RA}
		\mathcal Z(\bm x)=\inf_{\theta\in\RR}\theta+\sup_{\PP\in\Phat}\mathbb E_{\mathbb P}\left[U\left(Z(\bm x,\tilde{\bm\xi})-\theta\right)\right]
	\end{equation}
corresponding to $U$. The worst-case optimized certainty equivalent determines an optimized payment schedule of the uncertain wait-and-see cost $Z(\bm x,\tilde{\bm\xi})$ into a fraction $\theta$ that is paid here-and-now and a remainder $Z(\bm x,\tilde{\bm\xi}) - \theta$ that is paid after the uncertainty has been observed. Optimized certainty equivalents encapsulate mean-variance and conditional value-at-risk measures as special cases, see~\cite{ben2007old}. Similar objective criteria are used in~\cite{hanasusanto2016k,NRZ11:mixed01} to model the decision maker's risk-aversion.}
	\begin{coro}
		\label{prop:CP_RA}
		Consider the following family of copositive programs parametrized in $\delta$. 
		\begin{equation}
			\label{eq:CP_RA}
			\begin{array}{rcll}
				&\hspace{-6mm}\displaystyle\inf_{} &\hspace{-1mm} \displaystyle\bm c^\top\bm x{\color{black}+\theta}+\epsilon^2\lambda +\frac{1}{I}\sum_{i\in[I]}s_i\\
				&\hspace{-6mm}\st &\hspace{-1mm}\displaystyle\bm x\in\mathcal X,\; {\color{black} \theta\in\mathbb R,} \; \lambda\in\mathbb R_+,\; s_i,\kappa_{it}\in\RR,\;\bm\psi_{it},\bm\phi_{it}\in\RR^{N_2+J}\qquad\qquad\qquad\qquad\qquad\qquad\;\;\;\;\quad\quad\;\;\forall i\in[I]\;\forall t\in[T]\\
				&&\hspace{-2mm}\displaystyle\begin{bmatrix}
					\displaystyle\lambda\mathbb I+\Qb^\top\diag(\bm\phi_{it})\Qb&\displaystyle-\frac{1}{2}\alpha_t\Tb(\bm x)^\top-\Qb^\top\diag(\bm\phi_{it})\Wb^\top&\displaystyle-\lambda\xih_i-\frac{1}{2}\Qb^\top\bm\psi_{it}\\
					\displaystyle-\frac{1}{2}\alpha_t\Tb(\bm x)-\Wb\diag(\bm\phi_{it})\Qb&\displaystyle \Wb\diag(\bm\phi_{it})\Wb^\top+\delta\mathbb I&\displaystyle \frac{1}{2}(\Wb\bm\psi_{it}-\alpha_t\hb(\bm x))\\\displaystyle
					(-\lambda\xih_i-\frac{1}{2}\bm Q^\top\bm\psi_{it})^\top &\displaystyle\frac{1}{2}(\Wb\bm\psi_{it}-\alpha_t\hb(\bm x))^\top& \displaystyle s_i+\kappa_{it}
				\end{bmatrix}\succeq_{\mathcal C}\bm 0\\
				&&\hspace{-1mm}\qquad\qquad\qquad\qquad\qquad\qquad\qquad\qquad\qquad\qquad\qquad\qquad\qquad\qquad\qquad\qquad\forall i\in[I]\;\forall t\in[T]\\
				&&\hspace{-1mm}\displaystyle\kappa_{it}={\color{black} \alpha_t\theta} -\beta_t-\rb^\top\bm\psi_{it}+\lambda\|\xih_i\|_2^2-\sum_{j\in[N_2+J]}\phi_{itj}{\rbl}^2_j\qquad\qquad\qquad\qquad\qquad\qquad\;\;\,\,\forall i\in[I]\;\forall t\in[T]
			\end{array}
		\end{equation}
		If $\mathcal Z(\bm x)$ denotes the worst-case expected disutility function~\eqref{eq:WCE_RA}, then the following statements hold. 
		\begin{enumerate}[(i)]
			\item If $\delta=0$ and \eqref{eq:DRSP} has complete recourse, then \eqref{eq:CP_RA} is equivalent to \eqref{eq:DRSP}. 
			\item If $\delta=0$ and \eqref{eq:DRSP} fails to have complete recourse, then \eqref{eq:CP_RA} provides an upper bound on \eqref{eq:DRSP}. 
			\item If $\delta>0$, then \eqref{eq:CP_RA} provides a lower bound on \eqref{eq:DRSP}. 
			\item If $\mathcal X$ is compact, then the optimal value of \eqref{eq:CP_RA} converges to that of \eqref{eq:DRSP} for $\delta\downarrow 0$. Moreover, every cluster point $\bm x^\star$ of a sequence $\{\bm x^\star_\delta\}_{\delta\downarrow 0}$  of minimizers for \eqref{eq:CP_RA} is a minimizer for \eqref{eq:DRSP}. 
		\end{enumerate}
		
%		For $\delta=0$, problem \eqref{eq:CP_RA} provides an upper bound on \eqref{eq:DRSP}. Furthermore, for any $\delta>0$, \eqref{eq:CP_RA} provides a lower bound on \eqref{eq:DRSP}. Finally, if $\mathcal X$ is compact then the optimal value of \eqref{eq:CP_RA} converges to that of \eqref{eq:DRSP} for $\delta\downarrow 0$. Moreover, every cluster point $\bm x^\star$ of the sequence $\{\bm x^\star_\delta\}_{\delta\downarrow 0}$  of minimizers for \eqref{eq:CP_RA} is a minimizer for \eqref{eq:DRSP}. 		
	\end{coro}
\begin{proof}
This is an immediate generalization of Theorem~\ref{thm:CP_DRSP}. Details are omitted for brevity.
\end{proof}

\section{Linear Programming Reformulation for \texorpdfstring{$\bm Q=\bm 0$}{}}
\label{sec:tractability}
We assume now that the uncertainty affects only the constraints of the recourse problem \eqref{eq:Recourse}, that is, we assume that $\bm Q=\bm 0$. Unless stated otherwise, we further assume throughout this section that $\Xi=\mathbb R^K$, and that the ambiguity set is constructed using the 1-Wasserstein metric with reference distance~${d}(\bm\xi_1,\bm\xi_2)=\|\bm\xi_1-\bm\xi_2\|$, where the norm  $\|\cdot\|$ is defined through
\begin{equation}
\label{eq:case1_tractable_norm}
\|\bm\xi\|=\mathbf e^\top\max\left\{w_+\cdot\bm\xi,-w_-\cdot\bm\xi\right\}
\end{equation}
for some positive scaling parameters $w_+$ and $w_-$. Note, that \eqref{eq:case1_tractable_norm} reduces to the 1-norm if $w_+=w_-=1$. 
 
Finally, we always assume that \eqref{eq:DRSP} has sufficiently expensive recourse. Under these assumptions, the two-stage distributionally robust linear program \eqref{eq:DRSP} admits an equivalent reformulation as a tractable linear program. 
\subsection{Tractable Formulation}
	We first establish the main tractability result for the two-stage distributionally robust linear program~\eqref{eq:DRSP}.
	\begin{thm}
		\label{thm:case1_tractable}
		The two-stage distributionally robust optimization problem \eqref{eq:DRSP} is equivalent to the tractable linear program
		\begin{equation}
		\label{eq:case1_tractable}
		\begin{array}{rc@{\quad}l@{\quad}l}
		&\displaystyle\minimize_{} &\displaystyle \epsilon\lambda +\frac{1}{I}\sum_{i\in[I]} \bm q^\top\bm y_i\\
		&\subjectto & \displaystyle\bm x\in\mathcal X,\;\lambda\in\mathbb R_+,\; \bm y_i\in\mathbb R^{N_2}& \forall  i\in[I]\\
		&&\bm\phi_k,\;\bm\psi_k\in\mathbb R^{N_2}&\forall k\in[K]\\
		&& \displaystyle\bm T(\bm x)\hat{\bm\xi}_i+\bm h(\bm x)\leq \bm W\bm y_i& \forall i\in[I]\\
		&& \displaystyle\hspace{-2mm}\left.\begin{array}{l}\bm q^\top\bm\phi_k\leq \lambda,\; \bm q^\top\bm\psi_k\leq \lambda\\\bm T(\bm x)\mathbf e_k/w_+\leq \bm W\bm\phi_k,\;-\bm T(\bm x)\mathbf e_k/w_-\leq \bm W\bm\psi_k\end{array}\right\}& \forall k\in[K].
		\end{array}
		\end{equation}
	\end{thm}
	\begin{proof}
		By strong linear programming duality, which holds because problem \eqref{eq:DRSP} has sufficiently expensive recourse, we have  $ Z(\bm x,\bm\xi)= Z_{\mathrm d}(\bm x,\bm\xi)$ for every $\bm x\in\mathcal X$ and $\bm{\xi}\in\RR^K$. %The explicit formula \eqref{eq:Recourse} for the dual recourse problem $ Z_{\mathrm d}(\bm x,\bm\xi)$ allows us to reformulate \eqref{eq:dual_semiinf} for $r=1$ and $d(\bm{\xi}_1,\bm\xi_2)=\|\bm{\xi}_1-\bm{\xi}_2\|$ as 
		Theorem \ref{thm:moment_problem} thus implies that
		\begin{equation*}
		\begin{array}{l@{\quad}l}
		\mathcal Z(\bm x)=&\displaystyle\inf_{\lambda\geq 0} \epsilon\lambda +\frac{1}{N}\sum_{i\in[I]} \sup_{\bm\xi}\sup_{\substack{\bm p\geq\bm 0\\\bm W^\top \bm p=\bm q}}(\bm T(\bm x)\bm\xi+\bm h(\bm x))^\top \bm  p-\lambda\|\bm\xi-\hat{\bm\xi}_i\|.
		\end{array}
		\end{equation*}
		Invoking the definition of the dual norm and interchanging the order of the supremum operators over $\bm{\xi}$ and $\bm{p}$, we further obtain 
		\begin{equation*}
		\begin{array}{rl@{\quad}l}
		\mathcal Z(\bm x)
		=&\displaystyle\inf_{\lambda\geq 0} \epsilon\lambda +\frac{1}{N}\sum_{i\in[I]} \sup_{\substack{\bm p\geq \bm 0\\\bm W^\top \bm p=\bm q}}\sup_{\bm\xi}\inf_{\|\bm\gamma \|_*\leq \lambda}(\bm T(\bm x)\bm\xi+\bm h(\bm x))^\top \bm  p-\bm\gamma^\top\bm\xi+\bm\gamma^\top\hat{\bm\xi}_i.
		\end{array}
		\end{equation*}
		Next, we interchange the order of the innermost supremum over $\bm\xi$ and the infimum over $\bm{\gamma}$, which is allowed by the classical minimax theorem \cite[Proposition 5.5.4]{bertsekas2009convex} since  $\bm\gamma$ ranges over a compact set. This yields
		\begin{equation*}
		\begin{array}{rl@{\quad}l}
		\mathcal Z(\bm x)=&\displaystyle\inf_{\lambda\geq 0} \epsilon\lambda +\frac{1}{I}\sum_{i\in[I]} \sup_{\substack{\bm p\geq \bm 0\\\bm W^\top \bm p=\bm q}}\inf_{\|\bm\gamma \|_*\leq \lambda}\sup_{\bm\xi}(\bm T(\bm x)\bm\xi+\bm h(\bm x))^\top \bm  p-\bm\gamma^\top\bm\xi+\bm\gamma^\top\hat{\bm\xi}_i.
		\end{array}
		\end{equation*}
		Evaluating the inner maximization over $\bm\xi$ analytically further yields
		\begin{equation*}
		\begin{array}{rl@{\quad}l}
		\mathcal Z(\bm x)=&\displaystyle\inf_{\lambda\geq 0} \epsilon\lambda +\frac{1}{I}\sum_{i\in[I]} \sup_{\substack{\bm p\geq\bm 0\\\bm W^\top \bm p=\bm q}}\inf_{\|\bm\gamma \|_*\leq \lambda}\bm h(\bm x)^\top \bm p+\bm{\gamma}^\top\hat{\bm\xi}_i+\chi_{\{\bm\gamma=\bm T(\bm x)^\top\bm p\}}(\bm\gamma,\bm p)\\
		=&\displaystyle\inf_{\lambda\geq 0} \epsilon\lambda +\frac{1}{I}\sum_{i\in[I]} \sup_{\substack{\bm p\geq\bm 0\\\bm W^\top \bm p=\bm q}}\bm h(\bm x)^\top \bm p+(\bm T(\bm x)^\top\bm p)^\top\hat{\bm\xi}_i+\chi_{\{\|\bm T(\bm x)^\top\bm p\|_*\leq \lambda\}}(\bm p).
		\end{array}
		\end{equation*}
		The minimization over $\lambda$ in the last problem can also be evaluated analytically. In fact, the unique optimal solution is $\lambda^\star=\sup\left\{\|\bm T(\bm x)^\top\bm p\|_*~:~\bm p\in\mathbb R_+^M,\;\bm W^\top\bm p=\bm q\right\}$. Note that for any $\lambda<\lambda^\star$, the supremum over $\bm p$ would be unbounded, and any $\lambda>\lambda^\star$ would incur an unnecessarily high cost as $\lambda$ is penalized by $\epsilon$ in the objective function. We thus obtain
		\begin{equation}
		\label{eq:case1}
		\begin{array}{rl@{\quad}ll}
		\mathcal Z(\bm x)=&\displaystyle\inf& \displaystyle\epsilon\lambda +\frac{1}{I}\sum_{i\in[I]} \sup_{\substack{\bm p\geq\bm 0\\\bm W^\top \bm p=\bm q}}\bm h(\bm x)^\top \bm p+(\bm T(\bm x)^\top\bm p)^\top\hat{\bm\xi}_i\\
		&\st&\displaystyle \lambda\in\RR_+\\
		&& \displaystyle  \|\bm T(\bm x)^\top\bm p\|_*\leq\lambda \qquad\forall\bm p\in\mathbb R_+^M~:~\bm W^\top\bm p=\bm q. 
		\end{array}
		\end{equation}
		Next, the norm dual to \eqref{eq:case1_tractable_norm} is given by 
		\begin{equation*}
		\|\bm z\|_{*}=\max_{k\in[K]}\left[\max\left\{\frac{z_k}{w_+},-\frac{z_k}{w_-}\right\}\right].
		\end{equation*}
		Thus, the last constraint in \eqref{eq:case1} can be decomposed into a system of $\mathcal O(K)$ linear constraints as follows.
		\begin{equation*}
		\label{eq:case1_tractable_constraints}
		\begin{array}{cl@{\quad}ll}
		&& \|\bm T(\bm x)^\top\bm p\|_*\leq\lambda &\forall \bm p\in\mathbb R_+^M~:~\bm W^\top\bm p=\bm q \\
		\Longleftrightarrow&& \displaystyle\sup_{\substack{\bm p\geq\bm 0\\\bm W^\top\bm p=\bm q}}\mathbf e_k^\top\bm T(\bm x)^\top\bm p/w_+\leq\lambda,\; \displaystyle\sup_{\substack{\bm p\geq \bm 0\\\bm W^\top\bm p=\bm q}}-\mathbf e_k^\top\bm T(\bm x)^\top\bm p/w_-\leq\lambda& \forall k\in[K]\\
		\Longleftrightarrow&&  \exists \bm\phi_k,\bm\psi_k\in\mathbb R^{N_2}~:~\left.\begin{array}{l}\bm q^\top\bm\phi_k\leq \lambda,\; \bm q^\top\bm\psi_k\leq \lambda\\\bm T(\bm x)\mathbf e_k/w_+\leq \bm W\bm\phi_k,\;-\bm T(\bm x)\mathbf e_k/w_-\leq \bm W\bm\psi_k\end{array}\right\}& \forall k\in[K]
		\end{array}
		\end{equation*}
		Here, the second equivalence follows from dualizing the linear programs over $\bm p$, all of which are feasible because problem \eqref{eq:DRSP} has sufficiently expensive recourse. 
		The claim then follows from substituting the last constraint system into \eqref{eq:case1}. 
	\end{proof}

\begin{ex}[Regression]
	\label{ex:regression}
	Consider the least absolute deviations~(LAD) regression problem 
	\begin{equation*}
		\begin{array}{l@{\quad}l}
	\minimize&\displaystyle\EE_{\PP}\left[|\bm x^\top\xit+x_0-\tilde\chi|\right]\\
	\subjectto&\displaystyle(\bm x,x_0)\in\mathcal X.
	\end{array}
	\end{equation*}  
	The objective of this problem is to find the slope $\bm x$ and intercept $x_0$ of an affine function of the explanatory random variables $\xit$ that tightly approximates the independent variable $\tilde{\chi}$ in terms of the mean absolute deviation.  
	In statistics, however, the data-generating distribution $\PP$ of $(\xit,\tilde{\chi})$ is never known. Only the empirical distribution $\PPhat_I$ corresponding to a set of $I$ training samples is given. In this case $\PP$ is ambiguous, and it may make sense to solve the distributionally robust LAD problem
		\begin{equation*}
		\label{eq:regression}
		\begin{array}{l@{\quad}l}
		\minimize&\displaystyle\sup_{\PP\in\Phat}\EE_{\PP}\left[|\bm x^\top\xit+x_0-\tilde\chi|\right]\\
		\subjectto&\displaystyle(\bm x,x_0)\in\mathcal X,
		\end{array}
		\end{equation*}  
		which can be identified as an instance of the two-stage distributionally robust optimization problem \eqref{eq:DRSP} with recourse function
	\begin{equation*}
	Z((\bm x,x_0),(\bm\xi,\chi))=\min\{ y:y\in\mathbb R,\; y \geq \bm x^\top\bm\xi+x_0-\chi,\; y\geq\chi- x_0-\bm x^\top\bm\xi\}.
	\end{equation*}
	From equation \eqref{eq:case1} in the proof of Theorem \ref{thm:case1_tractable} it is evident that the distributionally robust LAD problem is equivalent to
	\begin{equation*}
	\begin{array}{ll@{\quad}l@{\quad}l}
	&\displaystyle\minimize &\displaystyle {\epsilon\|\bm x\|_*} +\frac{1}{N}\sum_{i\in[I]} |\bm x^\top\xih_i+x_0-\hat{\chi}_i|\\
	&\subjectto & \displaystyle(\bm x,x_0)\in\mathcal X. 
	%&& \displaystyle\bm X\hat{\bm\xi}_i-\bm x\leq \bm y_i,\;\bm x-\bm X\hat{\bm\xi}_i\leq \bm y_i& \forall i\in[I]\\
	%%&& \displaystyle\frac{\|\bm X_{:k}\|_1}{\min\{w_+,w_-\}}\leq \lambda& \forall k\in[K]
	%&& \displaystyle\|\bm X^\top(\bm p-\bm\eta)\|_{*}\leq\lambda&\forall \bm p,\bm\eta\in\mathbb R_+^L~:~\bm p+\bm\eta=\mathbf e.
	\end{array}
	\end{equation*}
	Note that the above formulation holds for arbitrary norms (not just the one defined in \eqref{eq:case1_tractable}). The second term in the objective function represents the empirical LAD loss, while the first term acts as a regularizer for the regression coefficient $\bm x$. If the reference distance is set to  the infinity norm, then we recover the celebrated LASSO regularizer \cite{tibshirani1996regression,wang2006regularized}. On the other hand,  if the reference distance is set to the 1-norm, then we obtain an infinity norm regularizer which has recently been employed in the context of logistic regression~\cite{shafieezadeh2015distributionally}. 
\end{ex}
\begin{ex}[Multi-Task Learning]
	We can extend Example \ref{ex:regression} to a distributionally robust multi-task learning problem \cite{caruana97,baxter} where several regression problems are to be solved simultaneously. 
		\begin{equation*}
		\label{eq:MTL}
		\begin{array}{l@{\quad}l}
		\minimize&\displaystyle\sup_{\PP\in\Phat}\EE_{\PP}[\|\bm X\xit+\bm x-\tilde{\bm\chi}\|_1]\\
		\subjectto&\displaystyle(\bm X,\bm x)\in\mathcal X
		\end{array}
		\end{equation*}  
	This model has many applications in marketing \cite{lenk1996hierarchical}, healthcare \cite{zhang2012multi}, natural language processing \cite{collobert2008unified}, etc. The distributionally robust multi-task learning model still constitutes an instance of problem \eqref{eq:DRSP}, where the recourse function is now given by 
	\begin{equation*}
	Z((\bm X,\bm x),(\bm\xi,\bm\chi))=\min\{\mathbf e^\top\bm y:\bm y\in\mathbb R^L,\;\bm y \geq \bm X\bm\xi+\bm x-\bm\chi,\;\bm y\geq \bm\chi-\bm x-\bm X\bm\xi\}.
	\end{equation*}
	By Theorem \ref{thm:case1_tractable}, this problem is equivalent to the linear program
	\begin{equation*}
	\begin{array}{rc@{\quad}l@{\quad}l}
	&\displaystyle\minimize &\displaystyle \frac{\epsilon}{\min\{w_+,w_-\}}\max_{k\in[K]}{\|\bm X_{:k}\|_1} +\frac{1}{N}\sum_{i\in[I]} \|\bm X\xih_i+\bm x-\hat{\bm \chi}_i\|_1\\
	&\subjectto & \displaystyle(\bm X,\bm x)\in\mathcal X.
%	&& \displaystyle\bm X\hat{\bm\xi}_i-\bm x\leq \bm y_i,\;\bm x-\bm X\hat{\bm\xi}_i\leq \bm y_i& \forall i\in[I]\\
	%%&& \displaystyle\frac{\|\bm X_{:k}\|_1}{\min\{w_+,w_-\}}\leq \lambda& \forall k\in[K]
	%&& \displaystyle\|\bm X^\top(\bm p-\bm\eta)\|_{*}\leq\lambda&\forall \bm p,\bm\eta\in\mathbb R_+^L~:~\bm p+\bm\eta=\mathbf e.
	\end{array}
	\end{equation*}
	Here, the first term in the objective function acts again as a regularizer for the regression coefficient $\bm X$, while the second term represents the empirical LAD loss. 
\end{ex}

\subsection{Complexity Analysis}
Unfortunately, tractability of the distributionally robust linear program \eqref{eq:DRSP} is lost when the reference distance is defined via a $p$-norm with $p>1$ even if all other conditions of Theorem \ref{thm:case1_tractable} remain valid.  
This can be shown by using a reduction from the NP-hard \textsc{Matrix Norm Maximization} problem \cite{S05:MatrixNorm}. 
 \\[-2mm]
 
 \fbox{\parbox{15cm}{ {\centering \textsc{Matrix Norm Maximization}\\}
 		\textbf{Instance.} Given a positive semidefinite matrix $\bm M\in\mathbb S_+^K$. \\
 		\textbf{Question.} For a fixed $q\in[1,\infty)$, compute the matrix norm $\|\bm M\|_{\infty,q}=\max_{\|\bm{z}\|_\infty\leq 1}\|\bm M\bm{z}\|_{q}$.
 	}}
 	\\

	\begin{thm}
		Computing the optimal value of \eqref{eq:DRSP} is {NP}-hard whenever the reference distance is set to $ d(\bm\xi_1,\bm\xi_2)=\|\bm\xi_1-\bm\xi_2\|_p$ for any $p>1$, even if $\bm Q=\bm 0$, $r=1$, $
		\Xi=\mathbb R^K$ and there are no first-stage decisions. 
	\end{thm}
	\begin{proof}
		Fix $p\in(1,\infty]$ and set $q=\frac{p}{p-1}\in[1,\infty)$. For any instance $\bm M\in\mathbb S_+^K$ of the \textsc{Matrix Norm Maximization} problem, construct an instance of the distributionally robust linear program \eqref{eq:DRSP} as follows. Set the parameters of the recourse problem \eqref{eq:Recourse} to $\bm Q=\bm 0$, $\bm q=\bm 0$, $\bm T(\bm x)=\left[\bm M\;\;-\bm M\right]^\top$, $\bm h(\bm x)=\bm 0$ and $\bm W=\left[\mathbb I\;\;\mathbb I\right]^\top$. Moreover, assume that there is only one sample $\xih_1=\bm 0$, and set $\epsilon=1$.  Equation \eqref{eq:case1} in the proof of Theorem \ref{thm:case1_tractable} implies that problem \eqref{eq:DRSP} is equivalent to
	\begin{equation*}
	\begin{array}{rc@{\quad}l@{\quad}l}
	&\displaystyle\minimize_{} & \lambda +\mathbf e^\top\bm y_1\\
	&\subjectto &\displaystyle \lambda\in\mathbb R_+,\; \bm y_1\in\mathbb R^{N_2}\\
	&& \displaystyle\bm 0\leq \bm y_1\\
	&& \displaystyle\|\bm M(\bm p_+-\bm p_-)\|_{q}\leq\lambda \quad\forall \bm p_+,\bm p_-\in\mathbb R_+^K~:~\bm p_++\bm p_-=\mathbf e.
	\end{array}
	\end{equation*}
Note that $\bm y_1=\bm 0$ at optimality irrespective of $\lambda$, and thus the optimal value of this problem coincides with 
	\begin{equation*}
	\begin{array}{rc@{\quad}l@{\quad}l}
	&\max\left\{ \|\bm M(\bm p_+-\bm p_-)\|_{q}:\bm p_+,\bm p_-\in\mathbb R_+^K,\;\bm p_++\bm p_-=\mathbf e\right\} =\displaystyle \max_{\|\bm{z}\|_\infty\leq 1}\|\bm M\bm{z}\|_{q}.
	\end{array}
	\end{equation*}
	We conclude that computing the optimal value of \eqref{eq:DRSP} is at least as hard as solving the NP-hard \textsc{Matrix Norm Maximization} problem. 
\end{proof}

\section{Numerical Results}
\label{sec:numerical}
We now assess the computational and statistical properties of the two-stage distributionally robust linear programs over 2-Wasserstein balls studied in Section~\ref{sec:CP}. All optimization problems are solved with MOSEK~v7 using the YALMIP interface~\cite{yalmip} on an 8-core 3.4~GHz computer with 16~GB~RAM.

\subsection{Approximation Quality}
We first assess the error introduced by approximating the copositive cone $\mathcal C$ in \eqref{eq:case2} with its semidefinite inner approximation~$\mathcal C^0$. To this end, we study recourse problems of the form 
\begin{equation}
\label{eq:Recourse_hard}
Z(\bm x,\bm{\xi})=\inf_{\bm y\in\mathbb R^{N_2}_+} \left\{ \mathbf e^\top \bm y : \bm A\bm\xi-\bm b\leq \bm y\right\} = \sum_{n\in[N_2]}\max\{\bm A_{n:}^\top\bm{\xi}-b_n,0\}= \max_{\bm\ell\in\{0,1\}^{N_2}}(\bm A\bm{\xi}-\bm b)^\top \bm{\ell},
\end{equation}
where $\bm A\in[0,1]^{N_2\times K}$, $\bm b\in[0,1]^K$ and the random vector $\xit\in\RR^K$ is supported on~$\Xi=[0,1]^K$. Note that the wait-and-see cost is independent of $\bm x$ and representable as a sum of $N_2$ max functions, which can be expressed as the pointwise maximum of $2^{N_2}$ affine functions in~$\bm{\xi}$. Recourse problems of this type are hard both in the stochastic as well as in the robust setting. Indeed, evaluating the expectation of $Z(\bm x,\xit)$ is $\#$P-hard even if $N_2=1$ and $\xit$ follows the uniform distribution on $\Xi$ \cite[Corollary~1]{HKW16:SPComplexity}. Similarly, evaluating the worst case of $Z(\bm x,\bm{\xi})$ over all $\bm\xi\in \Xi$ is strongly NP-hard \cite[Example~1.1.9]{Hanasusanto15:PhD}. The following proposition shows that the worst-case expectation of $Z(\bm x,\xit)$ over all distributions of $\xit\in\Xi$ within a given 2-Wasserstein ball can be expressed as the optimal value of a second-order cone program (SOCP) with $\mathcal O(2^{N_2})$ constraints.

\begin{prop}
	\label{prop:SOCP_equiv}
 If $\Phat=\mathcal B_\epsilon^{2}(\hat{\mathbb P}_I)$, then the worst-case expectation~\eqref{eq:WCE} of the wait-and-see cost \eqref{eq:Recourse_hard} amounts~to 
 \begin{equation}
\label{eq:SOCP_equiv}
 \begin{array}{rlll}
\displaystyle\mathcal Z(\bm x)=\inf & \displaystyle\epsilon^2\lambda +\frac{1}{I}\sum_{i\in[I]} s_i\\
 \st &\displaystyle \lambda\in\mathbb R_+,\; s_i\in\mathbb R_+,\;\bm\theta_{i\bm{\ell}},\bm\eta_{i\bm{\ell}}\in\RR_+^K&\forall i\in[I]\;\forall \bm\ell\in\{0,1\}^{N_2}\\
 &\displaystyle  s_i+b+\lambda \|\xih_i\|_2^2-\mathbf e^\top\bm\eta_{i\bm{\ell}}\geq 0& \forall i\in[I]\;\forall \bm\ell\in\{0,1\}^{N_2}\\
 &\left\|\begin{bmatrix}A^\top\bm{\ell}+2\lambda\xih_i+\bm\theta_{i\bm{\ell}}-\bm\eta_{i\bm{\ell}}\\s_i+\bm b^\top\bm{\ell}+\lambda \|\xih_i\|_2^2-\mathbf e^\top\bm\eta_{i\bm{\ell}}-\lambda\end{bmatrix}\right\|_2\leq s_i+\bm b^\top\bm{\ell}+\lambda \|\xih_i\|_2^2-\mathbf e^\top\bm\eta_{i\bm{\ell}}+\lambda& \forall i\in[I]\;\forall \bm\ell\in\{0,1\}^{N_2}.
 \end{array}
 \end{equation}
\end{prop}
\begin{proof}
	By Theorem~\ref{thm:moment_problem}, the worst-case expectation \eqref{eq:WCE} is representable as
	\begin{equation*}
	%\label{eq:semiinf_2piece}
	\begin{array}{rc@{\quad}l@{\quad}l}
	\mathcal Z(\bm x)=\displaystyle\inf_{\lambda\in\mathbb R_+}\displaystyle\epsilon^2\lambda+\frac{1}{I}\sum_{i\in[I]}\max_{\bm\xi\in\Xi}\max_{\bm\ell\in\{0,1\}^{N_2}}\bm{\ell}^\top\bm A\bm{\xi}-\bm b^\top\bm{\ell}-\|\bm\xi-\xih_i\|_2^2.
%	=&\displaystyle\inf_{\lambda\in\mathbb R_+}&\displaystyle\epsilon^2\lambda+\frac{1}{I}\sum_{i\in[I]}\max\left\{\max_{\bm\xi\in\Xi}\bm a^\top\bm\xi-b-\|\bm\xi-\xih_i\|_2^2,0\right\}.
	\end{array}
	\end{equation*}
	Thus, by introducing auxiliary epigraphical variables $s_i$, $i\in[I]$, we can reformulate the above optimization problem as the semi-infinite linear program
	\begin{equation*}
	\begin{array}{rc@{\quad}l@{\quad}l}
	\mathcal Z(\bm x)=&\displaystyle\inf_{} & \displaystyle\epsilon^2\lambda +\frac{1}{I}\sum_{i\in[I]} s_i\\
	&\st &\displaystyle \lambda\in\mathbb R_+,\; s_i\in\mathbb R_+&\forall i\in[I]\\
	&&\displaystyle \max_{\bm\xi\in\Xi}\bm{\ell}^\top\bm A\bm{\xi}-\bm{b}^\top\bm{\ell}-\|\bm\xi-\xih_i\|_2^2\leq s_i&\forall i\in[I]\;\forall \bm\ell\in\{0,1\}^{N_2}.
%	\max_{\bm\xi\in\Xi}\bm a^\top\bm\xi-b-\lambda\|\bm\xi-\xih_i\|_2^2 \leq s_i& \forall i\in[I],\;\forall\bm{\xi}\in\Xi.
	\end{array}
	\end{equation*}
	Strong quadratic programming duality implies that the $(i,\bm{\ell})$-th semi-infinite constraint is satisfied if and only if there exist $\bm{\theta}_{i\bm\ell},\bm{\eta}_{i\bm\ell}\in\RR_+^K$ that satisfy the hyperbolic constraint
	\begin{equation*}
	\begin{array}{rcll}
	\displaystyle\frac{1}{4}\left\|\bm A^\top\bm{\ell}+2\lambda\xih_i+\bm\theta_{i\bm{\ell}}-\bm\eta_{i\bm{\ell}}\right\|_2^2\leq \lambda\left(s_i+\bm{b}^\top\bm{\ell}+\lambda \|\xih_i\|_2^2-\mathbf e^\top\bm\eta_{i\bm{\ell}}\right),
	\end{array}
	\end{equation*}
%	
%		Next, by explicitly expressing the definition of the support set $\Xi$, we can reformulate the $i$-th semi-infinite constraint as
%	\begin{equation*}
%	\begin{array}{rcll}
%	\displaystyle \max_{\bm 0\leq\bm\xi\leq\mathbf e}\bm a^\top\bm\xi-b-\lambda\|\bm\xi-\xih_i\|_2^2 \leq s_i&\Longleftrightarrow&\displaystyle \max_{\bm\xi}\min_{\bm\theta_i,\bm\eta_i\geq\bm 0}\bm a^\top\bm\xi-b-\lambda\|\bm\xi-\xih_i\|_2^2 +\bm\theta_i^\top\bm\xi-\bm\eta_i^\top\bm\xi+\mathbf e^\top\bm\eta_i\leq s_i\\
%	&\Longleftrightarrow&\displaystyle \min_{\bm\theta_i,\bm\eta_i\geq\bm 0}\max_{\bm\xi}\bm a^\top\bm\xi-b-\lambda\|\bm\xi-\xih_i\|_2^2 +\bm\theta_i^\top\bm\xi-\bm\eta_i^\top\bm\xi+\mathbf e^\top\bm\eta_i\leq s_i\\
%	&\Longleftrightarrow&\displaystyle \exists\bm\theta_i,\bm\eta_i\geq\bm 0:\max_{\bm\xi}\bm a^\top\bm\xi-b-\lambda\|\bm\xi-\xih_i\|_2^2 +\bm\theta_i^\top\bm\xi-\bm\eta_i^\top\bm\xi+\mathbf e^\top\bm\eta_i\leq s_i\\
%	&\Longleftrightarrow&\displaystyle \exists\bm\theta_i,\bm\eta_i\geq\bm 0:\frac{1}{4}\|\bm a+2\lambda\xih_i+\bm\theta_i-\bm\eta_i\|_2^2\leq \lambda\left(s_i+b+\lambda \|\xih_i\|_2^2-\mathbf e^\top\bm\eta_i\right).
%	\end{array}
%	\end{equation*}
%	Here, the second equivalence holds from strong quadratic programming duality because $\Xi$ has a non-empty interior, while the last equivalence follows from evaluating the maximization over $\bm\xi$ analytically. 
which is equivalent to the standard SOCP constraints
	 \begin{equation*}
	 \begin{array}{l}
	 \displaystyle  \lambda\geq 0,\,s_i+b+\lambda \|\xih_i\|_2^2-\mathbf e^\top\bm\eta_{i\bm{\ell}}\geq 0,\\[1ex]
	 \displaystyle \left\|\begin{bmatrix}A^\top\bm{\ell}+2\lambda\xih_i+\bm\theta_{i\bm{\ell}}-\bm\eta_{i\bm{\ell}}\\s_i+\bm b^\top\bm{\ell}+\lambda \|\xih_i\|_2^2-\mathbf e^\top\bm\eta_{i\bm{\ell}}-\lambda\end{bmatrix}\right\|_2\leq s_i+\bm b^\top\bm{\ell}+\lambda \|\xih_i\|_2^2-\mathbf e^\top\bm\eta_{i\bm{\ell}}+\lambda.
	 \end{array}
	 \end{equation*} Thus, the worst-case expectation \eqref{eq:WCE} indeed coincides with the optimal value of \eqref{eq:SOCP_equiv}. 
\end{proof}

As it is hard to evaluate $\mathcal Z(\bm x)$ exactly---as reflected by the exponential size of the SOCP~\eqref{eq:SOCP_equiv}---we now investigate two efficient methods for evaluating $\mathcal Z(\bm x)$ approximately: the $\mathcal C_0$ approximation of the equivalent copositive program~\eqref{eq:CP} and a state-of-the-art quadratic decision rule approximation. The $\mathcal C_0$ approximation is obtained by solving~\eqref{eq:CP} with inputs  $\bm S=\mathbb I$, $\bm t=\mathbf e$, $\bm Q=\bm 0$, $\bm q=\mathbf e$, $\bm T(\bm x)=[\bm A^\top\;\bm 0]^\top$, $\bm h(\bm x)=[-\bm b^\top\;\bm 0^\top]^\top$, and $\bm W=[\mathbb I\;\mathbb I]^\top$, while approximating the copositive cone $\mathcal C$ with $\mathcal C_0$. 

In order to develop decision rule approximation, we first use \cite[Theorem~14.60]{RW09:VA} to reformulate \eqref{eq:WCE} as
\begin{equation}
\label{eq:recourse_equiv}
\mathcal Z(\bm x)=\inf_{\bm y\in\mathcal L_{K,N_2}} \left\{ \sup_{\PP\in\Phat}\mathbb E_{\mathbb P}\left[\mathbf e^\top \bm y(\xit)\right]:
\bm A\bm\xi-\bm b\leq \bm y(\bm{\xi})\; \forall\bm{\xi}\in\Xi,~ \bm y(\bm{\xi})\geq \bm 0\; \forall\bm{\xi}\in\Xi \right\},
\end{equation}
where $\mathcal L_{K,N_2}$ denotes the linear space of all measurable functions from $\mathbb R^K$ to $\mathbb R^{N_2}$. A tractable upper bound on $\mathcal Z(\bm x)$ is obtained by restricting $\mathcal L_{K,N_2}$ to the subspace of all affine functions; see {\em e.g.}~\cite{GK16:Wasserstein}. A tighter tractable upper bound can be obtained, however, by restricting $\mathcal L_{K,N_2}$ to the subspace of all quadratic functions and by conservatively approximating the emerging semi-infinite constraints by semidefinite constraints using the approximate $\mathcal S$-lemma. Quadratic decision rule approximations of this type are also studied in~\cite{HKWZ2015:newsvendor}.

We run numerical experiments for different values of the uncertainty dimension~$K$ and the sample size~$I$ and set the Wasserstein radius to $\epsilon=1/\sqrt{I}$, thus enforcing the scaling rule advocated in~\cite{blanchet2016sample} and~\cite{ZG15:Wasserstein}.\footnote{We also ran all experiments with $\epsilon=1$ and $\epsilon=1/I$ but did not observe any qualitative changes in the results.} All results are averaged over $100$ instances generated randomly as follows. We sample the dimension $N_2$ of the wait-and-see decision uniformly at random from $\{1,2,\ldots,\lceil\log(K+1)\rceil\}$, which guarantees that the SOCP~\eqref{eq:SOCP_equiv} grows at most polynomially with $K$ and $I$. Next, we sample $\bm A$ uniformly  from $[0,1]^{N_2\times K}$ and $\bm b$ uniformly from $[0,\mathbf e^\top\bm A_{1:}]\times\cdots\times [0,\mathbf e^\top\bm A_{N_2:}]$.  We then generate independent training samples $\{\xih_i\}_{i\in[I]}$ from the uniform distribution on~$[0,1]^K$. Lastly, we evaluate the worst-case expectation \eqref{eq:WCE} exactly by solving the SOCP~\eqref{eq:SOCP_equiv}, and also approximately by computing the $\mathcal C^0$ and the quadratic decision rule approximations. 

Table~\ref{tab:gaps} reports the optimality gaps of the two approximations relative to the exact worst-case expectation, averaged across all solvable instances. While the optimality gaps of the $\mathcal C^0$ approximation remain consistently below $2.3\%$, the state-of-the-art quadratic decision rule approximation can incur alarmingly large optimality gaps of more than $100\%$. On the other hand, while MOSEK is  able to solve all instances of the quadratic decision rule approximation, it encounters numerical difficulties when solving some of the larger instances of the~$\mathcal C^0$ approximation. For $K=64$, for instance, the underlying semidefinite constraints involve blocks of the size $139\times 139$, which pose a distinct challenge for state-of-the-art interior point solvers. The percentages of all  instances of the $\mathcal C^0$ approximation that could be solved to global optimality are reported in Table~\ref{tab:solvable}. The average runtimes of both approximations are presented in Table~\ref{tab:solvetime}. % shows . For $K=64$ and $I\geq 20$, we observe that it can take on average more than $1$ hour to solve the semidefinite programs resulting from the~$\mathcal C^0$ approximation. % On the positive side, however, as depicted in Table \ref{tab:gaps} the $\mathcal C^0$ approximation is consistently near optimal while the quadratic decision rules approximation can incur significant optimality gaps. 

\begin{table}[h!]
	\small
	\color{black}
	\centering
	\begin{tabular}{c|c|rr|rr|rr|rr|rr|rr|rr|}
		\multicolumn{2}{c}{}	 & \multicolumn{14}{c}{$K$} \\ \cline{3-16}
		\multicolumn{2}{c}{} 	 & \multicolumn{2}{c}{1} & \multicolumn{2}{c}{2} & \multicolumn{2}{c}{4} & \multicolumn{2}{c}{8} & \multicolumn{2}{c}{16} & \multicolumn{2}{c}{32} & \multicolumn{2}{c}{64}\\[0.0mm] \cline{3-16}
		\multirow{ 8}{*}{$I$}
		&5 & 0.0&0.8   & 0.0&2.8 &  0.0&5.2  & 0.0   & 3.9      & 0.5 & 6.7 & 0.3 & 5.5  & 0.5 & 3.0\\
		&10 &  0.0&2.2 & 0.0&6.1 & 0.0&11.0    & 0.0 & 13.0     & 0.5 & 15.0 & 0.4 & 16.1 & 0.8 & 10.2\\
		&20 &  0.0&4.9 & 0.0&10.6& 0.0& 17.6  & 0.0 & 20.0    & 0.5 & 27.6 & 0.5 & 33.1 & 1.9 & 19.0\\
		&40 &  0.0&8.6 & 0.0&14.9  & 0.0& 25.4   & 0.0 & 28.2  & 0.8 & 55.0 &  0.6 & 65.2 & 1.8 & 54.0\\
		&80 &  0.0&12.4& 0.0&19.9&  0.0&33.2  & 0.0 & 42.0   & 0.9 & 77.1 &  0.8 & 117.1 & 0.0 & 121.1\\
		&160 & 0.0&18.2& 0.0&26.8&  0.0& 43.1 & 0.0 & 51.2  & 1.3 & 129.6 & 1.3 & 201.2 & - & 220.8\\
		&320&  0.0&25.3& 0.0&34.4 & 0.0&58.9  & 0.0 & 74.3  & 2.3 & 237.7 & 3.7 & 254.3 & - & 416.6\\
		&640 & 0.0&33.4& 0.0&43.7 &  0.0&79.1 & 0.0 &102.3    & 2.3 & 343.9 & 0.2 & 498.3 & - & 1137.1 \\
		\cline{3-16}	
	\end{tabular}
	\caption{\color{black} Optimality gaps (in $\%$) of the $\mathcal C^0$ approximation (left) and the quadratic decision rule approximation (right). % All results are averaged over 100 instances. 
		\label{tab:gaps}}
\end{table}

%\begin{table}[h!]
%	\small
%	\color{black}
%	\centering
%	\begin{tabular}{c|c|rrrrrrr|}
%		\multicolumn{2}{c}{}	 & \multicolumn{7}{c}{$K$} \\ \cline{3-9}
%		\multicolumn{2}{c}{} 	 & \multicolumn{1}{c}{1} & \multicolumn{1}{c}{2} & \multicolumn{1}{c}{4} & \multicolumn{1}{c}{8} & \multicolumn{1}{c}{16} & \multicolumn{1}{c}{32} & \multicolumn{1}{c}{64}\\[0.0mm] \cline{3-9}
%		\multirow{ 8}{*}{$I$}
%		&5 &100 & 100 & 100 & 100 & 100& 100 & 100  \\
%		&10 & 100 & 100 & 100 & 100 & 100& 100 & 95 \\
%		&20 &100 & 100 & 100 & 100 & 100& 100 & 48  \\
%		&40 & 100 & 100 & 100 & 100 & 100& 100 & 6 \\
%		&80 & 100 & 100 & 100 & 100 & 100& 100 &1 \\
%		&160 &100 & 100 & 100 & 100 & 100& 88 & 0  \\
%		&320& 100 & 100 & 100 & 100 &100& 19 & 0  \\
%		&640 & 100 & 100 & 100 &100& 100& 2 & 0 \\
%		\cline{3-9}
%	\end{tabular}
%	\caption{\color{black} Percentage of solvable instances from using the $\mathcal C^0$ approximation. % All results are averaged over 100 instences. 
%		\label{tab:solvable}}
%\end{table}
{\color{black}
\begin{table}[h!]
	\small
	\color{black}
	\centering
	\begin{tabular}{|rrrrrrrrrrrrr|}
		\multicolumn{10}{c}{$(I, K)$}	  \\ \cline{1-10}
		\multicolumn{1}{r}{(10,64)} & (20,64) & (40,64) & (80,64) &  (160,32) & (160,64) & (320,32) & (320,64) & (640,32) & \multicolumn{1}{r}{(640,64)}\\ \cline{1-10}
		 95 & 48 & 6 & 1 & 88 & 0 & 19 & 0 & 2 & 0 \\\cline{1-10}
		\cline{3-9}
	\end{tabular}
	\caption{\color{black} Percentage of solvable instances from using the $\mathcal C^0$ approximation. We report only those $(I,K)$ pairs for which fewer than $100\%$ of all instances were solved to optimality. % All results are averaged over 100 instences. 
		\label{tab:solvable}}
\end{table}
}

\begin{table}[h!]
	\hspace{-12mm}
	\small
	\color{black}
	\centering
	\begin{tabular}{c|c|rr|rr|rr|rr|rr|rr|rr|}
		\multicolumn{2}{c}{}	 & \multicolumn{14}{c}{$K$} \\ \cline{3-16}
		\multicolumn{2}{c}{} 	 & \multicolumn{2}{c}{1} & \multicolumn{2}{c}{2} & \multicolumn{2}{c}{4} & \multicolumn{2}{c}{8} & \multicolumn{2}{c}{16} & \multicolumn{2}{c}{32} & \multicolumn{2}{c}{64}\\[0.0mm] \cline{3-16}
		\multirow{ 8}{*}{$I$}
		&5 & $<$0.1	&$<$0.1	&$<$0.1 &$<$0.1 &  0.1	&0.1 & 0.5	&0.3 	&2.3	&0.5 		& 41.5     &6.2  	& 1375.9 & 230.4\\
		&10 &$<$0.1	&$<$0.1 & 0.2  	&0.1    & 0.4	&0.1 	& 0.8	&0.5	&5.2	&1.0 	& 93.9 		& 12.1 	& 3091.6 &355.1\\
		&20 &  0.1	&0.1    & 0.5	&0.1   	& 0.7	&0.2	& 1.9	&0.9	&11.6	&1.4  &194.2 		 & 26.6 & 5799.2 &490.2\\
		&40 &  0.3	&0.2   	& 1.6   &0.3    & 1.2	&0.3	& 3.0	&0.7	&25.0	&3.0	&  398.7     	&52.2 & 11276.2 & 1617.3\\	
		&80 &  0.6	&0.3   	&  3.3  &0.5    &  2.2	&0.5	&  8.1	&1.1 	&53.5	&6.7 	&905.9 			& 116.9 & 19887.3 & 3134.0\\
		&160 & 1.3	&0.8  	& 3.2   &1.1 	&  2.6	&0.5	& 19.4	&2.4	& 109.9	&17.7 	& 1956.5 		& 253.7  & - & 7586.9\\
		&320&  0.8	&2.0 	& 2.1  	&0.5    & 12.7	&1.0	&41.5	&5.2	& 251.4	&40.8&3715.2 		& 415.4 & - & 16511.3\\
		&640 &  1.7	&0.7   	& 31.3  	&1.1    &  15.5	&2.1	& 86.8	&15.1	& 514.7	&79.1 	& 9777.1		& 1441.9 & - & 22365.2 \\
		\cline{3-16}	
	\end{tabular}
	\caption{\color{black} Solution times in seconds of the $\mathcal C^0$ approximation (left) and the quadratic decision rule approximation (right). % All results are averaged over 100 instances. 
		\label{tab:solvetime}}
\end{table}

\subsection{Out-of-Sample Performance}

Next, we assess the out-of-sample performance of different data-driven policies in the context of a multi-item newsvendor problem, where an inventory planner has to select a vector $\bm x\in\RR_+^K$ of order quantities for $K$ different products at the beginning of a sales period. %We assume that product costs $c_1,\ldots,c_K$ are fixed and are known to the newsvendor. 
We assume that the total order quantity $\mathbf e^\top\bm x$ may not exceed a given budget $B$. The demands of the products can be described by a random vector $\xit\in\RR_+^K$ that follows an unknown multivariate distribution $\PP^\star$. We also assume that there are no ordering costs but that excess  inventory of the $k$-th product incurs a per-unit holding cost of $b_k$, while unmet demand incurs a per-unit stock-out cost of $s_k$. The total cost of an order $\bm x$ incurred in scenario $\bm \xi$ thus amounts to 
%corresponding to input parameters 
%\begin{equation}
%\label{eq:Newsvendor_params}
%\bm Q=\bm 0,\;\bm q=\mathbf e,\;\bm T(\bm x)=[-\diag(\bm b)\;\diag(\bm s)]^\top,\;\bm h(\bm x)=[\bm %x^\top\diag(\bm b)\;-\bm x^\top\diag(\bm s)]^\top,\;\text{ and }\;\bm W=[\mathbb I\;\;\mathbb I]^\top.
%\end{equation}
% Thus, the resulting recourse problem is given by
\begin{equation*}
	%\label{eq:Recourse_newsvendor}
	Z(\bm x,\bm{\xi})=\inf_{\bm y\in\mathbb R^K}  \left\{ \mathbf e^\top\bm y : \diag(\bm b)(\bm x-\bm{\xi})\leq \bm y,~ \diag(\bm s)(\bm{\xi}-\bm x)\leq \bm y \right\},
\end{equation*}
where $\bm b=(b_1,\ldots,b_K)^\top$ and $\bm s=(s_1,\ldots, s_K)^\top$. By construction, this recourse problem has sufficiently expensive recourse as well as complete recourse. We assume that the inventory planner is both risk-averse and ambiguity-averse and thus solves the two-stage distributionally robust linear program
\begin{equation}
\label{eq:Newsvendor_DRO}
	\begin{array}{ll@{\quad}l@{\quad}l}
		&\minimize&\displaystyle\sup_{\PP\in\Phat}\PP\text{-CVaR}_\rho\left[Z(\bm x,\xit)\right]\\
		&\subjectto&\displaystyle\bm x\in\mathcal \RR_+^K\\
		&&\mathbf e^\top\bm x\leq B,
	\end{array}
\end{equation}
which minimizes the worst-case conditional value-at-risk (CVaR) of $Z(\bm x,\xit)$ at level $\rho\in(0,1]$, where the worst case is taken over all distributions within some ambiguity set $\Phat$. For any fixed $\PP\in \Phat$, the $\PP$-CVaR of $Z(\bm x,\tilde{\bm\xi})$ at level $\rho$ is defined through
\[
 	\PP\text{-CVaR}_\rho\left[Z(\bm x,\tilde{\bm\xi})\right] = \inf_{\theta\in\RR}\theta+\frac{1}{\rho} \mathbb E_{\mathbb P}\left[\max\{Z(\bm x,\tilde{\bm\xi})-\theta,0\}\right]
\]
and can be viewed as the conditional expectation of $Z(\bm x,\tilde{\bm\xi})$ above its $(1-\rho)$-percentile under $\PP$ \cite{RU00:CVaR}. {\color{black} Note that the worst-case CVaR can be viewed as an instance of the worst-case optimized certainty equivalent~\eqref{eq:WCE_RA} corresponding to the disutility function $U(y)=\max\{y,0\}$.}

In the following we review different approaches to construct $\Phat$ from $I$ demand samples $\xih_1,\ldots,\xih_I$, and we show that in each case problem~\eqref{eq:Newsvendor_DRO} can be reformulated as a conic program. The first possibility is to set $\Phat$ to the $2$-Wasserstein ball $\mathcal B_\epsilon^{2}(\hat{\mathbb P}_I)$ around the empirical distribution on the demand samples as in Section~\ref{sec:CP}.

%Given $I$ different demand samples $\xih_1,\ldots,\xih_I$ from previous sales periods, the ambiguity set $\Phat$ can be defined as the Wasserstein ball $\mathcal B_\epsilon^{2}(\hat{\mathbb P}_I)$ as in Section \ref{sec:CP}. In this case \eqref{eq:Newsvendor_DRO} reduces to a copositive program.

\begin{prop}
\label{prop:Wass-Newsvendor}
If $\Phat=\mathcal B_\epsilon^{2}(\hat{\mathbb P}_I)$, then \eqref{eq:Newsvendor_DRO} is equivalent to the copositive program 	
	\begin{equation}
	\label{eq:CP_CVaR}
	\begin{array}{llll}
	&\hspace{-6mm}\displaystyle\minimize{} & \displaystyle\theta+\frac{1}{\rho}\left(\epsilon^2\lambda +\frac{1}{I}\sum_{i\in[I]}s_i\right)\\
	&\hspace{-6mm}\subjectto &\displaystyle\bm x\in\RR_+^K,\; \lambda, s_i\in\RR_+,\;\theta\in\RR,\;\bm\psi_{i},\bm\phi_{i}\in\RR^{N_2+J}&\forall i\in[I]\\
	&&\mathbf e^\top\bm x\leq B\\
	&&\displaystyle\hspace{-1mm}\begin{bmatrix}
	\displaystyle\lambda\mathbb I&\displaystyle-\frac{1}{2}\bm T(\bm x)^\top&\displaystyle-\lambda\xih_i\\
	\displaystyle-\frac{1}{2}\bm T(\bm x)&\displaystyle \bm W\diag(\bm\phi_{i})\bm W^\top&\displaystyle \frac{1}{2}(\bm W\bm\psi_{i}-\bm h(\bm x))\\\displaystyle
	-\lambda\xih_i^\top &\displaystyle\frac{1}{2}(\bm W\bm\psi_{i}-\bm h(\bm x))^\top& \displaystyle s_i+\theta-\mathbf e^\top(\bm\psi_{i}+\bm\phi_i)+\lambda\|\xih_i\|_2^2
	\end{bmatrix}\succeq_{\mathcal C}\bm 0&\forall i\in[I],
	\end{array}
	\end{equation}
    where $\bm T(\bm x)=[-\diag(\bm b)\;\diag(\bm s)]^\top$, $\bm h(\bm x)=[\bm x^\top\diag(\bm b)$, $-\bm x^\top\diag(\bm s)]^\top$, and $\bm W=[\mathbb I\;\;\mathbb I]^\top$.
   \end{prop}
    \begin{proof}
%    	We can use Sion's minimax theorem~\cite{sion1958general} to express the worst-case CVaR in~\eqref{eq:Newsvendor_DRO} as
%    	\begin{equation*}
%    	%\label{eq:WC_CVaR}
%	     \sup_{\PP\in\Phat}\PP\text{-CVaR}_\rho\left[Z(\bm x,\xit)\right] = \inf_{\theta\in\RR}\theta+\frac{1}{\rho}\sup_{\PP\in\Phat}\mathbb E_{\mathbb P}\left[\max\{Z(\bm x,\tilde{\bm\xi})-\theta,0\}\right].
%	     \end{equation*}  
%	    Substituting this expression into~\eqref{eq:Newsvendor_DRO} yields a worst-case expected disutility minimization problem of the type studied in Section~\ref{sect:risk-aversion} with here-and-now decision $(\bm x,\theta)$. 
{\color{black} The claim follows immediately from Corollary~\ref{prop:CP_RA} and the observation that \eqref{eq:Newsvendor_DRO} has both sufficiently expensive as well as complete recourse.}
% implies that the worst-case expectation in \eqref{eq:WC_CVaR} coincides with the optimal value of a minimization problem over copositive cones. Substituting this copositive program into the  worst-case expectation in \eqref{eq:WC_CVaR} and combining the arising optimization problem with the outer minimization over $\bm x$, we obtain the copositive program \eqref{eq:CP_CVaR}. 
    \end{proof}
    As a second possibility, we can use the $I$ demand samples to estimate the sample mean $\hat{\bm{\mu}}=\frac{1}{I}\sum_{i\in[I]}\xih_i$ and the sample covariance matrix $\hat{\bm{\Sigma}}=\frac{1}{I}\sum_{i\in[I]}(\xih_i-\hat{\bm{\mu}})(\xih_i-\hat{\bm{\mu}})^\top$ of $\xit$, which can in turn be used to construct a Chebyshev ambiguity set of the form
         \begin{equation}
     	\label{eq:ambiguity_Chebyshev}
     	\mathcal P(\hat{\bm{\mu}}, \hat{\bm{\Sigma}}, \gamma_1,\gamma_2)=\left\{\PP\in\mathcal M^2(\RR_+^K):\begin{array}{l}(\EE_\PP[\xit]-\hat{\bm\mu})^\top\hat{\bm\Sigma}^{-1}(\EE_\PP[\xit]-\hat{\bm\mu})\leq\gamma_1\\ \EE_\PP[(\xit-\hat{\bm\mu})(\xit-\hat{\bm\mu})^\top]\preceq(1+\gamma_2)\hat{\bm\Sigma}\end{array}\right\},
     \end{equation}
   where $\gamma_1,\gamma_2\in\mathbb R_+$ represent two confidence parameters. This ambiguity set has been proposed in \cite{DY10:DRO}. 
    \begin{prop}
    	If $\Phat=\mathcal P(\hat{\bm{\mu}}, \hat{\bm{\Sigma}}, \gamma_1,\gamma_2)$, then \eqref{eq:Newsvendor_DRO} is equivalent to the copositive program 	
   		\begin{equation}
   			\label{eq:Chebyshev_DRO}
   			\begin{array}{llll}
   				&\hspace{-6mm}\displaystyle\minimize{} & \displaystyle\theta+\frac{1}{\rho}\left[s+\tr\left((\gamma_2\hat{\bm\Sigma}+\hat{\bm\mu}\hat{\bm\mu}^\top)\bm M\right)+\hat{\bm\mu}^\top\bm m+\sqrt{\gamma_1}\|\hat{\bm\Sigma}^{\frac{1}{2}}(\bm m+2\bm M\hat{\bm\mu})\|_2\right]\\
   				&\hspace{-6mm}\subjectto &\displaystyle\bm x\in\RR_+^K,\;\theta, s\in\RR,\;\bm{m},\bm\psi,\bm\phi\in\mathbb R^K,\;\bm M\in\mathbb S_+^K\\
   				&&\displaystyle \mathbf e^\top\bm x\leq B\\[1mm]
   				&&\displaystyle\hspace{-1mm}\begin{bmatrix}
   					\displaystyle\bm M&\displaystyle-\frac{1}{2}\bm T(\bm x)^\top&\displaystyle\frac{1}{2}\bm m\\
   					\displaystyle-\frac{1}{2}\bm T(\bm x)&\displaystyle \bm W\diag(\bm\phi)\bm W^\top&\displaystyle \frac{1}{2}(\bm W\bm\psi-\bm h(\bm x))\\[2mm]\displaystyle
   					\frac{1}{2}\bm{m}^\top &\displaystyle\frac{1}{2}(\bm W\bm\psi-\bm h(\bm x))^\top& \displaystyle s+\theta-\mathbf e^\top(\bm{\psi}+\bm\phi)
   				\end{bmatrix}\succeq_{\mathcal C}\bm 0,\;\begin{bmatrix}
   				\bm M & \frac{1}{2}\bm m\\
   				\frac{1}{2}\bm m^\top & s
   			\end{bmatrix}\succeq_{\mathcal C}\bm 0,
   		\end{array}
   	\end{equation}
   	    where $\bm T(\bm x)=[-\diag(\bm b)\;\diag(\bm s)]^\top$, $\bm h(\bm x)=[\bm x^\top\diag(\bm b)$, $-\bm x^\top\diag(\bm s)]^\top$, and $\bm W=[\mathbb I\;\;\mathbb I]^\top$.
    \end{prop}
    \begin{proof}
    As in the proof of Proposition~\ref{prop:Wass-Newsvendor}, we may transform~\eqref{eq:Newsvendor_DRO} to a worst-case expected disutility minimization problem by using Sion's minimax theorem~\cite{sion1958general}. The worst-case expectation in the resulting objective function can then be re-expressed as a maximization problem over completely positive cones by leveraging ideas from~\cite[Section~4.4]{NRZ11:mixed01}. Finally, problem~\eqref{eq:Chebyshev_DRO} is obtained via strong conic duality, which allows us to convert the completely positive maximization problem to an equivalent copositive minimization problem.
    \end{proof}
%By setting $\mathcal X=\{\bm x\in\RR_+^K:\mathbf e^\top\bm x\leq B\}$ 
%As shown in \cite{RU00:CVaR}, this problem can be reformulated as the nonlinear stochastic program
%\begin{equation*}
%	\begin{array}{ll@{\quad}l@{\quad}l}
%		&\minimize&\displaystyle\theta+\frac{1}{\rho}\EE_{\PP}\left[\max\left\{Z(\bm x,\xit)-\theta,0\right\}\right]\\
%		&\subjectto&\displaystyle\bm x\in\mathcal X,\;\theta\in\RR.
%	\end{array}
%\end{equation*}
%As the distribution $\PP$ is not accessible to the planner, a na\"ive approach is to employ the empirical distribution  $\PPhat_I$ as a substitute of the true unknown distribution $\PP$. This leads to the \emph{sample-average approximation (SAA)} scheme given by
%\begin{equation}
%\label{eq:Newsvendor_SAA}
%\begin{array}{ll@{\quad}l@{\quad}l}
%	&\minimize&\displaystyle\theta+\frac{1}{\rho}\EE_{\PPhat_I}\left[\max\left\{Z(\bm x,\xit)-\theta,0\right\}\right]\\
%	&\subjectto&\displaystyle\bm x\in\mathcal X,\;\theta\in\RR.
%\end{array}
%\end{equation}
%On the other hand, an ambiguity-averse planner solves a \emph{distributionally robust model (DRO)} given by
%\begin{equation}
%	\label{eq:Newsvendor_DRO}
%		\begin{array}{ll@{\quad}l@{\quad}l}
%			&\minimize&\displaystyle\theta+\frac{1}{\rho}\sup_{\PP\in\Phat}\EE_{\PP}\left[\max\left\{Z(\bm x,\xit)-\theta,0\right\}\right]\\
%			&\subjectto&\displaystyle\bm x\in\mathcal X,\;\theta\in\RR.
%		\end{array}
%\end{equation}
 Distributionally robust multi-item newsvendor problems with Chebyshev ambiguity sets are known to be NP-hard even if $\gamma_1=\gamma_2=0$ and $\xit$ is supported on $\mathbb R^K$, but they admit tractable conservative approximations based on quadratic decision rules \cite{HKWZ2015:newsvendor}. 

A third possibility is to set $\Phat=\{\PPhat_I\}$. This singleton ambiguity set corresponds to a Wasserstein ball around the empirical distribution with radius $\epsilon=0$. Problem~\eqref{eq:Newsvendor_DRO} then simply reduces to the corresponding sample average approximation (SAA) problem, which is equivalent to a tractable linear program.

In order to assess the performance of the Wasserstein, Chebyshev and SAA policies obtained from the respective distributionally robust optimization models, we conduct out-of-sample experiments for the $K$-item newsvendor problem with $K=3$ and training datasets containing $I=10,$ $20,$ $40,$ $80,$ $160,$ $320$, $640$ and $1{,}028$ independent samples. In all experiments, we replace each copositive cone $\mathcal C$ appearing in~\eqref{eq:CP_CVaR} and~\eqref{eq:Chebyshev_DRO} with its first inner approximation~$\mathcal C^0$.  We fix the vectors of holding and stock-out costs to $\bm b=\mathbf e$ and $\bm s=10\mathbf e$, respectively, and we set the ordering budget to $B=30$.  We further fix the risk level of the CVaR to $\rho=10\%$. 

The results of all experiments are averaged over $100$ random trials  generated in the following manner. The true demand distribution $\PP^\star$ of $\xit$ is assumed to be lognormal, that is, $\tilde{\xi}_k=\exp(\tilde{\chi}_k)$, $k\in[K]$, where the~$\tilde{\chi}_k$, $k\in[K]$, represent jointly normally distributed random variables with first- and second-order moments given by $\bm{\nu}\in\RR_+^K$ and $\bm\Sigma\in\mathbb S_+^K$, respectively. In each trial, we sample $\bm\nu$ uniformly at random from $[0,2]^K$ while the matrix $\bm{\Sigma}$ is generated randomly using the following procedure. We set the vector of standard deviations to $\bm{\sigma}=1/4\mathbf e$, sample a random correlation matrix $\bm C\in\mathbb S_+^K$ using the MATLAB command~`{\texttt{gallery(`randcorr',3)}}', and set $\bm{\Sigma}=\diag(\bm\sigma)\bm C\diag(\bm\sigma)+\bm{\nu}\bm{\nu}^\top$. 
% matrix $\bm S\in\RR^{K\times K}$ with independent standard normally distributed elements. Next, we set the correlation matrix to $\bm C=\diag(\bm u)\bm S^\top\bm S\diag(\bm u)$, where $\bm u$ is the vector whose elements are given by $u_k=1/\sqrt{U_{kk}}$, $k\in[K]$. 
Next, we sample $I$ independent training samples $\{\xih_i\}_{i\in[I]}$ from $\PP^\star$. We then compute the Wasserstein policy ${\bm x}^\star_{\text{Wass}}$ by solving \eqref{eq:CP_CVaR} with $\mathcal C^0$ instead of $\mathcal C$ and where the Wasserstein radius $\epsilon$ is chosen by 5-fold cross-validation so as to minimize the out-of-sample risk \cite{efron1994introduction, friedman2001elements}. Similarly, the Chebhyshev policy ${\bm x}^\star_{\text{Cheb}}$ is obtained by solving \eqref{eq:Chebyshev_DRO} with $\mathcal C^0$ instead of $\mathcal C$ and where the confidence parameters are again determined via 5-fold cross-validation. Finally, we compute the SAA policy  ${\bm x}^\star_{\text{SAA}}$ by solving  \eqref{eq:Newsvendor_DRO} with $\Phat=\{\PPhat_I\}$. The out-of-sample risk $\PP^\star\text{-CVaR}_\rho [Z(\bm x^\star,\xit)]$ of each of the three data-driven strategies  ${\bm x}^\star_{\text{Wass}}$,  ${\bm x}^\star_{\text{Cheb}}$  and ${\bm x}^\star_{\text{SAA}}$ is then estimated at high accuracy using 20,000 test samples from $\PP^\star$.

Figure \ref{fig:Newsvendor_results_1} visualizes the out-of-sample risk of the Wasserstein and Chebyshev policies relative to the SAA policy as a function of the training sample size $I$. Observe that the Wasserstein policy dominates the SAA policy with high confidence uniformly across all sample sizes. Moreover, for training datasets of size $I\leq 20$, both the Wasserstein and Chebyshev policies outperform the SAA policy with high confidence by more than $20\%$.  This suggests that the distributionally robust policies are preferable whenever there is significant ambiguity about the true distribution $\PP^\star$. While the Wasserstein policy consistently outperforms the SAA policy, the quality of  the Chebyshev policy starts to deteriorate for $I\geq 30$.
\begin{figure}[h!]
	\centering
	\includegraphics[width=0.35\textwidth]{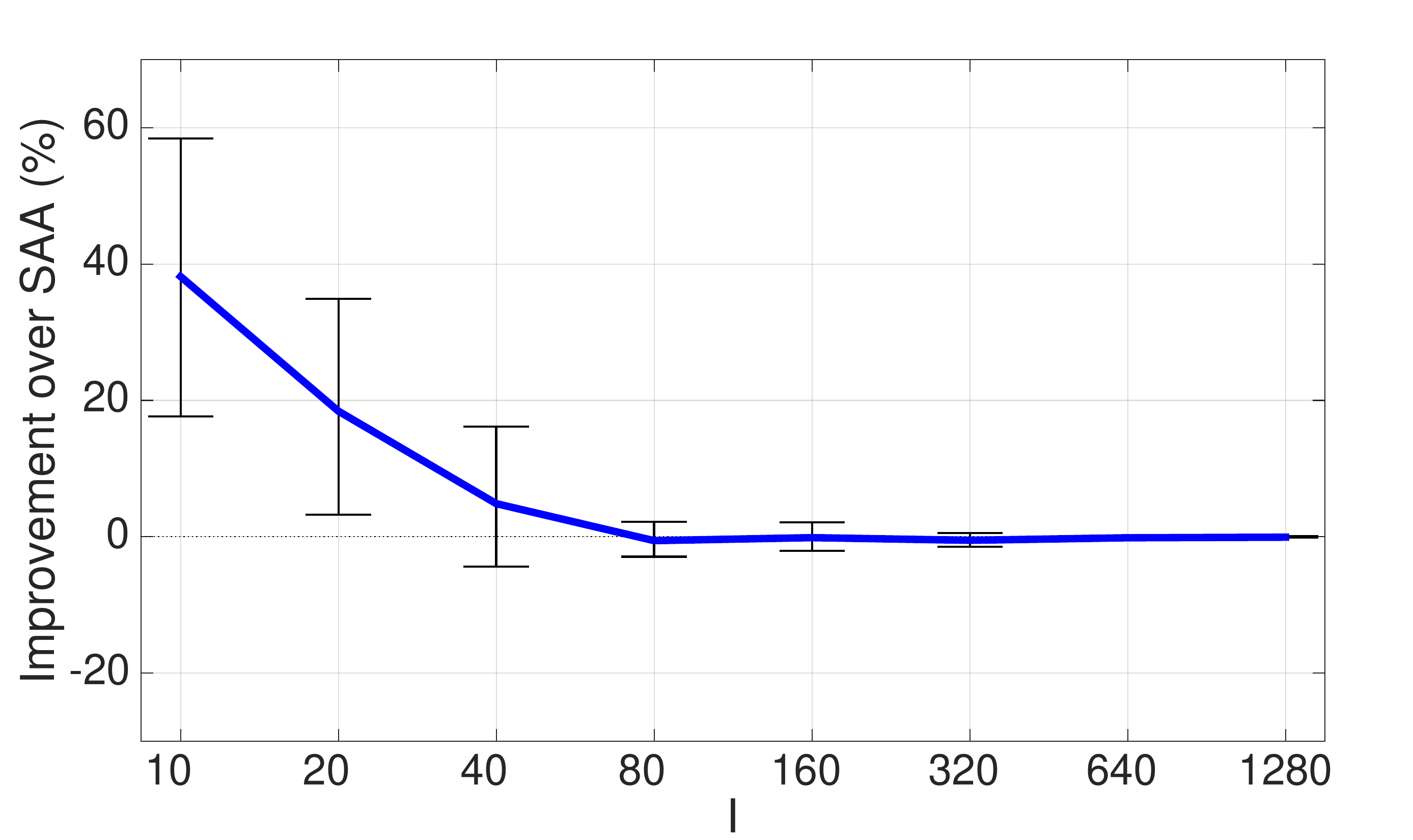} 
	\includegraphics[width=0.35\textwidth]{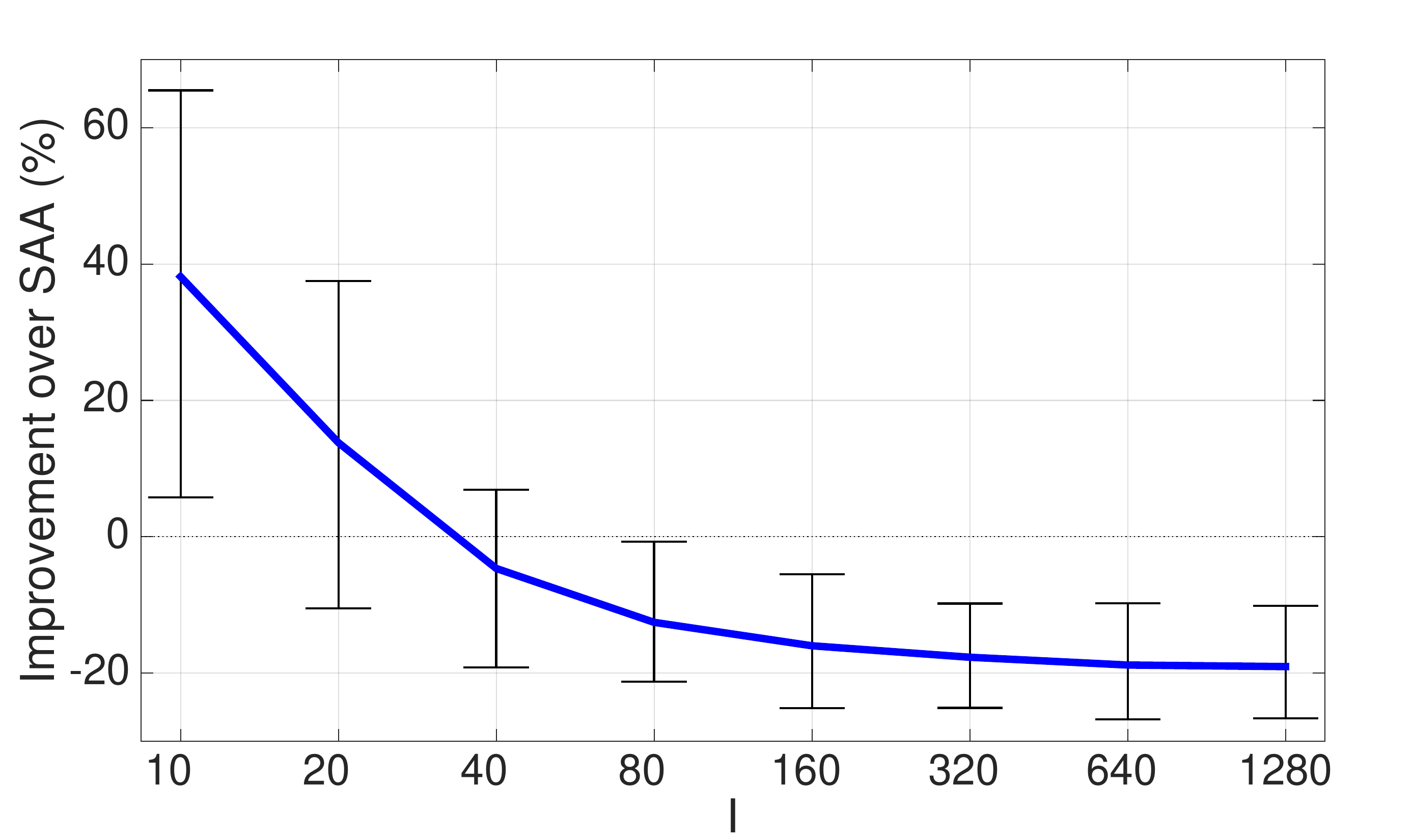} 
	\caption{Improvement of the Wasserstein (left) and Chebyshev (right) policy relative to the SAA policy in terms of out-of-sample CVaR. The solid blue lines represent the mean, and the error bars visualize the $20\%$ and $80\%$ quantiles of the relative improvement, respectively.}
	\label{fig:Newsvendor_results_1}
\end{figure}

Figure \ref{fig:Newsvendor_results_2} depicts the optimality gaps of the three policies with respect to the true optimal policy, which we estimate by solving another SAA problem using 20,000 samples from $\PP^\star$. For $I=10$, we find that the SAA policy is $\sim$$70\%$ suboptimal, while both  the Chebyshev and  Wasserstein policies are only $\sim$$25\%$ suboptimal on average. For $I= 20$, the SAA policy remains $\sim$$40\%$ suboptimal while the Wassertein policy exhibits a marginally better suboptimality of $\sim $$20\%$ on average. The  Chebyshev policy, on the other hand, reaches a steady state already for $I=10$ with an average suboptimality of about $25\%$. This suggests that the empirical estimates of the first- and second-order moments are already accurate enough for small training datasets. Unfortunately, as the sample size grows, the Chebyshev policy cannot improve as the first two moments are insufficient to describe the entire shape of the true distribution~$\PP^\star$. The Wasserstein policy, on the other hand, strikes a good balance between robustness and asymptotic consistency. In particular, we find that it converges quickly to the true optimal policy as the size of the training dataset grows.
\begin{figure}[h!]
	\centering
	\includegraphics[width=0.329\textwidth]{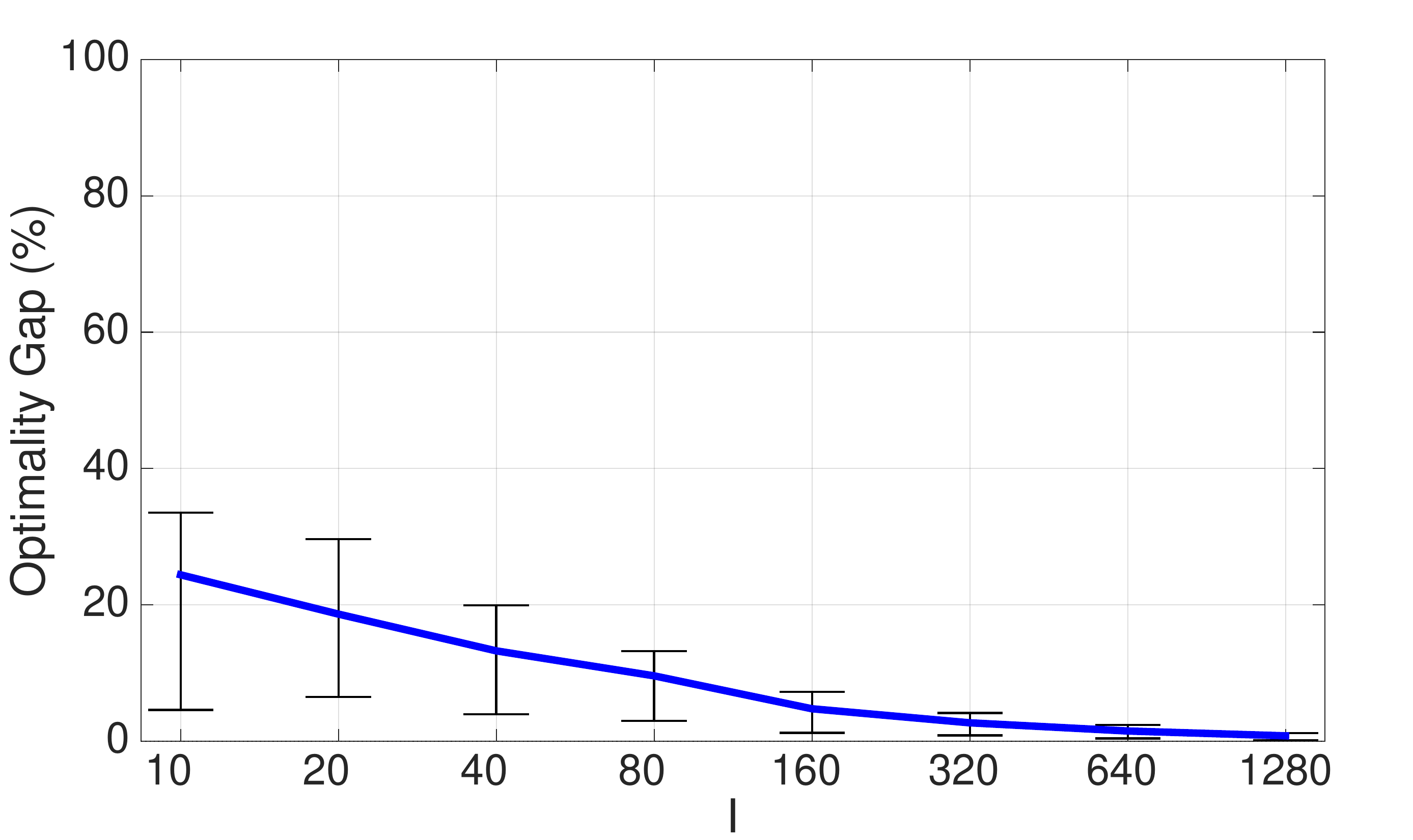} 
	\includegraphics[width=0.329\textwidth]{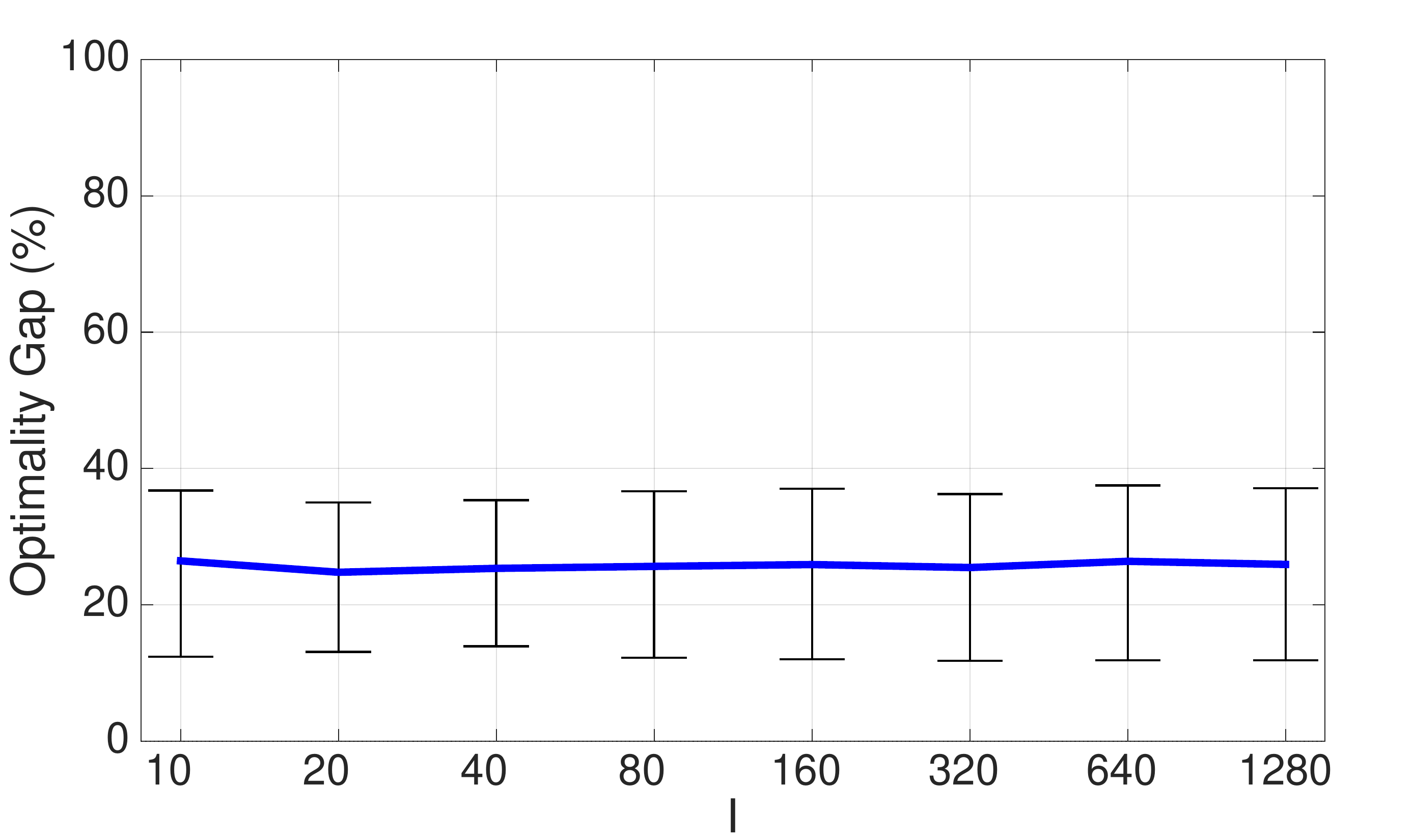} 
	\includegraphics[width=0.329\textwidth]{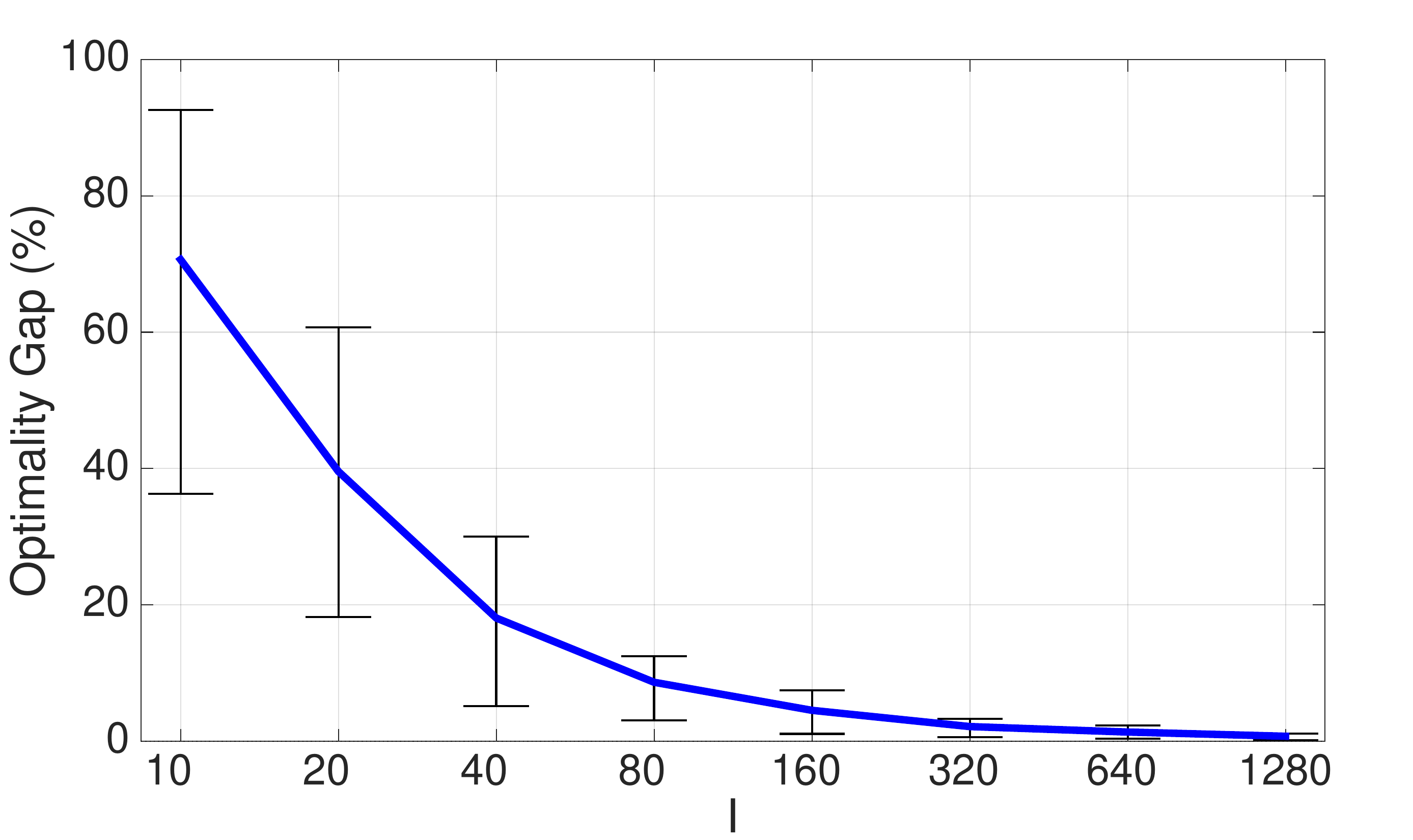} 
	\caption{Optimality gaps of the Wasserstein {(left)}, Chebyshev {(middle)}, and SAA {(right)} policies. The solid blue lines represent the mean, and the error bars visualize the $20\%$ and $80\%$ quantiles of the optimality gaps, respectively.}
	\label{fig:Newsvendor_results_2}
\end{figure}

%% Acknowledgments here
\paragraph{Acknowledgements} \hspace{2mm}\\
%\vspace{-5mm} 
The authors gratefully acknowledge the two anonymous reviewers whose comments led to substantial improvements of the paper. 
This research was supported by the Swiss National Science Foundation grant
 BSCGI0\underline{~}157733.
\\

%%%%%%%%%%
%\bibliographystyle{ormsv080} % outcomment this and next line in Case 1
%\bibliography{bibliography} % if more than one, comma separated

\bibliographystyle{plain}
\bibliography{bibliography}

\begin{thebibliography}{10}

\bibitem{ardestani2016linearized}
A.~Ardestani-Jaafari and E.~Delage.
\newblock Linearized robust counterparts of two-stage robust optimization
  problems with applications in operations management.
\newblock {\em Available on Optimization Online}, 2016.

\bibitem{baxter}
J.~Baxter.
\newblock A model of inductive bias learning.
\newblock {\em Journal of Artificial Intelligence Research}, 12(3):149--198,
  2000.

\bibitem{ben2009robust}
A.~Ben-Tal, L.~El Ghaoui, and A.~Nemirovski.
\newblock {\em Robust Optimization}.
\newblock Princeton University Press, 2009.

\bibitem{BGGN04:LDR}
A.~Ben-Tal, A.~Goryashko, E.~Guslitzer, and A.~Nemirovski.
\newblock Adjustable robust solutions of uncertain linear programs.
\newblock {\em Mathematical Programming A}, 99(2):351--376, 2004.

\bibitem{ben2007old}
A.~Ben-Tal and M.~Teboulle.
\newblock An old-new concept of convex risk measures: {T}he optimized certainty
  equivalent.
\newblock {\em Mathematical Finance}, 17(3):449--476, 2007.

\bibitem{bertsekas2009convex}
D.~P. Bertsekas.
\newblock {\em Convex Optimization Theory}.
\newblock Athena Scientific, 2009.

\bibitem{BDNT10:Models_Minimax}
D.~Bertsimas, X.~V. Doan, K.~Natarajan, and C.-P. Teo.
\newblock Models for minimax stochastic linear optimization problems with risk
  aversion.
\newblock {\em Mathematics of Operations Research}, 35(3):580--602, 2010.

\bibitem{BGL2014:robust_SAA}
D.~Bertsimas, V.~Gupta, and N.~Kallus.
\newblock Robust sample average approximation.
\newblock {\em Forthcoming in Mathematical Programming A}, 2017.

\bibitem{blanchet2016sample}
J.~Blanchet and Y.~Kang.
\newblock Sample out-of-sample inference based on {W}asserstein distance.
\newblock {\em Available on ArXiv}, 2016.

\bibitem{bomze2002solving}
I.~M. Bomze and E.~de~Klerk.
\newblock Solving standard quadratic optimization problems via linear,
  semidefinite and copositive programming.
\newblock {\em Journal of Global Optimization}, 24(2):163--185, 2002.

\bibitem{Burer09:copositive}
S.~Burer.
\newblock On the copositive representation of binary and continuous nonconvex
  quadratic programs.
\newblock {\em Mathematical Programming A}, 120(2):479--495, 2009.

\bibitem{caruana97}
R.~Caruana.
\newblock Multi-task learning.
\newblock {\em Machine Learning}, 28(1):41--75, 1996.

\bibitem{collobert2008unified}
R.~Collobert and J.~Weston.
\newblock A unified architecture for natural language processing: {D}eep neural
  networks with multitask learning.
\newblock In {\em International Conference on Machine Learning}, pages
  160--167, 2008.

\bibitem{DKP02:copositive}
E.~de~Klerk and D.~V. Pasechnik.
\newblock Approximation of the stability number of a graph via copositive
  programming.
\newblock {\em SIAM Journal on Optimization}, 12(4):875--892, 2002.

\bibitem{DY10:DRO}
E.~Delage and Y.~Ye.
\newblock Distributionally robust optimization under moment uncertainty with
  application to data-driven problems.
\newblock {\em Operations Research}, 58(3):595--612, 2010.

\bibitem{diananda62:copositive}
P.~H. Diananda.
\newblock On non-negative forms in real variables some or all of which are
  non-negative.
\newblock {\em Mathematical Proceedings of the Cambridge Philosophical
  Society}, 58(1):17--25, 1962.

\bibitem{dupacova:66}
J.~Dupa\v{c}ov\'{a}~(as \v{Z}\'{a}\v{c}kov\'{a}).
\newblock On minimax solutions of stochastic linear programming problems.
\newblock {\em \v{C}asopis pro p\v{e}stov�n� matematiky}, 91(4):423--430,
  1966.

\bibitem{efron1994introduction}
B.~Efron and R.~J. Tibshirani.
\newblock {\em An Introduction to the Bootstrap}.
\newblock CRC Press, 1994.

\bibitem{friedman2001elements}
J.~Friedman, T.~Hastie, and R.~Tibshirani.
\newblock {\em The Elements of Statistical Learning}.
\newblock Springer, 2001.

\bibitem{GK16:Wasserstein}
R.~Gao and A.~J. Kleywegt.
\newblock Distributionally robust stochastic optimization with {W}asserstein
  distance.
\newblock {\em Available on Optimization Online}, 2016.

\bibitem{georghiou2015generalized}
A.~Georghiou, W.~Wiesemann, and D.~Kuhn.
\newblock Generalized decision rule approximations for stochastic programming
  via liftings.
\newblock {\em Mathematical Programming A}, 152(1-2):301--338, 2015.

\bibitem{GS89:maxmin_exp_utility}
I.~Gilboa and D.~Schmeidler.
\newblock Maxmin expected utility with non-unique prior.
\newblock {\em Journal of Mathematical Economics}, 18(2):141--153, 1989.

\bibitem{goh2010distributionally}
J.~Goh and M.~Sim.
\newblock Distributionally robust optimization and its tractable
  approximations.
\newblock {\em Operations research}, 58(4-part-1):902--917, 2010.

\bibitem{G02:robust}
E.~Guslitser.
\newblock {Uncertainty-Immunized Solutions in Linear Programming}.
\newblock Master's thesis, Technion -- Israel Institute of Technology, 2002.

\bibitem{hadjiyiannis2011scenario}
M.~J. Hadjiyiannis, P.~J. Goulart, and D.~Kuhn.
\newblock A scenario approach for estimating the suboptimality of linear
  decision rules in two-stage robust optimization.
\newblock In {\em IEEE Conference on Decision and Control and European Control
  Conference}, pages 7386--7391, 2011.

\bibitem{Hanasusanto15:PhD}
G.~A. Hanasusanto.
\newblock {\em Decision Making under Uncertainty: Robust and Data-Driven
  Approaches}.
\newblock PhD thesis, Imperial College London, 2015.

\bibitem{HKWZ2015:newsvendor}
G.~A. Hanasusanto, D.~Kuhn, S.~W. Wallace, and S.~Zymler.
\newblock Distributionally robust multi-item newsvendor problems with
  multimodal demand distributions.
\newblock {\em Mathematical Programming A}, 152(1-2):1--32, 2015.

\bibitem{HKW16:SPComplexity}
G.~A. Hanasusanto, D.~Kuhn, and W.~Wiesemann.
\newblock A comment on {``}{C}omputational complexity of stochastic programming
  problems{''}.
\newblock {\em Mathematical Programming A}, 159(1):557--569, 2016.

\bibitem{hanasusanto2016k}
G.~A. Hanasusanto, D.~Kuhn, and W.~Wiesemann.
\newblock ${K}$-adaptability in two-stage distributionally robust binary
  programming.
\newblock {\em Operations Research Letters}, 44(1):6--11, 2016.

\bibitem{HRKW16:AJCC}
G.~A. Hanasusanto, V.~Roitch, D.~Kuhn, and W.~Wiesemann.
\newblock Ambiguous joint chance constraints under mean and dispersion
  information.
\newblock {\em Operations Research}, 65(3):751--767, 2017.

\bibitem{jiang2015risk}
R.~Jiang and Y.~Guan.
\newblock Risk-averse two-stage stochastic program with distributional
  ambiguity.
\newblock {\em Available on Optimization Online}, 2015.

\bibitem{kleywegt2002sample}
A.~J. Kleywegt, A.~Shapiro, and T.~Homem de~Mello.
\newblock The sample average approximation method for stochastic discrete
  optimization.
\newblock {\em SIAM Journal on Optimization}, 12(2):479--502, 2002.

\bibitem{kong2013scheduling}
Q.~Kong, C.-Y. Lee, C.-P. Teo, and Z.~Zheng.
\newblock Scheduling arrivals to a stochastic service delivery system using
  copositive cones.
\newblock {\em Operations Research}, 61(3):711--726, 2013.

\bibitem{lasserre2009convexity}
J.~B. Lasserre.
\newblock Convexity in semialgebraic geometry and polynomial optimization.
\newblock {\em SIAM Journal on Optimization}, 19(4):1995--2014, 2009.

\bibitem{lenk1996hierarchical}
P.~J. Lenk, W.~S. DeSarbo, P.~E. Green, and R.~M. Young.
\newblock Hierarchical {B}ayes conjoint analysis: {R}ecovery of partworth
  heterogeneity from reduced experimental designs.
\newblock {\em Marketing Science}, 15(2):173--191, 1996.

\bibitem{li2014distributionally}
X.~Li, K.~Natarajan, C.-P. Teo, and Z.~Zheng.
\newblock Distributionally robust mixed integer linear programs: {P}ersistency
  models with applications.
\newblock {\em European Journal of Operational Research}, 233(3):459--473,
  2014.

\bibitem{yalmip}
J.~L\"ofberg.
\newblock {YALMIP}: {A} toolbox for modeling and optimization in {MATLAB}.
\newblock In {\em IEEE International Symposium on Computer Aided Control
  Systems Design}, pages 284--289, 2004.

\bibitem{love2015phi}
D.~Love and G.~Bayraksan.
\newblock Phi-divergence constrained ambiguous stochastic programs for
  data-driven optimization.
\newblock {\em Available on Optimization Online}, 2016.

\bibitem{MEK15:Wasserstein}
P.~{Mohajerin Esfahani} and D.~Kuhn.
\newblock Data-driven distributionally robust optimization using the
  {W}asserstein metric: {P}erformance guarantees and tractable reformulations.
\newblock {\em Available on Optimization Online}, 2015.

\bibitem{MK87:NP_complete_quad}
K.~G. Murty and S.~N. Kabadi.
\newblock Some {NP}-complete problems in quadratic and nonlinear programming.
\newblock {\em Mathematical Programming}, 39(2):117--129, 1987.

\bibitem{natarajan2016reduced}
K.~Natarajan and C.-P. Teo.
\newblock On reduced semidefinite programs for second order moment bounds with
  applications.
\newblock {\em Mathematical Programming A}, 161(1-2):487--518, 2017.

\bibitem{NRZ11:mixed01}
K.~Natarajan, C.-P. Teo, and Z.~Zheng.
\newblock Mixed 0-1 linear programs under objective uncertainty: {A} completely
  positive representation.
\newblock {\em Operations Research}, 59(3):713--728, 2011.

\bibitem{parrilo2000structured}
P.~A. Parrilo.
\newblock {\em Structured semidefinite programs and semialgebraic geometry
  methods in robustness and optimization}.
\newblock PhD thesis, California Institute of Technology, 2000.

\bibitem{pflug2014multistage}
G.~Pflug and A.~Pichler.
\newblock {\em Multistage Stochastic Optimization}.
\newblock Springer, 2014.

\bibitem{pflug2007ambiguity}
G.~Pflug and D.~Wozabal.
\newblock Ambiguity in portfolio selection.
\newblock {\em Quantitative Finance}, 7(4):435--442, 2007.

\bibitem{RU00:CVaR}
R.~T. Rockafellar and S.~Uryasev.
\newblock Optimization of conditional value-at-risk.
\newblock {\em Journal of Risk}, 2:21--42, 2000.

\bibitem{RW09:VA}
R.~T. Rockafellar and R.~J.-B. Wets.
\newblock {\em Variational Analysis}.
\newblock Springer Science \& Business Media, 2009.

\bibitem{Scarf:58}
H.~E. Scarf.
\newblock A min-max solution to an inventory problem.
\newblock In K.~J. Arrow, S.~Karlin, and H.~E. Scarf, editors, {\em Studies in
  Mathematical Theory of Inventory and Production}, pages 201--209. Stanford
  University Press, Stanford, CA, 1958.

\bibitem{shafieezadeh2015distributionally}
S.~{Shafieezadeh Abadeh}, P.~{Mohajerin Esfahani}, and D.~Kuhn.
\newblock Distributionally robust logistic regression.
\newblock In {\em Advances in Neural Information Processing Systems}, pages
  1576--1584, 2015.

\bibitem{SBBJ15:CP-RANK}
N.~Shaked-Monderer, A.~Berman, I.~M. Bomze, F.~Jarre, and W.~Schachinger.
\newblock New results on the cp-rank and related properties of co(mpletely
  )positive matrices.
\newblock {\em Linear and Multilinear Algebra}, 63(2):384--396, 2015.

\bibitem{shapiro01:conic_duality}
A.~Shapiro.
\newblock On duality theory of conic linear problems.
\newblock In M.~A. Goberna and M.~A. L\'{o}pez, editors, {\em Semi-Infinite
  Programming}, pages 135--165. Kluwer Academic Publishers, 2001.

\bibitem{shapiro2003monte}
A.~Shapiro.
\newblock Monte {C}arlo sampling methods.
\newblock In A.~Ruszczy\'{n}ski and A.~Shapiro, editors, {\em Handbook in
  Operations Research and Management Science}, volume~10, pages 353--425.
  Elsevier, 2003.

\bibitem{shapiro2014lectures}
A.~Shapiro, D.~Dentcheva, and A.~Ruszczy\'nski.
\newblock {\em Lectures on Stochastic Programming: Modeling and Theory}.
\newblock SIAM, 2014.

\bibitem{SK02:minimax_analysis}
A.~Shapiro and A.~Kleywegt.
\newblock Minimax analysis of stochastic problems.
\newblock {\em Optimization Methods and Software}, 17(3):523--542, 2002.

\bibitem{sion1958general}
M.~Sion.
\newblock On general minimax theorems.
\newblock {\em Pacific Journal of Mathematics}, 8(1):171--176, 1958.

\bibitem{S05:MatrixNorm}
D.~Steinberg.
\newblock {Computation of Matrix Norms with Applications to Robust
  Optimization}.
\newblock Master's thesis, Technion -- Israel Institute of Technology, 2005.

\bibitem{thiele2009robust}
A.~Thiele, T.~Terry, and M.~Epelman.
\newblock Robust linear optimization with recourse.
\newblock {\em Available on Optimization Online}, 2009.

\bibitem{tibshirani1996regression}
R.~Tibshirani.
\newblock Regression shrinkage and selection via the lasso.
\newblock {\em Journal of the Royal Statistical Society Series B},
  58(1):267--288, 1996.

\bibitem{villani2008optimal}
C.~Villani.
\newblock {\em Optimal {T}ransport: {O}ld and {N}ew}.
\newblock Springer Science \& Business Media, 2008.

\bibitem{wang2006regularized}
L.~Wang, M.~D. Gordon, and J.~Zhu.
\newblock Regularized least absolute deviations regression and an efficient
  algorithm for parameter tuning.
\newblock In {\em IEEE International Conference on Data Mining}, pages
  690--700, 2006.

\bibitem{WKS13:drco}
W.~Wiesemann, D.~Kuhn, and M.~Sim.
\newblock Distributionally robust convex optimization.
\newblock {\em Operations Research}, 62(6):1358--1376, 2014.

\bibitem{wozabal2012framework}
D.~Wozabal.
\newblock A framework for optimization under ambiguity.
\newblock {\em Annals of Operations Research}, 193(1):21--47, 2012.

\bibitem{xu2016copositive}
G.~Xu and S.~Burer.
\newblock A copositive approach for two-stage adjustable robust optimization
  with uncertain right-hand sides.
\newblock {\em Available on arXiv}, 2016.

\bibitem{zeng2013solving}
B.~Zeng and L.~Zhao.
\newblock Solving two-stage robust optimization problems using a
  column-and-constraint generation method.
\newblock {\em Operations Research Letters}, 41(5):457--461, 2013.

\bibitem{zhang2012multi}
D.~Zhang, D.~Shen, Alzheimer's Disease~Neuroimaging Initiative, and {others}.
\newblock Multi-modal multi-task learning for joint prediction of multiple
  regression and classification variables in {A}lzheimer's disease.
\newblock {\em Neuroimage}, 59(2):895--907, 2012.

\bibitem{ZG15:Wasserstein}
C.~Zhao and Y.~Guan.
\newblock Data-driven risk-averse stochastic optimization with {W}asserstein
  metric.
\newblock {\em Available on Optimization Online}, 2015.

\end{thebibliography}

%\textbf{Grani A. Hanasusanto}  is an assistant professor of Operations Research and Industrial Engineering at the University of Texas at Austin. His research focuses on the design and analysis of tractable solution schemes for decision-making problems under uncertainty, with applications in operations management, energy systems, machine learning and data analytics.
%\\[-2mm]
%
%\textbf{Daniel Kuhn} is Professor of Operations Research at the College of Management of Technology at EPFL, where he holds the Chair of Risk Analytics and Optimization (RAO). His current research interests are focused on the modeling of uncertainty, the development of efficient computational methods for the solution of stochastic and robust optimization problems and the design of approximation schemes that ensure their computational tractability. This work is primarily application-driven, the main application areas being energy systems, operations management and engineering. 

\newpage
\appendix
\section{E-Companion: Proof of Theorem~\ref{thm:moment_problem}}

	\begin{proof}[Proof of Theorem~\ref{thm:moment_problem}]
		By the definition of the Wasserstein metric, the worst-case expected wait-and-see cost over the ambiguity set $\mathcal B_\epsilon^{r}(\hat{\mathbb P}_I)$ can be expressed as
		\begin{align}
		\label{eq:original_moment}
		\hspace{-10mm}\qquad\qquad\qquad\; \mathcal Z(\bm x)&=&\sup_{\mathbb P\in\mathcal M^r(\Xi)}&\left\{\mathbb E_{\mathbb P}\left[Z(\bm x,{\xit})\right]~:~W^r(\mathbb P,\hat{\mathbb P}_I)\leq \epsilon\right\}  \\\nonumber
		&=&\sup\;\;\;\;&\int_{\Xi}Z(\bm x,{\bm\xi})\;\;\PP(\mathrm d\bm\xi)\\\nonumber
		&&\st\;\;\;\;&\mathbb P\in\mathcal M^r(\Xi),\;\Pi\in\mathcal M^r(\Xi\times\Xi)\\\nonumber
		&&&\int_{\Xi\times\Xi}{d}(\bm\xi,{\bm\xi}')^r\;\Pi(\mathrm d\bm\xi,\mathrm d\bm\xi')\leq\epsilon^r\\\nonumber
		&&&\displaystyle\text{$\Pi$ is a joint distribution of $\xit$ and $\xit'$}
		\displaystyle\text{ with marginals $\mathbb P$ and $\PPhat_I$, respectively}.\nonumber
		\end{align}
		By the law of total probability we can decompose the transportation plan as $\Pi=\frac{1}{I}\sum_{i\in[I]}\PP_i$, where $\PP_i$ represents the distribution of $\xit$ conditional on $\xit'=\xih_i$. 
		Thus, $\mathcal Z(\bm x)$ coincides with the optimal value of the generalized moment problem \eqref{eq:primal_moment}. This establishes the first claim.
%		\cite{MEK15:Wasserstein,GK16:Wasserstein}, we obtain
%		\begin{equation*}
%		\begin{array}{c@{\quad}l@{\quad}ll}
%		\displaystyle\sup_{\mathbb P\in\hat{\mathcal P}}\mathbb E_{\mathbb P}\left[Z(\bm x,{\xit})\right]=&\displaystyle\sup&\displaystyle\frac{1}{I}\sum_{i\in[I]}\int_{\mathbb R^K}Z(\bm x,\bm\xi)\;\mathbb P_i(\textup{d}\bm\xi)\\
%		&\displaystyle\st&\displaystyle\mathbb P_i\in\mathcal M^1(\mathbb R^K)&\forall i\in[I]\\
%		&&\displaystyle\frac{1}{I}\sum_{i\in[I]}\int_{\mathbb R^K}\|\bm\xi-\hat{\bm\xi}_i\|\;\mathbb P_i(\textup{d}\bm\xi)\leq\epsilon.
%		\end{array}
%		\end{equation*}
%			
%		The moment problem \eqref{eq:primal_moment} was derived in \cite[Theorem 4.2]{MEK15:Wasserstein} for $r=1$ and in \cite[Proposition 2]{GK16:Wasserstein} for the generic cases. This shows that \eqref{eq:primal_moment} is equivalent to \eqref{eq:WCE}.

		The second claim follows from the observation that the semi-infinite linear program~\eqref{eq:dual_semiinf} is dual to the generalized moment problem \eqref{eq:primal_moment}. %Thus, the optimal value of \eqref{eq:dual_semiinf} constitutes an upper bound to $\mathcal Z(\bm x)$ by weak duality. 
		Indeed, if $Z(\bm x,\bm\xi)$ is finite for all $\bm\xi\in\Xi$, then strong duality between~\eqref{eq:primal_moment} and~\eqref{eq:dual_semiinf} holds for all $\epsilon>0$ due to a straightforward generalization of \cite[Proposition 3.4]{shapiro01:conic_duality} (see also \cite[Lemma 7]{HRKW16:AJCC}). If there exists $\overline{\bm\xi}\in\Xi$ with $Z(\bm x,\overline{\bm\xi})=\infty$, on the other hand, then the dual problem \eqref{eq:dual_semiinf} is infeasible as all its inner maximization problems are unbounded. In this case, however, the primal problem \eqref{eq:primal_moment} is unbounded. Indeed, since $W^r(\hat{\PP}_I,\hat{\PP}_I)=0$, the continuity of the reference metric $d$ implies that the mixture distribution $\PP^\star=(1-\tau)\hat\PP_I+\tau\delta_{\overline{\bm\xi}}$ is feasible in \eqref{eq:original_moment} for some $\tau\in(0,1]$, which implies that $\mathcal Z(\bm x) \geq \EE_{\PP^\star}[Z(\bm x,\tilde{\bm\xi})]=\infty$. Thus, the optimal value of the dual problem \eqref{eq:dual_semiinf} always coincides with $\mathcal Z(\bm x)$. This completes the proof.
	\end{proof}

%% Here starts the e-companion (EC)
%%%%%%%%%%%%%%%%%%%%%%%%%%%%%%%%%%%%%%%%%%%%%%%%%%%%%%%%%%
%\ECSwitch

%\ECDisclaimer
%%%%%%%%%%%%%%%%%%%%%%%%%%%%%%%%%%%%%%%%%%%%%%%%%%%%%%%%%%

%%% Main head for the e-companion
%\ECHead{E-Companion}

% References here (outcomment the appropriate case)

% CASE 1: BiBTeX used to constantly update the references
%   (while the paper is being written).
%\bibliographystyle{ormsv080} % outcomment this and next line in Case 1
%\bibliography{bibliogrphy} % if more than one, comma separated

% CASE 2: BiBTeX used to generate mypaper.bbl (to be further fine tuned)
%\input{mypaper.bbl} % outcomment this line in Case 2

%If you don't use BiBTex, you can manually itemize references as shown below.

%%%%%%%%%%%%%%%%%
\end{document}